%% file: ssp1-arxiv.tex
\documentclass{article}
\usepackage {amsmath, amsfonts, amssymb, mathrsfs, amsthm,stmaryrd, sectsty,hyphenat,enumitem,calc,url,color,array,tabu}
\usepackage{palatino}
\usepackage[all]{xy}
\usepackage[nottoc]{tocbibind}
\usepackage[toc,page]{appendix}
\usepackage{tocloft}

\input{"macros.tex"}

\newtheorem{thm}{Theorem}[subsection]
\newtheorem{lem}[thm]{Lemma}
\newtheorem{pro}[thm]{Proposition}
\newtheorem{cor}[thm]{Corollary}
\newtheorem{fct}[thm]{Fact}

\newtheorem{asm}[thm]{Assumption}

\newtheorem{ntt}[thm]{Notation}
\theoremstyle{definition}
\newtheorem{dfn}[thm]{Definition}
\newtheorem{rem}[thm]{Remark}
\newtheorem{exa}[thm]{Example}

\sectionfont{\center\sc\normalsize}
\subsectionfont{\bf\normalsize}

\numberwithin{equation}{section}
\renewcommand{\-}{\hyp{}}
\newlength{\sumcorr}

\newcolumntype{L}{>{$}l<{$}}

\begin{document}

\begin{mytitle} Supercuspidal $L$-packets \end{mytitle}
\begin{center} Tasho Kaletha \end{center}

{\let\thefootnote\relax\footnotetext{This research is supported in part by NSF grant DMS-1801687 and a Sloan Fellowship.}}

\begin{abstract}
Let $F$ be a non-archimedean local field and let $G$ be a connected reductive group defined over $F$. We assume that $G$ splits over a tame extension of $F$ and that the residual characteristic $p$ does not divide the order of the Weyl group. To each discrete Langlands parameter of the Weil group of $F$ into the complex $L$-group of $G$ we associate explicitly a finite set of irreducible supercuspidal representations of $G(F)$, and relate its internal structure to the centralizer of the parameter. We give evidence that this assignment is an explicit realization of the local Langlands correspondence.
\end{abstract}

\tableofcontents

\section{Introduction}

According to the conjectural local Langlands correspondence, the set of isomorphism classes of irreducible admissible representations of the group $G(F)$ of $F$-points of a connected reductive group $G$ defined over a non-archimedean local field $F$ should be partitioned into finite subsets, called $L$-packets, and each $L$-packet should correspond to a Langlands parameter, which is a homomorphism $W_F \times \tx{SL}_2(\C) \to {^LG}$ from the Weil-Deligne group of $F$ into the Langlands $L$-group of $G$, subject to certain conditions. The $L$-packet is expected to be in bijection with the set of representations of a certain finite group associated to $\varphi$. When the image of the parameter does not lie in a proper parabolic subgroup of $^LG$, the $L$-packet is expected to consist of essentially discrete series representations. When furthermore the parameter restricts trivially to $\tx{SL}_2(\C)$, the packet is expected to consist of supercuspidal representations -- this expectation was formulated in \cite[\S3.5]{DR09} and is a special case of the  the more precise conjecture of \cite{AMS}. We shall call such parameters and packets supercuspidal for short.

In \cite{KalRSP} we constructed a correspondence between supercuspidal parameters and supercuspidal $L$-packets, as well as the desired enumeration of the members of each packet, under the following assumptions: $G$ splits over a tamely ramified extension of $F$, the residual characteristic $p$ of $F$ is not a bad prime for the root system of $G$, and the Langlands parameter satisfies a certain regularity assumption. The construction works in both directions -- from parameters to packets and conversely -- and has the important feature of being explicit. 

In this paper we extend this construction to the case of arbitrary supercuspidal parameters, i.e. we drop the regularity assumption imposed on the parameters in \cite{KalRSP}. We do this at the cost of a slightly stricter assumption on $p$, which we now require to not divide the order of the Weyl group of $G$. In fact, when $p$ is not a bad prime for the root system of $G$, but possibly divides the order of the Weyl group, the construction given here still works and handles many non-regular supercuspidal parameters, but possibly not all of them. More precisely, we call a Langlands parameter \emph{torally wild} if it maps wild inertia into a torus inside of the dual group. When $G$ splits over a tame extension and $p$ does not divide the order of the Weyl group, all supercuspidal Langlands parameters are torally wild. In this paper we construct the $L$-packets associated to torally wild supercuspidal parameters when $G$ splits over a tame extension and $p$ is not a bad prime for $G$ and does not divide the connection index of any simple factor. This is in particular the case when $G$ splits over a tame extension and $p$ does not divide the order of the Weyl group.

The following table gives for each Dynkin type the sets of primes that are bad or divide the connection index in the first row, and those that divide the order of the Weyl group in the second row.

\[ \begin{tabu}{|l|l|l|l|l|l|l|l|l|}
\hline
A_n&B_n&C_n&D_n&E_6&E_7&E_8&F_4&G_2\\
\hline
p|n+1&2&2&2&2,3&2,3&2,3,5&2,3&2,3\\
\hline
p\leq n+1&p\leq n&p\leq n&p\leq n&2,3,5&2,3,5,7&2,3,5,7&2,3&2,3\\
\hline
\end{tabu}\]

As in the regular case treated in \cite{KalRSP}, the construction given here goes in both directions, and is explicit. An essential new phenomenon in the non-regular case is that the group $S_\varphi$ and its variations $\pi_0(S_\varphi)$ and $\pi_0(S_\varphi^+)$ that control the structure of the $L$-packet are often non-abelian. For the group $\pi_0(S_\varphi^+)$ that accommodates all inner forms, this already happens for $\tx{SL}_2$ and is a classical example discussed by Labesse and Langlands \cite{LL79}. But there are also examples for the group $\pi_0(S_\varphi)$, and even for $\pi_0(S_\varphi/Z(\hat G)^\Gamma)$, for split groups of classical type, such as the split group $\tx{Spin}_9$. This makes the internal structure of the resulting $L$-packets considerably more subtle.

Before we describe the complications that arise in the non-regular case and our strategy to handle them, we first review the construction in the regular case. A supercuspidal parameter $\varphi : W_F \to {^LG}$ is called strongly regular if $\tx{Cent}(\varphi(I_F),\hat G)$ is abelian. The notion of regularity is slightly weaker and more complicated to state. From a regular supercuspidal parameter we are able to extract the information necessary to write down a formula for the Harish-Chandra character of those supercuspidal representations that should populate the $L$-packet. On the other hand, we introduce the notion of a regular supercuspidal representation, classify all such, and give a formula for their Harish-Chandra characters, by reinterpreting the works of Adler, DeBacker, and Spice \cite{AS09}, \cite{DS18}. Each of the formulas extracted from a regular supercuspidal parameter then uniquely specifies a regular supercuspidal representation.

Extracting from the Langlands parameter $\varphi$ the information for the Harish-Chandra character is done by showing that $\varphi$ specifies an algebraic torus $S$ defined over $F$, then choosing Langlands-Shelstad $\chi$-data $(\chi_\alpha)_\alpha$ for the root system $R(S,G)$, using $(\chi_\alpha)_\alpha$ to obtain an embedding ${^LS} \to {^LG}$ through which $\varphi$ factors, using the factored Langlands parameter $\varphi_{S,\chi} : W_F \to {^LS}$ to obtain a character $\theta_\chi$ of $S(F)$, and then using $(\chi_\alpha)_\alpha$ and $\theta_\chi$ together to write down the character formula. The resulting formula is independent of the choice of $(\chi_\alpha)_\alpha$. It specifies for each embedding $j : S \to G$ a regular supercuspidal representation $\pi_j$ of $G(F)$ and the $L$-packet is the set of these representations. Essential for this procedure is that the regularity of $\varphi$ implies the regularity of $\theta_\chi$. 

Our work in the non-regular case begins with the observation that when $\varphi$ is a torally wild supercuspidal parameter then $\tx{Cent}(\varphi(I_F),\hat G)$, while in general not abelian, has abelian identity component. This property is certainly weaker than regularity, but it still allows to obtain from $\varphi$ an algebraic torus $S$ over $F$, and after choosing $\chi$-data also a character $\theta_\chi$ of $S(F)$. Moreover, while $\theta_\chi$ is in general not regular, it is still rather constrained. For example, when $\varphi$ is trivial on wild inertia (i.e. it is of depth zero), then $\theta_\chi$ is non-singular in a sense similar to that defined by Deligne-Lusztig \cite{DL76} in the setting of finite groups of Lie type. For finite groups of Lie type, the Deligne-Lusztig induction of a non-singular character of an elliptic maximal torus is a (usually reducible) cuspidal representation. Its components were studied by Lusztig \cite{Lus88}. It would be natural to expect that the structure of the corresponding reducible depth zero supercuspidal representations can be related, via the results of Moy and Prasad \cite{MP96}, to the structure of the cuspidal representation over the finite field, and that the situation of general depth can be reduced to depth zero using Yu's construction \cite{Yu01} and the Howe factorization process introduced in \cite{KalRSP}, neither of which assumes regularity.

With these observations made, the road to success seems mapped out. However, once one has set foot on that road one encounters a number of serious and initially rather unexpected challenges. This explains the length of this paper and the fact that two essential technical discussions needed here merited their own papers \cite{KalDC} and \cite{FKS}, and a key argument is borrowed from a third paper \cite{KalLLCD}, where it had a rather different purpose.

The first serious obstacle concerns the depth-zero case. Let $\kappa_{(S^\circ,\theta^\circ)}$ denote the cuspidal representation of a finite group of Lie type obtained via Deligne-Lusztig induction of a non-singular character $\theta^\circ$ of an elliptic maximal torus $S^\circ$. Lusztig has shown that $\kappa_{(S^\circ,\theta^\circ)}$ has multiplicity one -- it is a direct sum of pairwise inequivalent irreducible cuspidal representations -- and the set of these irreducible factors, which we shall denote by $[\kappa_{(S^\circ,\theta^\circ)}]$, is acted upon simply transitively by a finite abelian group -- the Pontryagin dual of the stabilizer of $\theta^\circ$ in the Weyl group. Work of Moy and Prasad relates depth-zero supercuspidal representations of a $p$-adic group to cuspidal representations of its reductive parahoric quotients. One would thus optimistically expect that results similar to Lusztig's hold for the depth-zero supercuspidal representations related to non-singular Deligne-Lusztig representations, and that simple Clifford theory would suffice in describing them. This is not the case. In fact, already the multiplicity one statement fails, although constructing an example takes quite a bit of effort, since the $p$-adic group has to be ramified and cannot be simple, simply connected, or adjoint. The culprit is the difference between the parahoric subgroup $G(F)_{x,0}$ and the stabilizer $G(F)_x$ of the point $x$. Lusztig's results concern the finite group of Lie type $G(F)_{x,0}/G(F)_{x,0+}$ and by inflation the compact open group $G(F)_{x,0}$, while Moy-Prasad theory classifies depth-zero supercuspidal representations in terms of irreducible representations of $G(F)_x$. It is in the passage from $G(F)_{x,0}$ to $G(F)_x$ where the complications arise. We can analyze the situation by working over the residue field $k_F$ and considering the connected reductive $k_F$-group $\ms{G}_x^\circ$ that is the reductive quotient of the special fiber of the parahoric group scheme associated to $x$, as well as the disconnected reductive $k_F$-group $\ms{G}_x$ that is the reductive quotient of the special fiber of the integral model of $G$ whose group of integral points is the stabilizer $G(F)_x$ of the point $x$. The notion of Deligne-Lusztig induction can be easily generalized to the disconnected setting. The problem is then that Lusztig's multiplicity one result fails in the setting of disconnected groups.

We deal with this complication by viewing Deligne-Lusztig induction as a geometric analog of parabolic induction and drawing inspiration from the classical theory that decomposes principal series representations in terms of intertwining operators and the $R$-group. A geometric analog of the classical intertwining operators was recently introduced in the work of Bonnafe-Dat-Rouquier \cite{BDR17}. For a connected reductive $k_F$-group $G$ it gives a naturally defined $G$-equivariant isomorphism $H^*_c(Y_{B_1},\bar\Q_l)_\theta \to H^*_c(Y_{B_2},\bar\Q_l)_\theta$ between the middle-degree compact cohomology groups of the Deligne-Lusztig varieties associated to two Borel subgroups $B_1$ and $B_2$, after passing to $\theta$-isotypic components for a non-singular character $\theta$. The generalization to disconnected groups $G$ is routine and provided in this paper. This isomorphism can be thought of as the geometric analog of the classical integral intertwining operator between the parabolic inductions from two different Borel subgroups containing the same maximal split torus. Just like the case of the classical integral intertwining operators, these geometric intertwining operators do not compose correctly and need to be renormalized. We are able to derive a result analogous to Arthur's result \cite{ArtIOR1} on the normalization of classical intertwining operators for $p$-adic groups, namely that there exists a normalization, without being able to specify a canonical one. We then prove the analogs of Harish-Chandra's Commuting Algebra Theorem and Basis Theorem in our setting -- if $\kappa_{(S,\theta)}$ is the reducible cuspidal representation obtained from a non-singular character $\theta$ of an  elliptic maximal torus $S$, then the set of self-intertwining operators on $\kappa_{(S,\theta)}$, indexed by the elements of $\Omega(S,G)(k_F)_\theta$, forms a basis of the algebra of $G$-endomorphisms, where $\Omega(S,G)(k_F)_\theta$ are the elements of the Weyl group that fix the character $\theta$. This implies that there is a bijection
\begin{equation} \label{eq:dlclass}
[\kappa_{(S,\theta)}] \leftrightarrow \tx{Irr}(N(S,G)(k_F)_\theta,\theta), \end{equation}
where on the left side we have the set of irreducible constituents of $\kappa_{(S,\theta)}$ and on the right side we consider all irreducible representations of $N(S,G)(k_F)_\theta$ whose restriction to $S(k_F)$ is $\theta$-isotypic. Equivalently, the right-hand side is the set of $\theta$-projective representations of $\Omega(S,G)(k_F)_\theta$. This bijection preserves multiplicities -- the multiplicity of an irreducible constituent of $\kappa_{(S,\theta)}$ is equal to the dimension of the corresponding representation $\rho$ of $N(S,G)(k_F)_\theta$, equivalently to the multiplicity of $\theta$ in $\rho|_{S(k_F)}$.

The bijective correspondence \eqref{eq:dlclass} is a natural generalization to disconnected groups of Lusztig's result for connected groups. The latter is reflected in the structure of $N(S,G)(k)_\theta$ as follows: when $G$ is connected, the character $\theta$ extends (non\-canonically) to a character of $N(S,G)(k)_\theta$. Therefore the set of representations (in this case linear characters) of $N(S,G)(k)_\theta$ whose restriction to $S(k)$ is $\theta$-isotypic is a torsor for $\Omega(S,G)(k)_\theta^*$, and \eqref{eq:dlclass} turns $[\kappa_{(S,\theta)}]$ into a torsor for $\Omega(S,G)(k)_\theta^*$ as well.

Using Moy-Prasad theory \eqref{eq:dlclass} immediately implies the analogous result for the $p$-adic field $F$, where the irreducible pieces of the non-singular depth-zero supercuspidal representation $\pi_{(S,\theta)}=\tx{c-Ind}_{G(F)_x}^{G(F)}\kappa_{(S,\theta)}$ are parameterized by the irreducible $\theta$-isotypic representations of $N(S,G)(F)_\theta$. These results stand in remarkable analogy with the classical theory on the decomposition of principal series representations in terms of the $R$-group, even though we are dealing here with the opposite end of the spectrum -- elliptic tori and supercuspidal representations.

The problem of finding a canonical normalization of the geometric intertwining operators remains so far unsolved. It is a natural problem -- given a connected reductive group over a finite field, an elliptic maximal torus, a non-singular character of that torus, and a group of automorphisms of this data, there is a degree 2 cohomology class. A choice of a Borel subgroup containing the maximal torus determines a cocycle within this class. We have proved that this class is trivial in the situations relevant to this paper. Choosing a normalization of the intertwining operators amounts to choosing a trivialization of the 2-cocycle associated to a given Borel subgroup. Given the naturality of the 2-cocycle, it is to be expected that there will be a natural trivialization. This would lead to a canonification of the bijection \eqref{eq:dlclass}, which depends on the choice of normalization of intertwining operators.

The failure of multiplicity one for the representation $\pi_{(S,\theta)}$ has its reflection on the dual side as well. We recall that when the Langlands parameter $\varphi$ is regular there is a canonical isomorphism $S_\varphi \cong \hat S^\Gamma$ between the centralizer of the parameter and the Galois-fixed points of the torus dual to $S$. At the same time, the $L$-packet $\Pi_\varphi$ is the set of $\pi_{(jS,j\theta)}$ for all possible embeddings of $j : S \to G$, and each $\pi_{(jS,j\theta)}$ is irreducible. The internal structure of the $L$-packet is then an immediate consequence of Tate-Nakayama duality which describes the set of embeddings of $S$ into all inner forms of $G$ as a torsor under the finite abelian group dual to $\pi_0(\hat S^\Gamma)$ and its variations. In the non-regular case the isomorphism $S_\varphi \to \hat S^\Gamma$ is replaced by an exact sequence
\begin{equation} \label{eq:cent}
1 \to \hat S^\Gamma \to S_\varphi \to \Omega(S,G)(F)_\theta \to 1, \end{equation}
where $\Omega(S,G)$ is the absolute Weyl group of the torus $S$. The group $\pi_0(S_\varphi)$ is often non-abelian, which makes the structure of its irreducible representations more complicated. At the same time, since the representation $\pi_{(jS,j\theta)}$ is often reducible, the $L$-packet $\Pi_\varphi$ is now the union of the sets $[\pi_{(jS,j\theta)}]$ of irreducible constituents of $\pi_{(jS,j\theta)}$, for all possible embeddings of $j : S \to G$. We refer to each subset $[\pi_{(jS,j\theta)}]$ of $\Pi_\varphi$ as a \emph{Deligne-Lusztig packet}. Tate-Nakayama duality is no longer sufficient to describe the internal structure of the $L$-packet $\Pi_\varphi$, but it reduces it to establishing a bijection between the Deligne-Lusztig packet $[\pi_{(jS,j\theta)}]$ corresponding to a particular embedding $j : S \to G$ and the set of those irreducible representations of $\pi_0(S_\varphi)$ whose restriction to $\pi_0(\hat S^\Gamma)$ contains a specific character $\chi$ related to the embedding of $j$ (this is in the setting of pure inner forms, which we have adopted in this introduction to ease notation). It turns out that the extension \eqref{eq:cent} also doesn't have multiplicity one. Again this precludes the use of simple Clifford theory to study its representations. What one needs is a relationship between the extension
\[ 1 \to S(F) \to N(jS,G)(F)_{j\theta} \to \frac{N(jS,G)(F)_{j\theta}}{S(F)} \to 1, \]
which encodes the structure of $[\pi_{(jS,j\theta)}]$ according to \eqref{eq:dlclass}, and the extension
\[ 1 \to \hat S^\Gamma \to S_{\varphi,\chi} \to \Omega(S,G)(F)_{\theta,\chi} \to 1, \]
which encodes the the appropriate representations of $\pi_0(S_\varphi)$. It is easy to see that the cokernels of these extensions are canonically isomorphic. It turns out, rather miraculously, that the push-out of the first extension along $\theta : S(F) \to \C^\times$ is isomorphic to the push-out of the second extension along $\chi$. The argument is an amplification of an argument used in a different context, namely to establish the validity of a suitable statement of the local Langlands correspondence for disconnected groups whose connected component is a torus \cite{KalLLCD}. It relies on the cohomological pairings for complexes of tori of length 2 \cite[Appendix A.3]{KS99} and their extension to the setting of rigid inner forms. The choice of isomorphism between the two extensions is linked to the normalization of intertwining operators.

We now discuss the difficulties with positive depth representations. From the parameter $\varphi$ we obtain, as discussed above, the torus $S$ and after choosing $\chi$-data $(\chi_\alpha)_\alpha$ we also obtain the character $\theta_\chi$. In the case when $\varphi$ is regular, treated in \cite[\S5]{KalRSP}, we write the formula
\begin{equation} \label{eq:charwish_intro}
e(G)\epsilon(\frac{1}{2},X^*(T)_\C-X^*(S)_\C,\Lambda)\sum_{w \in N(S,G)(F)/S(F)} \Delta_{II}^\tx{abs}[a,\chi](\gamma^w)\theta_\chi(\gamma^w) \end{equation}
and use it to select for each embedding $j : S \to G$ a regular supercuspidal representation $\pi_j$ whose character on shallow elements of $S(F)$ is given by this formula. Of course we need to know that such a representation exists. This uses the material of \cite[\S3]{KalRSP}, where a bijection is established between the set of pairs $(S,\theta)$ consisting of a tame elliptic maximal torus $S$ and a regular character $\theta$ and the set of regular supercuspidal representation $\pi_{(S,\theta)}$; as well as the material of \cite[\S4]{KalRSP}, where the Adler-DeBacker-Spice character formula of \cite{AS09} and \cite{DS18} is reinterpreted in the case of $\pi_{(S,\theta)}$ as the formula
\begin{equation} \label{eq:chargiven_intro}
e(G)\epsilon(\frac{1}{2},X^*(T)_\C-X^*(S)_\C,\Lambda)\!\!\!\!\!\!\!\!\!\!\!\!\sum_{w \in N(S,G)(F)/S(F)}\!\!\!\!\!\!\!\!\!\!\!\! \Delta_{II}^\tx{abs}[a,\chi'](\gamma^w)e_{f,\tx{ram}}(\gamma^w)e^\tx{ram}(\gamma^w)\theta(\gamma^w), 
\end{equation}
where now $(\chi'_\alpha)_\alpha$ is $\chi$-data computed in terms of $\theta$, and $\gamma \in S(F)$ is shallow. The two formulas \eqref{eq:charwish_intro} and \eqref{eq:chargiven_intro} look very similar, except for the occurrence of the characters $e_{f,\tx{ram}}$ and $e^\tx{ram}$ in the second formula, and the usage of the particular $\chi$-data $\chi'$, which has the property of being minimally ramified. One then uses the fact that $\theta_\chi$ is regular when $\chi$ is chosen minimally ramified, and that the characters $e_{f,\tx{ram}}$ and $e^\tx{ram}$ are invariant under $N(S,G)(F)$, so that $\theta_\chi'=\theta_\chi \cdot e_{f,\tx{ram}} \cdot e^\tx{ram}$ is also regular. This means that a simple reparameterization of the correspondence $(S,\theta) \mapsto \pi_{(S,\theta)}$ identifies the two formulas and enables the construction of the $L$-packet.

Consider now the case when $\theta$ is no longer regular, but is still non-singular. One can combine the material of \cite[\S3]{KalRSP} with the results on non-singular supercuspidal representations of depth zero from this paper to obtain a (usually reducible) positive depth supercuspidal representation $\pi_{(S,\theta)}$. The material of \cite[\S4]{KalRSP} applies to this representation, so we have \eqref{eq:chargiven_intro}. It is easy to see that multiplication by $e_{f,\tx{ram}}$ preserves non-singularity. But multiplication by $e^\tx{ram}$ does not. So we cannot simply replace $\theta$ by $\theta'=\theta\cdot e_{f,\tx{ram}}\cdot e^\tx{ram}$ as in the regular case. 
There is a parallel phenomenon on the Galois side. In the depth-zero case, the character $\theta_\chi$ obtained from a general supercuspidal parameter $\varphi$ after factoring through an $L$-embedding $^Lj_\chi : {^LS} \to {^LG}$ constructed from minimally ramified $\chi$-data is non-singular. But in the positive depth case this is no longer true. 

It turns out that these problems are related to issues with Yu's construction of supercuspidal representations. It was noticed by Loren Spice that there is an error in Yu's paper, which breaks the proofs of the intertwining statements \cite[Proposition 14.1, Theorem 14.2]{Yu01}. Counterexamples to these statements were then produced by Fintzen \cite{Fin19}. It was also shown in \cite{Fin19} that despite this error, Yu's construction still produces irreducible supercuspidal representations. But the failure of the intertwining results suggests that the construction may not be in optimal form. For example, the validity of the intertwining results is an essential input in the computation of Harish-Chandra characters in \cite{Spice17}. In \cite{FKS}, a natural modification of Yu's construction is proposed, which restores the validity of Yu's intertwining statements. If one computes the character formula for the resulting representations at shallow elements, one obtains an analog of \eqref{eq:chargiven_intro}, in which the auxiliary characters $e_{f,\tx{ram}}$ and $e^\tx{ram}$ have disappeared, and the $\chi$-data $\chi'$ has been replaced with a different $\chi$-data $\chi''$. The latter is no longer minimally ramified. Instead, it has the very useful property that if it is used to factor a supercuspidal parameter $\varphi$ and obtain from it a character $\theta_{\chi''}$ of $S(F)$, then $\theta_{\chi''}$ will be non-singular. In this way, the analog of \eqref{eq:chargiven_intro} for the modified Yu construction of \cite{FKS} becomes identified with \eqref{eq:charwish_intro}, which allows the construction of positive depth supercuspidal $L$-packets to proceed beyond the regular case.

The reason for this pleasant property of the $\chi$-data $\chi''$ is that it mirrors the inductive structure of Yu's construction. Recall that part of a Yu-datum is a tower $G^0 \subset \dots \subset G^d$ of twisted Levi subgroups of $G$. This tower is easily obtained from the pair $(S,\theta)$. The $\chi$-data $\chi''$ is obtained by putting together $\chi$-data for $R(Z(G^i)^\circ,G^{i+1})$ for all $i=0,\dots,d-1$. It is shown in \cite{KalDC} that a set of $\chi$-data for $R(Z(G^i)^\circ,G^{i+1})$ leads to an $L$-embedding $^LG^i \to {^LG^{i+1}}$. It is shown further that if the $\chi$-data for $R(S,G)$ is obtained by putting together $\chi$-data for $R(Z(G^i)^\circ,G^{i+1})$, then the $L$-embedding $^LS \to {^LG}$ is obtained by composing the various $L$-embeddings $^LS \to {^LG^0}$ and $^LG^i \to {^LG^{i+1}}$. In this way, the supercuspidal $L$-packets on each $G^i$ are related to each other as functorial transfers.

An important conjectural property of $L$-packets is their stability, and more generally their endoscopic transfer. The packets constructed in this paper do satisfy stability, as well as endoscopic transfer for all $s \in \hat S^\Gamma \subset S_\varphi$. This is proved in \cite{FKS}. The more general case of endoscopy, for elements $s \in S_\varphi \sm \hat S^\Gamma$, is current work in progress. In addition, it is shown in \cite{FKS} that the packets constructed here satisfy another property, namely \cite[Conjecture 4.3]{KalDC}. The theory of endoscopy implies that there is at most one construction of a refined local Langlands correspondence for supercuspidal parameters that satisfies endoscopic transfer and \cite[Conjecture 4.3]{KalDC}.

We now discuss the structure of this paper. In Section \ref{sec:nsdl-fin} we consider a possibly disconnected reductive group defined over a finite field $k$, an elliptic maximal torus $S \subset G$, and a non-singular character $\theta : S(k) \to \bar\Q_l^\times$ in the sense of Deligne-Lusztig. In Subsection \ref{sub:ff-setup} we explain exactly what we mean by that. Let $N(S,G)(k)_\theta$ resp. $\Omega(S,G)(k)_\theta$ be the stabilizers of $\theta$ in the $k$-points of the normalizer of $S$ in $G$ resp. the $k$-points of the Weyl group. In Subsection \ref{sub:nsdlft} we review Lusztig's results for connected groups, to the effect that the Deligne-Lusztig virtual character $\kappa_{(S,\theta)}$, which in this case is a cuspidal representation of $G(k)$, has multiplicity one and the set of its irreducible components receives a natural simply transitive action of the Pontryagin dual of the abelian group $\Omega(S,G)(k)_\theta$. In Subsection \ref{sub:natiop} we define the concept of a natural intertwining operator $H^*_c(Y_{B_1},\bar\Q_l)_\theta \to H^*_c(Y_{B_2},\bar\Q_l)_\theta$ between the $\theta$-isotypic components of the middle degree compact cohomologies of the Deligne-Lusztig varieties associated to two Borel subgroups $B_1$ and $B_2$ containing $S$. Such an operator is well-defined up to a scalar. This definition is elementary, but it allows us to study the existence of normalized collections of such operators. We are especially interested in normalized collections that are equivariant with respect to the action of the group of automorphisms of $G$ that preserve $S$ and $\theta$, or some subgroup thereof. In Subsection \ref{sub:gio} we review the work of Bonnafe-Dat-Rouquier \cite{BDR17}, which gives rise to an equivariant collection of natural intertwining operators, which is however not normalized. In Subsection \ref{sub:ff-disc} we generalize this discussion to disconnected reductive groups. In Subsection \ref{sub:dlpackpar} we prove that any normalized collection of natural intertwining operators that is equivariant with respect to the action of $N(S,G)(k)_\theta$ provides the bijection \eqref{eq:dlclass} between the set of irreducible constituents of $\kappa_{(S,\theta)}$ and the set of representations of $N(S,G)(k)_\theta$ whose restriction to $S(k)$ is $\theta$-isotypic. We do this by obtaining from the operators $H^*_c(Y_{B_1},\bar\Q_l)_\theta \to H^*_c(Y_{B_2},\bar\Q_l)_\theta$ a collection of self-intertwining operators of $H^*_c(Y_{B_1},\bar\Q_l)_\theta$ indexed by elements of $N(S,G)(k)_\theta$, which gives us an action of $N(S,G)(k)_\theta$ on $H^*_c(Y_{B_1},\bar\Q_l)_\theta$ that extends the action of $S(k)$ on the right given by $\theta$. We prove that these operators form a basis for the intertwining algebra, and then decompose $H^*_c(Y_{B_1},\bar\Q_l)_\theta$ under the action of $N(S,G)(k)_\theta$.

In Section \ref{sec:nsdl-dz} we consider a connected reductive group defined over a non-archimedean local field $F$, an elliptic maximally unramified maximal torus $S \subset G$, and a character $\theta : S(F) \to \C^\times$ that is non-singular in the sense of Definition \ref{dfn:nsc}. When $\theta$ is of depth zero, this definition implies that (but is stronger than) the character $\theta^\circ$ of $\ms{S}^\circ(k)=S(F)_{0:0+}$ obtained from $\theta$ is non-singular with respect to the finite group of Lie type $\ms{G}_x^\circ(k)=G(F)_{x,0:0+}$ in the sense of Deligne-Lusztig, where $x$ is the point of the Bruhat--Tits building of $G$ associated to $S$. In Subsection \ref{sub:ass_check} we verify some assumptions made in Section \ref{sec:nsdl-fin}, including the vanishing of the degree 2 cohomology class that is the obstruction to the existence of normalized intertwining operators. This allows us to apply the results of Section \ref{sec:nsdl-fin} to the disconnected reductive group $\ms{G}_x$. Compact induction and Moy-Prasad theory translate these results to analogous results for depth-zero supercuspidal representations in Subsection \ref{sub:nsdl-dz}.
In Subsection \ref{sub:nsdl-pd} we combine these results with the Howe factorization algorithm of \cite[\S3.6]{KalRSP} and Yu's construction to obtain a multiplicity preserving bijection between $[\pi_{(S,\theta)}]$ and $\tx{Irr}(N(S,G)(F)_\theta,\theta)$ in the case where $S \subset G$ is a tame elliptic maximal torus and $\theta : S(F) \to \C^\times$ is a positive depth character that is non-singular.

The construction of $L$-packets is the subject of Section \ref{sec:pack}. In Subsection \ref{sub:lparam} we discuss how to extract from a supercuspidal parameter $\varphi$ and $\chi$-data a tame torus $S$ and a character $\theta_\chi$ of $S(F)$. The arguments are similar to \cite[\S5.2]{KalRSP}, but we have to pay attention to the subtleties introduced by non-singularity that we discussed above. The construction of the $L$-packets takes place in Subsection \ref{sub:lpackconst}. We then proceed with the study of their internal structure. In Subsection \ref{sub:sphi} we give an example that $S_\varphi$ can be finite and non-abelian even for a classical root system of type $B_4$, establish the exact sequence \eqref{eq:cent}, show that it has multiplicity one when $\varphi$ has depth zero and $G$ is unramified or simply connected, and then give an example where it does not have multiplicity one. In Subsection \ref{sub:is1} we use Tate-Nakayama duality to reduce the internal structure of the $L$-packet to that of a single Deligne-Lusztig packet of depth zero. That latter case is dealt with in Subsection \ref{sub:is2}. In the final Subsection \ref{sub:endo} we sketch an argument showing that the $L$-packet with this internal parameterization satisfies stability and some cases of endoscopic transfer. The details of this argument appear in \cite{FKS}, while its generalization to all cases of endoscopic transfer will be the subject of a forthcoming paper.

There are a number of appendices to this paper containing information of more technical nature. Some of them review, and possibly extend, known material in a way convenient for our purposes. Such are \S\ref{app:cliff} containing an overview of basic Clifford theory, \S\ref{app:basis} containing an abstract version of the Harish-Chandra basis and commuting algebra theorems, \S\ref{app:repext} containing a discussion of representations of extensions with abelian quotients that may fail the multiplicity one property, \S\ref{app:dlreview} describing the behavior of Deligne-Lusztig induction under homomorphisms of algebraic groups with abelian kernel and cokernel. Others contain results about Bruhat-Tits theory, such as the compatibility of parahoric subgroups with restriction of scalars \S\ref{app:parahoric}, or the concept of absolutely special vertices \S\ref{app:absvert} generalizing the concept of hyperspecial vertices, as well as that of superspecial vertices of \cite[Definition 3.4.8]{KalRSP}. In \S\ref{app:gendz} we extend to the case of ramified groups the results of \cite[\S6.1]{DR09} about genericity of depth zero supercuspidal representations. This allows us to prove the existence and uniqueness of generic constituents in our depth-zero $L$-packets. Finally, \S\ref{app:d2n} contains technical results about the root system $D_{2n}$.

\tb{Acknowledgements:} This paper was born out of discussions with Cheng-Chiang Tsai, whose insightfulness played an essential role. Brian Conrad provided the argument for Appendix \ref{app:parahoric}, and Gopal Prasad that for Proposition \ref{pro:supergp}. Michael Harris and Rapha\"el Rouquier offered stimulating conversations, interest, and support. It is a pleasure to thank them all.

{\let\thefootnote\relax\footnotetext{This research is supported in part by NSF grant DMS-1801687 and a Sloan Fellowship.}}

\section{Non-singular Deligne-Lusztig packets over finite fields} \label{sec:nsdl-fin}

Let $G$ be a connected reductive group defined over a finite field $k$, $S \subset G$ an elliptic maximal torus, $\theta : S(k) \to \bar\Q_l^\times$ a character. Assume that $\theta$ is non-singular, in the sense of \cite[Definition 5.15]{DL76}. Recall \cite[Lemma 3.4.14]{KalRSP} that this is equivalent to demanding that for each $\alpha \in R(S,G)$ the character $\theta\circ N \circ \alpha^\vee$ of $(k')^\times$ is non-trivial, where $k'/k$ is some finite extension splitting $S$ and $N : S(k') \to S(k)$ is the norm map.

Let $\mc{R}_S^G(\theta)$ be the virtual representation of $G(k)$ obtained from $(S,\theta)$ via Deligne-Lusztig induction. Let $\sigma(G)$ and $\sigma(S)$ be the $k$-ranks of $G$ and $S$, respectively. According to \cite[Theorem 8.3]{DL76}, $\kappa_{(S,\theta)}^G := (-1)^{\sigma(G)-\sigma(S)}\mc{R}_S^G(\theta)$ is an actual cuspidal representation of $G(k)$.  It need not be irreducible. Let $N(S,G)(k)_\theta$ and $\Omega(S,G)(k)_\theta$ be the stabilizers of $\theta$ in the normalizer of $S$ in $G$, and the Weyl group, respectively. A result of Lusztig states that each irreducible constituent of $\kappa_{(S,\theta)}^G$ occurs with multiplicity $1$ (cf. Theorem \ref{thm:multfree}) and the character group $\Omega(S,G)(k)_\theta^*$ of the (as it turns out abelian) group $\Omega(S,G)(k)_\theta$ has a natural simply transitive action on the set $[\kappa_{(S,\theta)}^G]$ of irreducible constituents of $\kappa_{(S,\theta)}^G$ (cf. Proposition \ref{pro:dlp}).

For our applications to $p$-adic we will also need to consider $k$-groups $G$ that are not necessarily connected. We will extend the discussion of $\kappa_{(S,\theta)}^G$ to that case. An essential new feature is that the irreducible constituents of this representation may occur with higher multiplicity. We will generalize Lusztig's results alluded to above to a multiplicity-preserving bijection between $[\kappa_{(S,\theta)}^G]$ and the set $\tx{Irr}(N(S,G)(k)_\theta,\theta)$, where $N(S,G)(k)_\theta$ is the stabilizer of the pair $(S,\theta)$ in $G(k)$ and we are considering those irreducible representations whose restriction to $S(k)$ is $\theta$-isotypic. 

In the connected case this bijection is immediate from Lusztig's result and the elementary fact (cf. Proposition \ref{pro:dlcharext}) that $\theta$ extends (non-canonically) to a character of $N(S,G)(k)_\theta$. In the disconnected case we have to develop different arguments. They are based on the concept of an intertwining operator in the setting of Deligne-Lusztig induction and rely on recent results of Bonnaf\'e-Dat-Rouquier. The bijection will depend on a choice of normalization of the intertwining operators. Luszitg's results in the connected case will still play an essential role in our argument. It will be convenient to refer to $[\kappa_{(S,\theta)}^G]$ as a \emph{non-singular Deligne-Lusztig packet}.

\subsection{The set-up} \label{sub:ff-setup}

Let $G$ be a smooth $k$-group scheme whose neutral connected component $G^0$ is reductive. Consider given a central subgroup $Z \subset G$.

\begin{asm}\label{asm:dl_mu} The component group $\pi_0(G)$ is a finitely generated abelian group, and $[G^0 \cdot Z : G] < \infty$.
\end{asm}

The commutativity of $\pi_0(G)$ implies in particular that any subgroup $G'$ of $G$ containing $G^0$ is normal. It is immediate that $G'$ also satisfies Assumption \ref{asm:dl_mu}.

\begin{ntt} \label{ntt:dl_s}
For any maximal torus $S^0 \subset G^0$ we set $S=S^0 \cdot Z$.
\end{ntt}

We emphasize that this equality is on the level of algebraic groups, and that $S^0(k) \cdot Z(k)$ may well be a proper subgroup of $S(k)$.

We write $N(S,G)$ for the normalizer of $S$ in $G$ and $\Omega(S,G)=N(S,G)/S$. It is immediate that $S^0=S \cap G^0$, which implies $N(S,G)=N(S^0,G)$. Furthermore $N(S^0,G^0) \cdot S$ is of finite index in $N(S,G)$ by Assumption \ref{asm:dl_mu}. It is the fact that $N(S,G)$ may be strictly larger than $N(S^0,G^0) \cdot S$, and in particular act on $G^0$ by outer automorphisms, that necessitates the discussion of disconnected groups.

We will be particulary interested in the case that $S^0$ is an elliptic maximal torus of $G^0$. Let $\theta : S(k) \to \bar\Q_l^\times$ be a character whose restriction $\theta^0$ to $S^0(k)$ is non-singular.

\subsection{The basic bicharacter} \label{sub:bichar-dl}

For the next two results we assume that $G=G^0$. Recall from \cite{Lus88} the notion of a regular embedding. It is an embedding $G \to G'$ of connected reductive groups defined over $k$ that induces an isomorphism on the level of adjoint groups, and such that $G'$ has connected center. We will identify $G'_\tx{ad}$ with $G_\tx{ad}$. By Lang's theorem $G'(k) \to G_\tx{ad}(k)$ is surjective. There is a unique maximal torus $S' \subset G'$ containing $S$, and again the natural map $S'(k) \to S_\tx{ad}(k)$ is surjective.

\begin{lem} \label{lem:dlbichar}
\begin{enumerate}
\item Let $w \in \Omega(S,G)(k)$ and $s_\tx{ad} \in S_\tx{ad}(k)$. Choose a lift $s_\tx{sc} \in S_\tx{sc}(\bar k)$. Then the element $ws_\tx{sc}w^{-1}s_\tx{sc}^{-1} \in S_\tx{sc}(\bar k)$ belongs to $S_\tx{sc}(k)$ and is independent of the choice of $s_\tx{sc}$.
\item The map
	\begin{equation} \label{eq:dzbichar1}
	\Omega(S,G)(k)_\theta \times \tx{cok}(S(k) \to S_\tx{ad}(k)) \to \bar\Q_l^\times,\qquad (w,s_\tx{ad}) \to \theta(ws_\tx{sc}w^{-1}s_\tx{sc}^{-1})
	\end{equation}
	is well-defined and bi-multiplicative.
\item Its left kernel is trivial, i.e. the induced map
	\begin{equation} \label{eq:wtheta1}
	\Omega(S,G)(k)_\theta \to \tx{cok}(S(k) \to S_\tx{ad}(k))^* \end{equation}
	is injective.
\item If $G \to G'$ is a regular embedding and $\theta' : S'(k) \to \bar\Q_l^\times$ any extension of $\theta$, then the map
    \[ \Omega(S,G)(k)_\theta \times \tx{cok}(S(k) \to S'(k)) \to \bar\Q_l^\times,\qquad (w,s') \mapsto \theta'(ws'w^{-1}s'^{-1}) \]
    is well-defined and bi-multiplicative, and equals the composition of the above pairing with the natural map $\tx{cok}(S(k) \to S'(k)) \to \tx{cok}(S(k) \to S_\tx{ad}(k))$.

\end{enumerate}
\end{lem}
\begin{proof}
There is $z \in Z(G_\tx{sc})$ s.t. $F(s_\tx{sc})=z_\tx{sc} \cdot s_\tx{sc}$, where $F$ denotes the Frobenius endomorphism. Hence $ws_\tx{sc}w^{-1}s_\tx{sc}^{-1} \in S_\tx{sc}(k)$. The independence of the choice of $s_\tx{sc}$ is immediate. If $s_\tx{ad}$ is the image of $s \in S(k)$, then the image of $ws_\tx{sc}w^{-1}s_\tx{sc}^{-1}$ under $S_\tx{sc}(k) \to S(k)$ equals $wsw^{-1}s^{-1}$, which lies in the kernel of $\theta$. Multiplicativity in $s_\tx{ad}$ is obvious. Multiplicativity in $\Omega(S,G)(k)_\theta$ follows from $uvs_\tx{sc}v^{-1}u^{-1}s_\tx{sc}^{-1}=u(vs_\tx{sc}v^{-1}s_\tx{sc}^{-1})u^{-1} \cdot us_\tx{sc}u^{-1}s_\tx{sc}^{-1}$ and the fact that $u$ fixes $\theta$.

The fact that $(w,s') \mapsto \theta'(ws'w^{-1}s'^{-1})$ is well-defined is immediate. Given $s' \in S'(k)$ let $s_\tx{ad} \in S_\tx{ad}(k)$ be its image and choose a lift $s_\tx{sc} \in S_\tx{sc}(\bar k)$ of $s_\tx{ad}$. Then $ws'w^{-1}s'^{-1}$ belongs to $S(k)$ and equals the image of $ws_\tx{sc}w^{-1}s_\tx{sc}^{-1}$ in $S(k)$.

To prove that the left kernel of (either) pairing is trivial, we observe that since $\theta$ is non-singular, so is $\theta'$, but then \cite[Proposition 5.16]{DL76} implies that $\theta'$ is in general position. If $w$ is such that $\theta(ws_\tx{sc}w^{-1}s_\tx{sc}^{-1})=1$ for all $s_\tx{ad} \in S_\tx{ad}(k)$, then $\theta'(ws'w^{-1}s'^{-1})=1$ for all $s' \in S'(k)$, but then $w=1$.
\end{proof}

\begin{cor} \label{cor:weylab1} The group $\Omega(S,G)(k)_\theta$ is abelian.
\end{cor}

We now drop the assumption that $G=G^0$. Write $N(S,G)$ for the normalizer of $S$ in $G$, $\Omega(S,G)=N(S,G)/S$, and $N(S,G)(k)_\theta$ and $\Omega(S,G)(k)_\theta$ for the stabilizers of $\theta$ in these two groups. A-priori $\Omega(S,G)(k)_\theta$ may be a proper subgroup of $\Omega(S,G)(k)_{\theta^0}$. This discrepancy is measured by the following object.

\begin{lem} \label{lem:dzbichar}
\begin{enumerate}
\item The map
\begin{equation} \label{eq:dzbichar}  \frac{\Omega(S,G)(k)_{\theta^0}}{\Omega(S,G)(k)_\theta} \times \frac{S(k)}{S^0(k)} \to \C^\times,\quad (w,s) \mapsto \theta^0(wsw^{-1}s^{-1}) \end{equation}
is well-defined, bi-multiplicative, and has a trivial left kernel.
\item If $w \in \Omega(S^0,G^0)(k)_{\theta^0}$, then the value at $(w,s)$ is equal to the value of the bicharacter of Lemma \ref{lem:dlbichar} at $(w,\bar s)$, where $\bar s \in S^0_\tx{ad}(k)$ is the image of $s$ under the map $S \to S/Z_G=S^0/(Z_G \cap G^0) \to S^0/Z_{G^0}$.
\end{enumerate}
\end{lem}
\begin{proof}
Consider $s \in S(k)$ and $w \in \Omega(S,G)(k)_{\theta^0}$. Then $wsw^{-1}s^{-1}$ lies in $S^0(k)$ and this allows us to consider $\theta^0(wsw^{-1}s^{-1})$. If $s \in S^0(k)$, then also $wsw^{-1} \in S^0(k)$ and $\theta^0(wsw^{-1}s^{-1})=\theta^0(wsw^{-1})\theta^0(s)^{-1}=1$. For any $s,w$ we have $\theta^0(wsw^{-1}s^{-1})=\theta(wsw^{-1})\theta(s)^{-1}$. This shows that $\theta^0(wsw^{-1}s^{-1})$ is trivial for $w \in \Omega(S,G)(k)_\theta$, and is multiplicative in $s$. For multiplicativity in $\Omega(S,G)(k)_\theta$ the argument is as in the proof of Lemma \ref{lem:dlbichar} -- we have for $u,v \in \Omega(S,G)(k)_{\theta^0}$
\[ uvsv^{-1}u^{-1}s^{-1}=u(vsv^{-1}s^{-1})u^{-1}(usu^{-1}s^{-1})\]
and the two terms in parenthesis belong to $S^0(k)$. Since $u$ fixes $\theta^0$ the desired multiplicativity follows. Finally, for a fixed $w$ we have $\theta^0(wsw^{-1}s^{-1})=1$ for all $s$ precisely when $w \in \Omega(S,G)(k)_\theta$.

For the second point, let $\dot s \in S^0(\bar k)$ be a preimage of $\bar s$. Then $s\dot s^{-1} \in S(\bar k)$ centralizes $G^0$, so $wsw^{-1}s^{-1}=w\dot sw^{-1}\dot s^{-1}$.
\end{proof}

\begin{cor} \label{cor:bichar} 
The bicharacter \eqref{eq:dzbichar} induces group homomorphisms
\[ \frac{\Omega(S,G)(k)_{\theta^0}}{\Omega(S,G)(k)_\theta} \into \left(\frac{S(k)}{S^0(k)}\right)^*\qquad\tx{and}\qquad \frac{S(k)}{S^0(k)} \onto \left(\frac{\Omega(S,G)(k)_{\theta^0}}{\Omega(S,G)(k)_\theta}\right)^* \]
the first of which is injective, and the second surjective. Moreover, we have the commutative diagram
\[ \xymatrix{
S(k)/S^0(k)\ar[r]\ar[d]&\left(\Omega(S,G)(k)_{\theta^0}/\Omega(S,G)(k)_\theta\right)^*\ar[d]\\
\tx{cok}(S^0(k) \to S^0_\tx{ad}(k))\ar[r]&\Omega(S^0,G^0)(k)_\theta^*
}
\]
where the top map is the surjective homomorphism, the bottom map is the dual of \eqref{eq:wtheta1}, the left map is induced by $S \to S/Z_G =S^0/(Z_G \cap G^0) \to S^0_\tx{ad}$, and the right map is induced by the natural inclusion $\Omega(S^0,G^0) \to \Omega(S,G)$.
\end{cor}

\begin{rem} Note that the injective homomorphism describes $\Omega(S,G)(k)_\theta$ as the kernel of the map $\Omega(S,G)(k)_{\theta^0} \to (S(k)/S^0(k))^*$. Since this map involves only $\theta^0$, we see that $\Omega(S,G)(k)_\theta$ depends on $\theta$ only through its restriction $\theta^0$.
\end{rem}

\begin{rem} In fact, we have the following four, a-priori different, Weyl groups
\[ \xymatrix{
	\Omega(S,G^0)(k)_\theta\ar@{^(->}[r]\ar@{^(->}[d]&\Omega(S,G)(k)_\theta\ar@{^(->}[d]\\
	\Omega(S,G^0)(k)_{\theta^0}\ar@{^(->}[r]&\Omega(S,G)(k)_{\theta^0}
}
\]
If $S(k)=S^0(k)\cdot Z(k)$, then the two vertical maps are isomorphisms, but in general they can be proper inclusions. We know from Corollary \ref{cor:weylab1} that $\Omega(S,G^0)(k)_{\theta^0}$ is abelian. In the applications for $p$-adic groups it will be true that $\Omega(S,G)(k)_{\theta^0}$ is also abelian, but we do not know if this is true in general.
\end{rem}

\subsection{Lusztig's results for $G^0$} \label{sub:nsdlft}

In this subsection we assume $G=G^0$. We review here the main results of \cite{Lus88} in our special case, but formulate them without reference to the dual group of $G$. A crucial technical result is the following.

\begin{thm}[Lusztig, \cite{Lus88}] \label{thm:multfree} The representation $\kappa_{(S,\theta)}^G$ is multiplicity free.
\end{thm}

The conjugation action of $G_\tx{ad}(k)$ on $G(k)$ induces an action of $\tx{cok}(G(k) \to G_\tx{ad}(k))$ on the set of isomorphism classes of irreducible representations of $G(k)$. The embedding $S \to G$ induces an isomorphism of groups $\tx{cok}(S(k) \to S_\tx{ad}(k)) \to \tx{cok}(G(k) \to G_\tx{ad}(k))$, and since conjugation by $S_\tx{ad}(k)$ preserves the pair $(S,\theta)$, the action of $\tx{cok}(G(k) \to G_\tx{ad}(k))$ preserves the set $[\kappa_{(S,\theta)}^G]$.

The dual of \eqref{eq:wtheta1} is a natural surjective map of abelian groups
\begin{equation} \label{eq:wtheta1'}
\tx{cok}(G(k) \to G_\tx{ad}(k)) = \tx{cok}(S(k) \to S_\tx{ad}(k)) \to \Omega(S,G)(k)_\theta^*.
\end{equation}

\begin{pro} \label{pro:dlp}
The natural action of $\tx{cok}(G(k) \to G_\tx{ad}(k))$ on the Deligne-Lusztig packet $[\kappa_{(S,\theta)}^G]$ factors through a simply transitive action of $\Omega(S,G)(k)_\theta^*$.
\end{pro}
\begin{proof}
Consider again a regular embedding $G \to G'$, let $S' \subset G'$ be the maximal torus containing $S$, and let $\theta' : S'(k) \to \bar\Q_l^\times$ be an extension of $\theta$. Recall from Appendix \ref{app:dlreview} that there is a natural isomorphism $\kappa_{(S,\theta)}^G \to \kappa_{(S',\theta')}^{G'}$ that intertwines the embedding $G(k) \to G'(k)$. In other words, the representation $\kappa_{(S,\theta)}^G$ of $G(k)$ is the restriction of the representation $\kappa_{(S',\theta')}^{G'}$ of $G'(k)$. Since $\theta'$ is in general position, the latter is irreducible. We can thus apply Clifford theory to the exact sequence
\[ 1 \to G(k) \to G'(k) \to S'(k)/S(k) \to 1, \]
bearing in mind Theorem \ref{thm:multfree}. By Lemma \ref{lem:cliff1} the action of $S'(k)/S(k)$ on $G(k)$ induces a transitive action on the set of irreducible constituents of $\kappa_{(S,\theta)}^G$, whose kernel is the annihilator of the subgroup of $[S'(k)/S(k)]^*$ that stabilizes $\kappa_{(S',\theta')}^{G'}$. To see what the latter is, let $\delta : S'(k)/S(k) \to \bar\Q_l^\times$. By Appendix \ref{app:dlreview} we have $\delta \otimes \kappa_{(S',\theta')}^{G'} = \kappa_{(S',\delta\theta')}^{G'}$. By \cite[Proposition 5.26 and Corollary 6.3]{DL76} this is isomorphic to $\kappa_{(S',\theta')}^{G'}$ if and only if the characters $\theta'$ and $\delta\theta'$ are conjugate under $\Omega(S',G')(k)$. That is, if and only if there exists $w \in \Omega(S',G')(k)=\Omega(S,G)(k)$ such that $\delta = w\theta' \cdot \theta'^{-1}$. The stabilizer of $\kappa_{(S',\theta')}^{G'}$ in $[S'(k)/S(k)]^*$ is thus the image of $\Omega(S,G)(k)_\theta$ under \eqref{eq:wtheta1}. Therefore the kernel of the action of $[S'(k)/S(k)]$ on $[\kappa_{(S,\theta)}^G]$ is precisely the kernel of \eqref{eq:wtheta1'}.
\end{proof}

We thus see that the set $[\kappa_{(S,\theta)}^G]$ is a torsor for the finite abelian group $\Omega(S,G)(k)_\theta^*$. For our applications it will be important to reinterpret this torsor using the group $N(S,G)(k)_\theta$.

\begin{pro} \label{pro:dlcharext} The character $\theta$ extends to the group $N(S,G)(k)_\theta$.
\end{pro}
\begin{proof} According to Lang's theorem we have $\Omega(S,G)(k)=N(S_\tx{sc},G_\tx{sc})(k)/S_\tx{sc}(k)$ and this implies $N(S,G)(k)=N(S_\tx{sc},G_\tx{sc})(k)\cdot S(k)$. It is thus enough to show that $\theta|_{S_\tx{sc}(k)}$ extends to $N(S_\tx{sc},G_\tx{sc})(k)_\theta$. Since the latter is a subgroup of the stabilizer of $\theta|_{S_\tx{sc}(k)}$ in $N(S_\tx{sc},G_\tx{sc})(k)$, we may as well assume that $G=G_\tx{sc}$. Then $G$ is a product of $k$-simple factors and we may deal with each factor individually and then take the product of the extensions. Thus we may assume that $G$ is $k$-simple. Then $G=\tx{Res}_{k'/k}G'$, with $G'$ an absolutely simple simply connected group defined over a finite extension $k'$ of $k$. Since $S$ is defined over $k$, it is of the form $S=\tx{Res}_{k'/k}S'$ with $S' \subset G'$ an elliptic maximal torus defined over $k'$. If $G'$ is of type other than $D_{2n}^{(1)}$, the group $\Omega(S,G)(k)_\theta$ is cyclic, hence $H^2(\Omega(S,G)(k)_\theta,\bar\Q_l^\times)=1$ and the extendibility of $\theta$ follows from Lemma \ref{lem:cliff1}. If $G'$ is of type $D_{2n}^{(1)}$ then $\Omega(S,G)(k)_\theta$ is one of $\{1\}$, $\Z/2\Z$, or $(\Z/2\Z)^2$. In the first two cases the extendibility of $\theta$ follows by the same argument, while in the third case it follows from Lemma \ref{lem:cliff1} together with Lemma \ref{lem:d2nr}.
\end{proof}

We thus obtain a second torsor for the group $\Omega(S,G)(k)_\theta^*$, namely the set of extensions of $\theta$ to $N(S,G)(k)_\theta$. This set is thus in non-canonical bijection with $[\kappa_{(S,\theta)}^G]$. In the following we shall discuss what it takes to specify such a bijection.

\subsection{Natural intertwining operators} \label{sub:natiop}

We continue to assume $G=G^0$. Recall that a specific representation of $G(k)$ in the isomorphism class of $\kappa_{(S,\theta)}^G$ is obtained as follows. Let $F$ be the Frobenius endomorphism of $G$. Let $U \subset G$ be the unipotent radical of a Borel subgroup of $G$, defined over $\bar k$, and containing $S$. The corresponding Deligne-Lusztig variety
\[ Y_U = \{ gU \in G/U| g^{-1}F(g) \in U \cdot FU \} \]
receives an action of $G(k)$ by left multiplication and of $S(k)$ by right multiplication. The $l$-adic cohomology group $H^i_c(Y_U,\bar\Q_l)$ inherits both actions, and we may consider the $\theta$-isotypic component for the right action of $S(k)$
\[ H^i_c(Y_U,\bar\Q_l)_\theta = \{v \in H^i_c(Y_U,\bar\Q_l)| vs = \theta(s) \forall s \in S(k) \}. \]
According to \cite[Corollary 9.9]{DL76} and \cite{He08} this component vanishes for all but one $i$, namely $i=d(U,FU)$, where $d(U,FU)$ denotes the number of root hyperplanes separating the Weyl chambers of $U$ and $FU$ respectively, and is also equal to the dimension of the variety $X_U=Y_U/S(k)$. We have $\kappa_{(S,\theta)}^G=H^{d_U}_c(Y_U,\bar\Q_l)_\theta$ for $d_U=d(U,FU)$.

Let $V$ be the unipotent radical of another Borel subgroup containing $S$. Then $H^{d_V}_c(Y_V,\bar\Q_l)_\theta$ is another model for $\kappa_{(S,\theta)}^G$. Thus there exists a $G(k)$-equivariant isomorphism $H^{d_U}_c(Y_U,\bar\Q_l)_\theta \to H^{d_V}_c(Y_V,\bar\Q_l)_\theta$. By Theorem \ref{thm:multfree}, the set of all $G(k)$-equivariant morphisms $H^{d_U}_c(Y_U,\bar\Q_l)_\theta \to H^{d_V}_c(Y_V,\bar\Q_l)_\theta$ is a $\bar\Q_l$-vector space of dimension $\#[\kappa_{(S,\theta)}^G]$. Our first observation is that it has a distinguished line.

Let $G \to G'$ be a regular embedding, $S' \subset G'$ the maximal torus containing $S$, and $\theta' : S'(k) \to \bar\Q_l^\times$ an extension of $\theta$. As recalled in Appendix \ref{app:dlreview}, we have a natural isomorphism $H^{d_U}_c(Y^G_U,\bar\Q_l)_\theta \to H^{d_U}_c(Y^{G'}_U,\bar\Q_l)_{\theta'}$ of $\bar\Q_l$-vector spaces that intertwines the inclusion $G(k) \to G'(k)$. In this way, we obtain an action of $G'(k)$ on $H^{d_U}_c(Y_U,\bar\Q_l)_\theta$ that extends the action of $G(k)$. The same is true for $H^{d_V}_c(Y_V,\bar\Q_l)_\theta$. Since $\theta'$ is in general position, these representations of $G'(k)$ are irreducible. Being isomorphic, the set of $G'(k)$-equivariant isomorphisms between them is a 1-dimensional $\bar\Q_l$-vector space.

\begin{dfn} \label{dfn:natiop} A \emph{natural intertwining operator} $H^{d_U}_c(Y_U,\bar\Q_l)_\theta \to H^{d_V}_c(Y_V,\bar\Q_l)_\theta$ is a $G'(k)$-equivariant linear map.
\end{dfn}

\begin{lem} This notion is independent of the choices of $G \to G'$ and $\theta'$.
\end{lem}
\begin{proof}
Another choice of $\theta'$ is of the form $\theta' \cdot \delta$ for a character $\delta : S'(k)/S(k) \to \bar\Q_l^\times$. Let $Y'_U$ be the Deligne-Lusztig variety for $G'$. As reviewed in Appendix \ref{app:dlreview} the action of $G'(k)$ on $H^{d_U}(Y_U,\bar\Q_l)_\theta$ obtained from $\theta' \cdot \delta$ is equal to the twist by $\delta$ of the action obtained from $\theta'$, where now $\delta$ is viewed as a character of $G'(k)/G(k)$. It is now clear that if a linear map $H^{d_U}_c(Y_U,\bar\Q_l)_\theta \to H^{d_V}_c(Y_V,\bar\Q_l)_\theta$ is equivariant for the $G'(k)$-actions obtained from $\theta'$, then it is also equivariant for the $G'(k)$-actions obtained from $\theta'\cdot\delta$. The independence from the choice of $G \to G'$ follows from \cite[Lemma 7.1]{Lus88}.
\end{proof}

\begin{cor} \label{cor:natcomp} The composition of two natural intertwining operators is again a natural intertwining operator.
\end{cor}

\begin{dfn} \label{dfn:normiop}
Assume given a subset $X$ of the set of unipotent radicals of Borel subgroups containing $S$.
\begin{enumerate}
	\item A \emph{collection of natural intertwining operators on $X$} consists of a non-zero natural intertwining operator $\Phi_{V,U} : H^{d_U}_c(Y_U,\bar\Q_l)_\theta \to H^{d_V}_c(Y_V,\bar\Q_l)_\theta$ for any $U,V \in X$, and with the convention $\Phi_{U,U}=\tx{id}$.
	\item The collection $\Phi$ is called \emph{normalized} if $\Phi_{U_3,U_2} \circ \Phi_{U_2,U_1} = \Phi_{U_3,U_1}$ for all $U_1,U_2,U_3 \in X$.
	\item Let $\Gamma$ be a group acting on $G$ by automorphisms and preserving $S$, $\theta$, and $X$. The collection $\Phi$ is called \emph{$\Gamma$-equivariant} if $\gamma \circ \Phi_{U,V} \circ \gamma^{-1}=\Phi_{\gamma(U),\gamma(V)}$ for all $\gamma \in \Gamma$.
	\item The collection $\Phi$ is called \emph{coherent}, if it is $\Gamma$-equivariant, its restriction to any $\Gamma$-orbit is normalized, and for all $U,V \in X$ and $\gamma \in \Gamma$ we have $\Phi_{\gamma(V),\gamma(U)}\circ\Phi_{\gamma(U),U} = \Phi_{\gamma(V),V}\circ\Phi_{V,U}$.
\end{enumerate}
\end{dfn}

We shall now investigate the question whether, for a given $\Gamma$, normalized $\Gamma$-equivariant collections exist. Corollary \ref{cor:natcomp} allows us to measure the failure of a collection $\Phi$ to be normalized by the following function of $(U_1,U_2,U_3) \in X^3$

\begin{equation} \label{eq:etaphi} \eta_\Phi(U_1,U_2,U_3) = \Phi_{U_1,U_2} \circ \Phi_{U_2,U_3} \circ \Phi_{U_1,U_3}^{-1} \in \bar\Q_l^\times. \end{equation}
Any other collection is of the form $\epsilon\Phi$ for a function $\epsilon : X \times X \to \bar\Q_l^\times$ with $\epsilon(U,U)=1$. Here $[\epsilon\Phi]_{U_1,U_2}=\epsilon(U_1,U_2)\Phi_{U_1,U_2}$. The collection $\epsilon\Phi$ is normalized if and only if $\epsilon(U_1,U_2)\epsilon(U_2,U_3)\epsilon(U_1,U_3)^{-1}=\eta_\Phi(U_1,U_2,U_3)^{-1}$. If $\Phi$ is $\Gamma$-equivariant, then $\epsilon\Phi$ is $\Gamma$-equivariant if and only if $\epsilon$ is $\Gamma$-invariant in the sense that $\epsilon(\gamma(U_1),\gamma(U_2))=\epsilon(U_1,U_2)$.

\begin{lem} \label{lem:normcoh} A normalized $\Gamma$-equivariant collection exists if and only if a coherent collection exists.
\end{lem}
\begin{proof}
A normalized $\Gamma$-equivariant collection is automatically coherent. Conversely, assume that $\Phi$ is coherent. Choose arbitrarily a set of representatives $U_0,\dots,U_k$ for $X/\Gamma$. For any $V_1,V_2 \in X$ let $U_1,U_2$ be such that $V_i \in \Gamma \cdot U_i$ and define
\[ \epsilon(V_1,V_2) := \Phi_{V_1,V_2}^{-1}\Phi_{V_1,U_1}\Phi_{U_1,U_0}\Phi_{U_2,U_0}^{-1}\Phi_{U_2,V_2} \in \bar\Q_l^\times. \]
Then one checks easily that $\epsilon$ is $\Gamma$-invariant and satisfies $\eta_\Phi(V_1,V_2,V_3)^{-1}=\epsilon(V_1,V_2)\epsilon(V_2,V_3)\epsilon(V_1,V_3)^{-1}$, so that $\epsilon\Phi$ is normalized and $\Gamma$-equivariant.
\end{proof}

It is thus enough to understand under what circumstances coherent collections exist. For this, we shall assume the existence of a $\Gamma$-equivariant collection and see under what conditions it can be modified to ensure coherence.

\begin{lem} \label{lem:normiop} Let $\Phi$ be a $\Gamma$-equivariant collection on $X$ and let $U \in X$. Assume that the stabilizer of $U$ is normal in $\Gamma$ and let $\bar\Gamma$ be the quotient.
\begin{enumerate}
\item Then the function
\begin{equation} \label{eq:etaphigamma}
\eta_{\Phi,U}(a,b,c)=\eta_\Phi(aU,bU,cU) \in \bar\Q_l^\times
\end{equation}
is a homogeneous $2$-cocycle of $\bar\Gamma$. It's class is independent of $\Phi$.
\item Given $\epsilon_U \in C^1(\bar\Gamma,\bar\Q_l^\times)$ with $\partial\epsilon=\eta_{\Phi,U}^{-1}$ define a collection $\epsilon_U\Phi$ on $Y=\Gamma \cdot U \subset X$ by $[\epsilon_U\Phi]_{{^aU},{^bU}}=\epsilon_U(a,b)\Phi_{{^aU},{^bU}}$. This is a normalized $\Gamma$-equivariant collections on $Y$ and every such collection is given in this way.
\end{enumerate}
\end{lem}
\begin{proof}
Left to the reader.
\end{proof}

In order to treat all $\Gamma$-orbits in $X$ simultaneously we introduce for any two $U,V \in X$ the function
\begin{equation} \label{eq:beta}
\beta_{\Phi,V,U}(a,b) := \Phi_{{^aV},{^aU}}\circ\Phi_{{^aU},{^bU}}\circ\Phi_{{^bV},{^bU}}^{-1}\circ\Phi_{{^aV},{^bV}}^{-1} \in \bar\Q_l^\times.	
\end{equation}
It is an element of $C^1(\Gamma,\bar\Q_l^\times)$. From now on we assume that the stabilizers in $\Gamma$ of all $U \in X$ are equal. In particular, $\beta_{\Phi,V,U}$ is inflated from $C^1(\bar\Gamma,\bar\Q_l^\times)$.
\begin{lem} \label{lem:beta}
Given $U,V,W \in X$ we have
\begin{enumerate}
\item $\beta_{\Phi,U,U}=1$, $\beta_{\Phi,U,V}=\beta_{\Phi,V,U}^{-1}$, and $\beta_{\Phi,U,W} \cdot\beta_{\Phi,W,V}\cdot\beta_{\Phi,V,U}  = 1$.
\item $\partial\beta_{\Phi,V,U} = \eta_{\Phi,U} \cdot \eta_{\Phi,V}^{-1}$.
\item If $V ={^xU}$ then $\beta_{\Phi,V,U}=\eta_{\Phi,U}(ax,a,b)\eta_{\Phi,U}(ax,bx,b)^{-1}$.
\end{enumerate}
\end{lem}
\begin{proof} Left to the reader.
\end{proof}

\begin{cor} \label{cor:defeta}
The cohomology class $[\eta]$ of $\eta_{\Phi,U}$ depends neither on $\Phi$, nor on $U$.
\end{cor}

\begin{dfn} \label{dfn:cohspl}
A collection $\{\epsilon_U \in C^1(\bar\Gamma,\bar\Q_l^\times)|U \in X\}$ is called a \emph{coherent splitting of $\{\eta_{\Phi,U}|U \in X\}$} if for any $U \in X$ we have $\partial\epsilon_U = \eta_{\Phi,U}^{-1}$ and for any $U,V \in X$ we have $\epsilon_V=\beta_{\Phi,V,U}\cdot\epsilon_U$.
\end{dfn}

Given a coherent splitting $\epsilon=\{\epsilon_U\}$ we define the collection $\epsilon\Phi$ as follows: For each $\Gamma$-orbit $Y \subset X$ choose arbitrarily an element $U \in Y$ and define $\epsilon\Phi|_Y$ by $\epsilon_U\Phi$ as in Lemma \ref{lem:normiop}. By Lemma \ref{lem:beta} the resulting $\epsilon\Phi|_Y$ is independent of the choice of $U$. For any $U,V \in X$ belonging to different $\Gamma$-orbits, define $[\epsilon\Phi]_{U,V}=\Phi_{U,V}$. It is immediate to check that $\epsilon\Phi$ is a coherent collection.

\begin{cor} \label{cor:normiop} Assume that a $\Gamma$-equivariant collection on $X$ exists and that the stabilizers in $\Gamma$ of all $U \in X$ are equal. The following are equivalent.
\begin{enumerate}
	\item The class $[\eta]$ is trivial.
	\item A $\Gamma$-equivariant normalized collection of natural intertwining operators exists on some $\Gamma$-orbit of $X$.
	\item A coherent splitting of $\{\eta_{\Phi,U}|U \in X\}$ exists.
    \item A coherent collection of natural intertwining operators on $X$ exists.
    \item A normalized $\Gamma$-equivariant collection of natural intertwining operators on $X$ exists.
\end{enumerate}
\end{cor}
\begin{proof} Left to the reader.
\end{proof}

\begin{fct} \label{fct:cohsplituniq}	
If $\{\epsilon_U| U \in X\}$ is a coherent splitting of $[\eta]$, then any other such is of the form $\{\epsilon_U \cdot \delta| U \in X\}$ for $\delta \in Z^1(\bar\Gamma,\bar\Q_l^\times)=\tx{Hom}(\bar\Gamma,\bar\Q_l^\times)$. In this way, the set of coherent splittings is a torsor for $\tx{Hom}(\bar\Gamma,\bar\Q_l^\times)$.
\end{fct}

\subsection{Geometric intertwining operators} \label{sub:gio}

We continue to assume $G=G^0$. In order to apply Lemma \ref{lem:normiop} or Corollary \ref{cor:normiop}, one needs to start with a $\Gamma$-equivariant collection of natural intertwining operators. In this subsection we shall review work of Bonnaf\'e-Dat-Rouquier \cite{BDR17} that provides a canonical such collection. More precisely, for any two $U,V$ the authors construct a non-zero intertwining operator $\Psi_{V,U} : H^{d_U}_c(Y_U,\bar\Q_l)_\theta \to H^{d_V}_c(Y_V,\bar\Q_l)_\theta$ that is naturally given, i.e. independent of any additional choices. It is very easy to see, and we shall do so below, that the resulting collection $\Psi$ is equivariant with respect to the full group of automorphisms of $G$ that preserve $S$ and $\theta$ and consists of natural intertwining operators in the sense of Definition \ref{dfn:natiop}.

The construction of $\Psi_{V,U}$ is as follows. Given $U,V$, Bonnaf\'e-Dat-Rouquier define \cite[\S6.A]{BDR17} the variety
\[ Y_{U,V} = \{(g,h) \in G/U \times G/V| g^{-1}h \in U\cdot V, h^{-1}F(g) \in V\cdot F(U) \}\]
and the closed subvariety
\[ Y_{U,V}^{(2)} = \{(g,h) \in Y_{U,V}| g^{-1}F(g) \in U \cdot F(U) \}. \]
Just like $Y_U$, these varieties are equipped with an action of $G(k)$ by left multiplication and of $S(k)$ by right multiplication.
It is shown in \cite[Lemma 6.1]{BDR17} that the map $Y_{U,V}^{(2)} \to Y_U$ given by $(g,h) \to g$ is a smooth affine fiber bundle of dimension equal to the codimension of $Y_{U,V}^{(2)}$ in $Y_{U,V}$. This dimension is $d_{U,V}=\tx{dim}(U \cap F(U)) - \tx{dim}(U \cap V \cap F(U))$. Pulling back along the inclusion $Y_{U,V}^{(2)} \to Y_{U,V}$ gives a map $H_c^i(Y_{U,V},\bar\Q_l) \to H_c^i(Y_{U,V}^{(2)},\bar\Q_l)$. Pushing forward along the smooth morphism $Y_{U,V}^{(2)} \to Y_U$ gives an isomorphism $H_c^i(Y_{U,V}^{(2)},\bar\Q_l) \to H_c^{i-d_{U,V}}(Y_U,\bar\Q_l)$. Both of these are equivariant for the actions of $G(k)$ and $S(k)$. One of the main results, \cite[Theorem 6.2]{BDR17}, is that the composed map 
\begin{equation} \label{eq:bdriso0}
H^*_c(Y_{U,V}^{G^0},\bar\Q_l)_\theta \to H^{*-d_{U,V}}_c(Y_U^{G^0},\bar\Q_l)_\theta	
\end{equation}
is an isomorphism. This is combined with the observation that the map $Y_{U,V} \to Y_{V,FU}$ given by $(g,h) \mapsto (h,F(g))$ induces an isomorphism of \'etale sites, and hence also an isomorphism $H_c^i(Y_{U,V},\bar\Q_l) \to H_c^i(Y_{V,FU},\bar\Q_l)$, again equivariant for $G(k)$ and $S(k)$. Composing the three isomorphisms $H_c^i(Y_{U,V},\bar\Q_l)_\theta \to H_c^{i-d_{U,V}}(Y_U,\bar\Q_l)_\theta$, $H_c^i(Y_{U,V},\bar\Q_l) \to H_c^i(Y_{V,FU},\bar\Q_l)$, and the inverse of $H_c^i(Y_{V,FU},\bar\Q_l)_\theta \to H_c^{i-d_{V,FU}}(Y_V,\bar\Q_l)_\theta$, leads to the isomorphism
\begin{equation} \label{eq:defpsi} \Psi_{V,U} : H_c^{d_U}(Y_U,\bar\Q_l)_\theta \to H_c^{d_V}(Y_V,\bar\Q_l)_\theta. \end{equation}
We find it convenient to remember the definition of this isomorphism in terms of the following symbolic diagram
\[ Y_U \leftarrow Y_{U,V}^{(2)} \rightarrow Y_{U,V} \leftrightarrow Y_{V,F(U)} \leftarrow Y_{V,F(U)}^{(2)} \rightarrow Y_{V}. \]
Note that this isomorphism also depends on $G$ and $\theta$, so when there is a danger of confusion we will write $\Psi^G_{U,V}$, or $\Psi^\theta_{U,V}$, to record this dependence.

\begin{lem} \label{lem:iopindep} Let $\tilde G \to G$ be a morphism with abelian kernel and cokernel. The diagram
\[ \xymatrix{
H^{d_U}_c(Y^{\tilde G}_U,\bar\Q_l)_{\tilde\theta}\ar[d]\ar[r]^{\Psi^{\tilde G}_{V,U}}&H^{d_V}_c(Y^{\tilde G}_V,\bar\Q_l)_{\tilde\theta}\ar[d]\\
H^{d_U}_c(Y^{G}_U,\bar\Q_l)_{\theta}\ar[r]^{\Psi^{G}_{V,U}}&H^{d_V}_c(Y^{G}_V,\bar\Q_l)_{\theta}}
\]
commutes, where the vertical arrows are the isomorphisms reviewed in Appendix \ref{app:dlreview}.
\end{lem}
\begin{proof}
This follows from the commutativity of the following diagram, in which all vertical arrows are induced by $\tilde G \to G$.
\[ \xymatrix{
Y^{\tilde G}_U\ar[d]&\ar[l]Y_{U,V}^{\tilde G,(2)}\ar@{^(->}[r]\ar[d]&Y^{\tilde G}_{U,V}\ar[d]\ar[r]&Y^{\tilde G}_{V,FU}\ar[d]\ar[l]&Y_{V,FU}^{\tilde G,(2)}\ar[d]\ar@{_(->}[l]\ar[r]&Y^{\tilde G}_{V}\ar[d]\\
Y^{G}_U&\ar[l]Y_{U,V}^{G,(2)}\ar@{^(->}[r]&Y^{G}_{U,V}\ar[r]&Y^{G}_{V,FU}\ar[l]&Y_{V,FU}^{G,(2)}\ar@{_(->}[l]\ar[r]&Y^{G}_{V}
}
\]	
\end{proof}

\begin{cor} \label{cor:iopnat} The operator $\Psi_{V,U}$ is natural in the sense of Definition \ref{dfn:natiop}.
\end{cor}

Let $\alpha$ be an automorphism of $G$ commuting with $F$ and preserving $S$. Then $\alpha : G \to G$ restricts to an isomorphism $\alpha : Y_U \to Y_{\alpha(U)}$ of varieties, and on cohomology we obtain the isomorphism $\alpha : H^i_c(Y_U,\bar\Q_l) \to H^i_c(Y_{\alpha(U)},\bar\Q_l)$ satisfying $\alpha(v)\alpha(s)=\theta(s)\alpha(v)$ for $v \in H^i_c(Y_U,\bar\Q_l)$ and $s \in S(k)$. Noting that $d_{\alpha(U)}=d_U$, we obtain $\alpha : H^{d_U}_c(Y_U,\bar\Q_l)_{\theta} \to H^{d_{\alpha(U)}}_c(Y_{\alpha(U)},\bar\Q_l)_{\theta\circ\alpha^{-1}}$.

\begin{lem}  \label{lem:iopconj1} For any $U,V$ we have the equality
\[ \alpha \circ \Psi_{V,U}^\theta \circ \alpha^{-1}  = \Psi_{\alpha(V),\alpha(U)}^{\theta\circ\alpha^{-1}}. \]
\end{lem}
\begin{proof} This follows from the commutativity of the following diagram
\[ \xymatrix@C-1pc{
	Y_U\ar[d]^{\alpha}&\ar[l]Y_{U,V}^{(2)}\ar@{^(->}[r]\ar[d]^{(\alpha,\alpha)}&Y_{U,V}\ar[d]^{(\alpha,\alpha)}\ar[r]&Y_{V,FU}\ar[d]^{(\alpha,\alpha)}\ar[l]&Y_{V,FU}^{(2)}\ar[d]^{(\alpha,\alpha)}\ar@{_(->}[l]\ar[r]&Y_{V}\ar[d]^{\alpha}\\
	Y_{\alpha(U)}&\ar[l]Y_{\alpha(U),\alpha(V)}^{(2)}\ar@{^(->}[r]&Y_{\alpha(U),\alpha(V)}\ar[r]&Y_{\alpha(V),F\alpha(U)}\ar[l]&Y_{\alpha(V),F\alpha(U)}^{(2)}\ar@{_(->}[l]\ar[r]&Y_{\alpha(V)}\\
}
\]
where $(\alpha,\alpha)$ sends $(g,h)$ to $(\alpha(g),\alpha(h))$.
\end{proof}

We consider the function $\eta_\Psi(U_1,U_2,U_3)$ of \eqref{eq:etaphi}.
By Lemma \ref{lem:iopindep} this function depends only on the simply connected cover of the derived subgroup of $G$. For the study if $\eta_\Psi$ we may therefore assume that $G$ is simply connected. Then it is the product of $k$-simple factors.

\begin{lem} \label{lem:ioprod} Let $G=G_1 \times G_2$, $U=U_1 \times U_2$, $V=V_1 \times V_2$, $S=S_1 \times S_2$, $\theta=\theta_1 \otimes \theta_2$. Then
\begin{enumerate}
	\item The isomorphism $G_1 \times G_2 \to G$ restricts to an isomorphism $Y_{U_1}^{G_1} \times Y_{U_2}^{G_2} \to Y_U^G$.
	\item This induces an isomorphism 
	\[ H_c^{d_{U_1}}(Y_{U_1}^{G_1},\bar Q_l)_{\theta_1} \otimes H_c^{d_{U_1}}(Y_{U_2}^{G_2},\bar Q_l)_{\theta_2} \to H_c^{d_U}(Y_U^G,\bar\Q_l)_\theta. \]
	\item The latter isomorphism identifies $\Psi_{V_1,U_1}^{G_1} \otimes \Psi_{V_2,U_2}^{G_2}$ with $\Psi_{V,U}^G$.
\end{enumerate}
\end{lem}
\begin{proof}
The first claim is immediate from the definitions. The second follows from the K\"unneth formula and the vanishing theorem \cite[Cor 9.9]{DL76}. The third follows from the fact that the analogous decomposition as for $Y_U^G$ also holds for $Y^G_{V,U}$ and $Y^{G,(2)}_{V,U}$, as well as the maps between them.
\end{proof}

\begin{fct} Let $\Gamma_1$ and $\Gamma_2$ be two groups and $M$ a trivial $\Gamma_1 \times \Gamma_2$-module. Let $z \in Z^2(\Gamma_1 \times \Gamma_2,M)$ have the property $z((\gamma_1,1),(1,\gamma_2))=0$ and $z((1,\gamma_2),(\gamma_1,1))=0$ for all $\gamma_1 \in \Gamma_1$ and $\gamma_2 \in \Gamma_2$. Then
\[ z((\gamma_1,\gamma_2),(\gamma_1',\gamma_2'))=z((\gamma_1,1),(\gamma_1',1))+z((1,\gamma_2),(1,\gamma_2')).\]
In other words
\[ z = \tx{Inf}^{\Gamma_1\times\Gamma_2}_{\Gamma_1}\tx{Res}^{\Gamma_1 \times \Gamma_2}_{\Gamma_1}z + \tx{Inf}^{\Gamma_1\times\Gamma_2}_{\Gamma_2}\tx{Res}^{\Gamma_1 \times \Gamma_2}_{\Gamma_2}z. \]
\end{fct}

\begin{cor} \label{cor:etaprod}
Let $G=G_1 \times \dots \times G_k$, $S=S_1 \times \dots \times S_k$, $\Gamma=\Gamma_1 \times \dots \times \Gamma_k$, $U=U_1 \times \dots \times U_k$. Then
\[ \eta_{\Psi,U}^G = \sum_{i=1}^k \tx{Inf}_{\Gamma_i}^\Gamma \eta_{\Psi,U_i}^{G_i}. \]
\end{cor}
\begin{proof}
By induction we may assume $k=2$. Applying above Fact it is enough to show that when $U=U_1 \times U_2$, $U'=V_1 \times U_2$, and $U''=V_1 \times V_2$, we have $\eta_\Psi(U,U',U'')=1$. This follows from Lemma \ref{lem:ioprod} by observing $\Psi^{G_1}_{V_1,V_1}=\tx{id}$ and $\Psi^{G_2}_{U_2,U_2}=\tx{id}$.
\end{proof}

Let $\phi : G(k) \to \bar\Q_l^\times$ be a character whose restriction to $G_\tx{sc}(k)$ is trivial. Recall from \cite[Corollary 1.27]{DL76} that there is a naturally given isomorphism of $G(k)$-representations $\phi\otimes H^i_c(Y_U,\bar\Q_l)_\theta \to H^i_c(Y_U,\bar\Q_l)_{\phi\cdot\theta}$, as follows. The action of $G(k) \times S(k)$ on $Y_U$ by the formula $x \mapsto gxt$ extends to an action of $[(G \times S)/Z](k)$, where $Z$ is the center of $G$ embedded anti-diagonally in $G \times S$. Let $Y^\tx{sc}_U$ denote the Deligne-Lusztig variety for $G_\tx{sc}$. The natural map $G_\tx{sc} \to G$ induces an etale map $Y^\tx{sc}_U \to Y_U$ which realizes the $[(G \times S)/Z](k)$-representation $H^i_c(Y_U,\bar\Q_l)$ as the induction of the $[(G_\tx{sc} \times S_\tx{sc})/Z_\tx{sc}](k)$-representation $H^i_c(Y^\tx{sc}_U,\bar\Q_l)$.
We now use the basic fact that if $\eta : A \to B$ is a homomorphism of finite groups, $\phi : B \to \bar\Q_l^\times$ is a character, and $(\rho,V)$ is a representation of $A$, then multiplication by $\phi^{-1}$ is an automorphism of the vector space $\{f : B \to V|f(\eta(a)b)=a f(b)\}$ that is an isomorphism
\[ \tx{Ind}_A^B ((\phi\circ\eta)\otimes\rho) \to \phi \otimes \tx{Ind}_B^A\rho. \]
Note that this automorphism is functorial in $V$. We apply this basic fact to the induced representation $H^i_c(Y_U,\bar\Q_l)$ and the character $(g,s) \mapsto \phi(gs)$ of $[(G \times S)/Z](k)$. It gives us an automorphism of the $\bar\Q_l$-vector space $H^i_c(Y_U,\bar\Q_l)$ whose restriction to each $\theta$-isotypic component realizes the isomorphism $\phi\otimes H^i_c(Y_U,\bar\Q_l)_\theta \to H^i_c(Y_U,\bar\Q_l)_{\phi\cdot\theta}$.

\begin{lem} \label{lem:iopcharshift} For any two $U,V$ the following diagram commutes.
\[ \xymatrix{
\phi\otimes H^{d_U}_c(Y_U,\bar\Q_l)_\theta\ar[d]\ar[r]^{\Psi_{V,U}^\theta}&\phi\otimes H^{d_V}_c(Y_V,\bar\Q_l)_\theta\ar[d]\\
H^{d_U}_c(Y_U,\bar\Q_l)_{\phi\cdot\theta}\ar[r]^{\Psi_{V,U}^{\phi\cdot\theta}}&H^{d_V}_c(Y_V,\bar\Q_l)_{\phi\cdot\theta}
}\]
Given an automorphism $\alpha$ of $G$ commuting with $F$ and preserving $S$ we obtain a morphism $\alpha : Y_U \to Y_{\alpha(U)}$ and hence a linear map $H^i_c(Y_U,\bar\Q_l) \to H^i_c(Y_{\alpha(U)},\bar\Q_l)$, which we shall also denote by $\alpha$. Then the following diagram of $G(k)$-representations commutes
\[ \xymatrix{
(\phi\otimes H^i_c(Y_U,\bar\Q_l)_\theta)^\alpha\ar[d]\ar[rr]^-{\tx{id}\otimes\alpha}&&\phi\otimes H^i_c(Y_{\alpha(U)},\bar\Q_l)_{\theta\circ\alpha^{-1}}\ar[d]\\
(H^i_c(Y_U,\bar\Q_l)_{\phi\cdot\theta})^\alpha\ar[rr]^-{\alpha}&&H^i_c(Y_{\alpha(U)},\bar\Q_l)_{(\phi\cdot\theta)\circ\alpha^{-1}}
}\]
where for a $G(k)$-representation $\pi$ we write $\pi^\alpha(g)=\pi(\alpha^{-1}(g))$.
\end{lem}
\begin{proof}
Consider the first diagram. An argument as in the proof of \cite[Proposition 1.25]{DL76} shows that the $[(G \times S)/Z](k)$-spaces $Y_{U,V}$ and $Y_{U,V}^{(2)}$ are induced from the $[(G_\tx{sc} \times S_\tx{sc})/Z_\tx{sc}](k)$-spaces $Y_{U,V}^\tx{sc}$ and $Y_{U,V}^{\tx{sc},(2)}$. In addition, the maps $H^i_c(Y_{U,V},\bar\Q_l) \to H^i_c(Y_{U,V}^{(2)},\bar\Q_l) \to H^{i-d_{U,V}}_c(Y_U,\bar\Q_l)$ as well as $H^i_c(Y_{U,V},\bar\Q_l) \to H^{i}_c(Y_{V,FU},\bar\Q_l)$ are induced from the corresponding maps for $Y^\tx{sc}$.
The claim now follows from the functoriality in $V$ of the automorphism $f \mapsto \phi\cdot f$, in the abstract set-up explained above.

Now consider the second diagram. The two horizontal maps are induced from the corresponding maps for $Y^\tx{sc}$, due to the  commutativity of
\[ \xymatrix{
	Y^\tx{sc}_U\ar[d]\ar[r]^\alpha&Y^\tx{sc}_{\alpha(U)}\ar[d]\\
	Y_U\ar[r]^\alpha&Y_{\alpha(U)}
}
\]
and the claim follows again from the functoriality of $f \mapsto \phi\cdot f$.
\end{proof}

\begin{cor} \label{cor:etacharshift} The 2-cocycle $\eta_{\Psi,U}$ does not change if we replace $\theta$ by $\phi\cdot\theta$.
\end{cor}

\subsection{Induction and intertwining for disconnected groups} \label{sub:ff-disc}

We now drop the assumption $G=G^0$ that was in force for the few preceding subsections. The generalization of the Deligne-Lusztig varieties and linking varieties to this case is straightforward. In fact, the definitions are exactly the same as in the connected case. We begin with the Deligne-Lusztig variety, following \cite[\S2.D,2.E,3.A]{BDR17}. For this, let $B^0$ be a Borel $\bar k$-subgroup of $G^0$ containing $S^0$ and let $U$ be its unipotent radical. We keep in mind that $U \subset G^0 \subset G$. Define
\[ Y_U^G  = \{g \in G/U|\, g^{-1}F(g) \in U \cdot FU\}. \]
We have added the superscript $G$ because we will have to compare these varieties for different groups.

\begin{lem} \label{lem:dl_ind1}
Let $G^0 \subset G' \subset G$ be an intermediate group. The multiplication map on $G$ induces an isomorphism
\begin{equation} \label{eq:dl_ind1c}
G^F \times_{(G')^F} Y_{U}^{G'} \to Y_U^G	
\end{equation}
If the natural map $(S/S')(k) \to (G/G')(k)$ is an isomorphism, then the multiplication map on $G$ induces an isomorphism
\begin{equation} \label{eq:dl_ind1d}
Y_U^{G'} \times_{(S')^F} S^F \to Y_U^G.
\end{equation}
\end{lem}
\begin{proof}
Left to the reader.
\end{proof}

As in the connected case we define the Lusztig functor
\[ \mc{R}_{S,U}^G : \mc{D}^b(S) \to \mc{D}^b(G),\qquad M \mapsto \Gamma_c(Y_U^G,\bar\Q_l) \otimes^\mb{L}_{S(k)} M \]
between the equivariant derived categories for $S(k)$ and $G(k)$.

\begin{cor} \label{cor:dl_ind2}
Let $G^0 \subset G' \subset G$ be an intermediate group.
\begin{equation} \label{eq:dl_ind1a}
\mc{R}_{S,U}^G \circ \tx{Ind}_{(S')^F}^{S^F} \cong \mc{R}_{S',U}^{G} \cong \tx{Ind}_{(G')^F}^{G^F} \circ \mc{R}_{S',U}^{G'}.
\end{equation}
If the inclusion $S \to G$ induces an isomorphism $S^F/(S')^F \to G^F/(G')^F$, then we have
\begin{equation} \label{eq:dl_ind1b}
\tx{Res}_{(G')^F}^{G^F} \circ \mc{R}_{S,U}^G \cong \mc{R}_{S',U}^{G'}\circ \tx{Res}_{(S')^F}^{S^F}.	
\end{equation}
\end{cor}
\begin{proof}
Immediate from Lemma \ref{lem:dl_ind1}.
\end{proof}

Recall the character $\theta : S(k) \to \bar\Q_l^\times$ with non-singular restriction $\theta^0$ to $S^0(k)$.

\begin{cor} If $i \neq d_U$ then $\mc{R}_{S,U}^{i,G}(\theta) = 0$.
\end{cor}
\begin{proof}
The $G^F$-module $\mc{R}_{S,U}^{i,G}(\theta)$ is a direct summand of 
\[ \mc{R}_{S,U}^{i,G}(\tx{Ind}_{(S^0)^F}^{S^F}\tx{Res}_{(S^0)^F}^{S^F}\theta),\] 
which by \eqref{eq:dl_ind1a} equals $\mc{R}_{S^0,U}^{i,G^0}(\theta^0)$. The latter vanishes unless $i=d_U$ by \cite[Corollary 9.9]{DL76} and \cite{He08}.
\end{proof}

\begin{cor} \label{cor:dl_ind4}
We have the natural isomorphisms 
\begin{equation} \label{eq:ind2a}
\tx{Res}_{(G^0)^F}^{(G^0\cdot S)^F}\mc{R}_{S,U}^{d_U,G^0 \cdot S}(\theta) \to \mc{R}_{S^0,U}^{d_U,G^0}(\theta^0)	
\end{equation}
and 
\begin{equation} \label{eq:ind2b}
\mc{R}_{S,U}^{d_U,G}(\theta) \to \tx{Ind}_{(G^0 \cdot S)^F}^{G^F}\mc{R}_{S,U}^{d_U,G^0 \cdot S}(\theta).	
\end{equation}
\end{cor}
\begin{proof}
The first claim is immediate from \eqref{eq:dl_ind1b} and Lang's theorem, while the second is immediate from \eqref{eq:dl_ind1a}.
\end{proof}

\begin{rem} \label{rem:kappaext}
The above Corollary states that the $G^F$-module $\mc{R}_{S,U}^{d_U,G}(\theta)$ is obtained by endowing $\mc{R}_{S^0,U}^{G^0}(\theta^0)$ with a natural $(G')^F$-structure, depending on $\theta$, and then inducing to $G^F$. The natural $(G')^F$-structure can be made explicit as follows. Lang's theorem implies $(G')^F=(S')^F \cdot (G^0)^F$, so it is enough to endow $\mc{R}_{S^0,U}^{G^0}(\theta^0)$ with an $S^F$-structure. Now $S^F$ acts by conjugation on $G^0/U$ and this action preserves the subvariety $Y_U^{G^0}$. Therefore we obtain an action on $H_c^{d_U}(Y_U^{G^0},\bar\Q_l)$. This action commutes with the right action of $(S^0)^F$ and hence preserves the $\theta^0$-isotypic component for that action, i.e. we obtain an action of $S^F$ on $\mc{R}_{S^0,U}^{d_U,G^0}(\theta^0)$. Together with the action of $(G^0)^F$ we obtain an action of $(G^0)^F \rtimes S^F$ on $\mc{R}_{S^0,U}^{d_U,G^0}(\theta^0)$. The restriction of this action along the embedding $(S^0)^F \to (G^0)^F \rtimes S^F$ sending $s$ to $s \rtimes s^{-1}$ is equal to multiplication by $\theta^0(s)$. Therefore, upon twisting the action of $S^F$ by the linear character $\theta^{-1}$ we obtain an action of $(G^0)^F \rtimes S^F$ that restricts trivially along the anti-diagonal embedding of $(S^0)^F$ and hence descends to an action of $(G')^F=(G^0)^F \cdot S^F$. It is easy to see, using the fact that \eqref{eq:dl_ind1c} and \eqref{eq:dl_ind1d} become
\[ S^F \times_{(S^0)^F} Y_U^{G^0} = Y_U^{G'} = Y_U^{G^0} \times_{(S^0)^F} S^F, \]
that this does indeed describe the action of $(G')^F$ on $\mc{R}_{S^0,U}^{d_U,G^0}(\theta^0)$ transported via \eqref{eq:ind2a}.
\end{rem}

We will now extend the discussion of linking varieties and geometric intertwining operators to the disconnected group $G$. As with $Y_U^G$, the definitions of the linking varieties are completely analogous to those in the connected case:
\[ Y_{U,V}^G = \{(gU,hV) \in G/U \times G/V|\, g^{-1}h \in U \cdot V, h^{-1}F(g) \in V \cdot F(U) \}\]
and
\[ Y_{U,V}^{(2),G} = \{(gU,hV) \in Y_{U,V}^{G}|\, g^{-1}F(g) \in U \cdot F(U)\}. \]

\begin{lem} \label{lem:dl_ind3}
Given an intermediate group $G^0 \subset G' \subset G$, the multiplication map on $G$ induces isomorphisms of varieties
\begin{equation} \label{eq:dl_ind3a}
G^F \times_{(G')^F} Y_{U,V}^{G'} \to Y_{U,V}^G,\qquad G^F \times_{(G')^F} Y_{U,V}^{(2),G'} \to Y_{U,V}^{(2),G}.
\end{equation}
\end{lem} 
\begin{proof}
The proof is the same as for Lemma \ref{lem:dl_ind1} and is left to the reader.
\end{proof}

The maps
\[ Y_{U,V}^G \leftarrow Y_{U,V}^{(2),G} \rightarrow Y_U^{G}\]
are obtained via induction (cf. \cite[\S1.24]{DL76}) along $(G^0)^F \to G^F$ from the maps
\[ Y_{U,V}^{G^0} \leftarrow Y_{U,V}^{(2),G^0} \rightarrow Y_U^{G^0}.\]
This implies that $Y_{U,V}^G \to Y_U^G$ is again a smooth affine fiber bundle of dimension $d_{U,V}$, hence induces a push-forward functor on compactly supported cohomology, which combined with pulling back along the closed immersion $Y_{U,V}^{(2),G} \to Y_{U,V}^G$ induces a morphism $\Gamma_c(Y_{U,V},\bar\Q_l)\to \Gamma_c(Y_U,\bar\Q_l)[d_{U,V}]$ that is $G^F \times S^F$-equivariant.

\begin{lem} \label{lem:bdriso1}
The morphism
\[ \Gamma_c(Y_{U,V}^{G},\bar\Q_l)\otimes_{S^F}^\mb{L}\theta \to \Gamma_c(Y_U^{G},\bar\Q_l)[d_{U,V}] \otimes_{S^F}^\mb{L}\theta\]
is an isomorphism.	
\end{lem}
\begin{proof}
We consider first the morphism 
\[ \Gamma_c(Y_{U,V}^{G},\bar\Q_l)\otimes_{(S^0)^F}^\mb{L}\theta^0 \to \Gamma_c(Y_U^{G},\bar\Q_l)[d_{U,V}] \otimes_{(S^0)^F}^\mb{L}\theta^0.\]
According to Lemmas \ref{lem:dl_ind1} and \ref{lem:dl_ind3} this morphism is given by applying $\tx{Ind}_{(G^0)^F}^{G^F}$ to the morphism \eqref{eq:bdriso0}, and hence is an isomorphism. At the same time the above morphism is stil equivariant for the right-action of $S^F$ and hence restricts between an isomorphism of the $\theta$-isotypic components.
\end{proof}

Lemma \ref{lem:bdriso1} allows us to obtain a $G(k)$-equivariant isomorphism
\begin{equation} \label{eq:defpsi1} \Psi_{V,U}^{G} : H_c^{d_U}(Y_U^{G},\bar\Q_l)_\theta \to H_c^{d_V}(Y_V^{G},\bar\Q_l)_\theta. \end{equation}
just as in the case with $G^0$.

\begin{dfn} \label{dfn:kappast}
Write $\kappa_{(S,\theta)}^G$ for the isomorphism class of the representation $H_c^{d_U}(Y_U^G,\bar\Q_l)_\theta$. We refer to the set $[\kappa_{(S,\theta)}^G]$ of irreducible constituents of $\kappa_{(S,\theta)}^G$ as a \emph{Deligne-Lusztig packet}. 
\end{dfn}

\begin{fct} \label{fct:ftwist}
Let $f$ be an automorphism of $G$. Then $\kappa^{G}_{(S,\theta)} \circ \tx{Ad}(f)^{-1} = \kappa^{G}_{(f(S),\theta\circ f^{-1})}$. 
\end{fct}
\begin{proof}
The restriction of $f$ to $Y_U^G$ is an isomorphism of varieties $Y_U^G \to Y_{f(U)}^G$ that translates the $(G^F \times S^F)$-action on $Y_U^G$ to the twist by $f$ of the $(G^F \times S^F)$-action on $Y_{f(U)}^G$.
\end{proof}

\begin{lem} \label{lem:dl_kappatwist}
For a character $\phi : G(k) \to \bar\Q_l^\times$ trivial on $G^0_\tx{sc}(k)$ we have
\[ \kappa_{(S,\phi\cdot\theta)}^{G} = \phi \otimes\kappa_{(S,\theta)}^{G}. \]
\end{lem}
\begin{proof}
As discussed just before Lemma \ref{lem:iopcharshift} we have the isomorphism $\kappa_{(S^0,\phi^0\cdot\theta^0)}^{G^0} = \phi^0 \otimes\kappa_{(S^0,\theta^0)}^{G^0}$. Recall from Remark \ref{rem:kappaext} that $\kappa_{(S,\theta)}^{G'}$ is an extension to $G'(k)=G^0(k) \cdot S(k)$ of $\kappa_{(S^0,\theta^0)}^{G^0}$. The second diagram of Lemma \ref{lem:iopcharshift} applied to $\alpha=\tx{Ad}(s)$ for $s \in S(k)$ implies that the isomorphism $\kappa_{(S^0,\phi^0\cdot\theta^0)}^{G^0} = \phi^0 \otimes\kappa_{(S^0,\theta^0)}^{G^0}$ is $S(k)$-equivariant with respect to the extension to $S(k)$ of both sides, thus giving the isomorphism $\kappa_{(S,\phi\cdot\theta)}^{G'} = \phi \otimes\kappa_{(S,\theta)}^{G'}$, where $G'=G^0 \cdot S$. Finally $\kappa_{(S,\phi\cdot\theta)}^{G} = \phi \otimes\kappa_{(S,\theta)}^{G}$ follows formally from Corollary \ref{cor:dl_ind4}.
\end{proof}

\begin{pro} \label{pro:inddisj}
The representations $\kappa_{(S_1,\theta_1)}^G$ and $\kappa_{(S_2,\theta_2)}^G$ are either equal or disjoint. They are equal if and only if $(S_1,\theta_1)$ is $G(k)$-conjugate to $(S_2,\theta_2)$.
\end{pro}
\begin{proof}
If $(S_1,\theta_1)$ is $G(k)$-conjugate to $(S_2,\theta_2)$ then Fact \ref{fct:ftwist} implies that $\kappa_{(S_1,\theta_1)}^G$ and $\kappa_{(S_2,\theta_2)}^G$ are equal. Conversely assume that $\kappa_{(S_1,\theta_1)}^G$ and $\kappa_{(S_2,\theta_2)}^G$ share a common constituent. Set $G'=G^0 \cdot S$. By Corollary \ref{cor:dl_ind4} and the Mackey formula there exists $g \in G(k)$ so that $\kappa_{(S_1,\theta_1)}^{G'} \circ \tx{Ad}(g)$ and $\kappa_{(S_2,\theta_2)}^{G'}$ share a common constituent. By Fact \ref{fct:ftwist} we have $\kappa_{(S_1,\theta_1)}^{G'} \circ \tx{Ad}(g)=\kappa_{(S_1^g,\theta_1^g)}^{G'}$. Applying $\tx{Res}^{G'}_{G^0}$ we see that $\kappa_{(S_1^{0,g},\theta_1^{0,g})}^{G^0}$ and $\kappa_{(S_2^0,\theta_2^0)}^{G^0}$ share a common constituent. By \cite[Theorem 6.8]{DL76} there exists $h \in G^0(k)$ so that $(S_1^{0,gh},\theta_1^{0,gh})=(S_2^0,\theta_2^0)$. The first part of the proof allows us to replace $(S_1,\theta_1)$ by $(S_1^{gh},\theta_1^{gh})$ and assume $S_1^0=S_2^0$ and $\theta_1^0=\theta_2^0$. Then $S_1=S_1^0\cdot Z=S_2^0\cdot Z=S_2$ and we will write $S=S_1=S_2$. The character $\delta=\theta_2\theta_1^{-1}$ factors through $\pi_0(S)(k)$. Write $\theta=\theta_1$ so that $\theta_2=\theta\delta$. Using Lemma \ref{lem:dl_kappatwist} we see that $\kappa_{(S,\theta\delta)}^{G'}=\kappa_{(S,\theta)}^{G'}\otimes\delta$. 

The non-disjointness of $\kappa_{(S,\theta)}^{G'}\otimes\delta$ and $\kappa_{(S,\theta)}^{G'}$ implies that there exist two irreducible constituents $\pi_1,\pi_2$ of $\kappa_{(S,\theta)}^{G'}$ such that $\pi_1 \cong \pi_2\otimes\delta$. This implies $\tx{Res}^{G'}_{G^0}(\pi_1) \cong \tx{Res}^{G'}_{G^0}(\pi_2)$. Corollary \ref{cor:dl_ind4} states that $\kappa_{(S_1,\theta_1)}^{G'}$ is an extension to $G'(k)$ of the representation $\kappa_{(S^0,\theta^0)}^{G^0}$, and Theorem \ref{thm:multfree} states that the latter is multiplicity\-free. Therefore $\pi_1 \cong \pi_2$. Write $\pi=\pi_1=\pi_2$. 

Theorem \ref{thm:multfree}, Fact \ref{fct:ssindres}, and Lemma \ref{lem:cliff2} applied to the exact sequence
\[ 1 \to G^0(k) \to G'(k) \to \pi_0(S)(k) \to 1 \]
imply that $\delta$ annihilates the kernel of the action of $\pi_0(S)(k)$ on the set of irreducible constituents of $\kappa_{(S^0,\theta^0)}^{G^0}$. This action on that set factors through the map $S(k) \to [S/Z](k)=[S^0/(Z \cap G^0)](k) \to S^0_\tx{ad}(k)$. Proposition \ref{pro:dlp} implies that the kernel of this action is the kernel of the map $\pi_0(S)(k) \to \tx{cok}(S^0(k) \to S^0_\tx{ad}(k)) \to \Omega(S^0,G^0)(k)_{\theta^0}^*$ given by composing $S(k) \to S^0_\tx{ad}(k)$ with \eqref{eq:wtheta1'}. Corollary \ref{cor:bichar} states that this composition is the dual of the map $\Omega(S^0,G^0)(k)_{\theta^0} \to \Omega(S,G)(k)_{\theta^0}/\Omega(S,G)(k)_\theta \to \pi_0(S)(k)^*$ realized by
 $(s,w) \mapsto \theta^0(wsw^{-1}s^{-1})$. It follows that there exists an element $w \in \Omega(S^0,G^0)(k)_{\theta_0}$ such that $\delta(s)=\theta^0(wsw^{-1}s^{-1})$, i.e. $\theta_2(s)=\theta_1(wsw^{-1})$.
\end{proof}

\begin{fct} \label{fct:psigind}
The isomorphisms of Corollary \ref{cor:dl_ind4} identify $\Psi_{V,U}^{G^0}$ with $\Psi_{V,U}^{G^0 \cdot S}$ and $\tx{Ind}_{(G^0 \cdot S)^F}^{G^F}(\Psi_{V,U}^{G^0 \cdot S})$ with $\Psi_{V,U}^G$.
\end{fct}

\subsection{The internal structure of Deligne-Lusztig packets} \label{sub:dlpackpar}

We will now apply the geometric intertwining operators to the study of the internal structure of the representation $\kappa_{(S,\theta)}^G$ of Definition \ref{dfn:kappast}.

Let $\Gamma=N(S,G)(k)_{\theta^0}$, $\bar\Gamma=N(S,G)(k)_{\theta^0}/S(k)$, $X$ the set of unipotent radicals of all Borel $\bar k$-subgroups of $G^0$ containing $S^0$. According to Lemma \ref{lem:iopconj1} and Corollary \ref{cor:iopnat} the collection $\Psi^{G^0}$ is a $\Gamma$-equivariant collection of natural intertwining operators on $X$. Hence the class $[\eta]$ of Corollary \ref{cor:defeta} is defined.

\begin{asm} \label{asm:dl_normintex}
The group $\bar\Gamma$ is abelian and the class $[\eta]$ is trivial.
\end{asm}

We shall see in a moment that when $G=G^0$ the validity of this assumption is immediate from Lusztig's Theorem \ref{thm:multfree}. For general $G$ we do not know if this assumption is always satisfied. For those $G$ that occur in our study of $p$-adic groups we will prove in the next section that this assumption is indeed satisfied.

Let $U$ be the unipotent radical of a Borel $\bar k$-subgroup of $G^0$ containing $S^0$. According to Corollary \ref{cor:normiop} there exists a coheren splitting $\epsilon$ of the 2-cocycle $\eta_{\Psi^{G^0},U}$. Recall that this means that $\epsilon\cdot\Psi^{G^0}$ is a normalized $\Gamma$-equivariant collection of natural intertwining operators for the group $G^0$.

\begin{fct} The collection $\epsilon \cdot \Psi^G$ is a normalized $\Gamma$-equivariant collection of intertwining operators for the group $G$.
\end{fct}
\begin{proof}
This follows from Corollary \ref{fct:psigind} and the fact that for $n \in N(S,G)(k)_\theta$ the conjugation action $\tx{Ad}(n)$ on the variety $Y_U^G$ preserves the subvariety $Y_U^{G^0}$, so the action on $H^*_c(Y_U^G,\bar\Q_l)_\theta$ is induced from that on $H^*_c(Y_U^{G^0S},\bar\Q_l)_\theta$, and the latter is identified with that on $H^*_c(Y_U^{G^0},\bar\Q_l)_{\theta^0}$.
\end{proof}

\begin{rem} Assumption \ref{asm:dl_normintex} is equivalent to the claim that the cocycles $\{\eta_{\Psi,U}|U \in X\}$ are cohomologically trivial. The choice of $\{\epsilon_U\}$ is the choice of trivialization. While we will prove the existence of trivialization in the cases that we need, we will not be able to find a natural trivialization. Since the cocycles are naturally given, studying them would hopefully lead to a natural choice of trivialization. We shall come to this in a forthcoming paper.
\end{rem}

For $n \in N(S,G)(k)_\theta$ we define a self-intertwining operator
\begin{equation} \label{eq:self_int}
R_U^{G,\epsilon}(n) : H^i_c(Y_U^G,\bar\Q_l)_\theta \to H^i_c(Y_U^G,\bar\Q_l)_\theta
\end{equation}
as follows. The map $gU \mapsto gUn^{-1}$ is an isomorphism of varieties $Y_U^G \to Y_{{^nU}}^G$, where we write $^nU=nUn^{-1}$, that commutes with the left $G(k)$-action and translates the action of $s \in S(k)$ to the action of $nsn^{-1} \in S(k)$. It induces a $G(k)$-equivariant isomorphism $r_{n^{-1}} : H^i_c(Y_U^G,\bar\Q_l) \to H^i_c(Y_{^nU}^G,\bar\Q_l)$ that respects the $\theta$-isotypic components for the right action by $S(k)$. We define $R_U^{G,\epsilon}(n) = [\epsilon\Psi^G]_{U,{^nU}}\circ r_{n^{-1}}$.

The resulting map $n \mapsto R_U^{G,\epsilon}(n)$ gives an action of $N(S,G)(k)_\theta$ on $H^{d_U}(Y_U^G,\bar\Q_l)_\theta$ that commutes with the action of $G(k)$ and extends the action of $S(k)$ on this vector space obtained by inverting the action coming from right multiplication, i.e. the action given by $s \mapsto \theta(s)^{-1}$. In other words, we now have an action of $G(k) \times N(S,G)(k)_\theta$ on $H^{d_U}(Y_U,\bar\Q_l)_\theta$.

\begin{fct} \label{fct:normiopconj}
Given $U,V \in X$ and $n \in N(S,G)(F)_\theta$ we have
\[ \Psi_{V,U}\circ R_U^{G,\epsilon}(n) \circ \Psi_{V,U}^{-1} = R_V^{G,\epsilon}(n). \]
\end{fct}
\begin{proof}
The construction of the operator \eqref{eq:self_int} and Lemma \ref{lem:iopconj1}, which remains valid in the disconnected case with the same proof, we see that the left hand side of the claimed equation is $\epsilon_U(n) \cdot \Psi_{V,U} \circ \Psi_{U,^nU} \circ \Psi_{^nV,^nU}^{-1} \circ r_{n^{-1}}$. According to \eqref{eq:beta} this equals $\epsilon_U(n)\beta_{V,U}(n)\Psi_{V,^nV}\circ r_{n^{-1}}$. Definition \ref{dfn:cohspl} implies that this equals the right hand side of the claimed equation.
\end{proof}

\begin{cor} The isomorphism class of the representation of $G(k) \times N(S,G)(k)_\theta$ on $H_c^{d_U}(Y_U^G,\bar\Q_l)_\theta$ does not depend on the choice of $U$.
\end{cor}
\begin{proof}
Immediate from Fact \ref{fct:normiopconj}.
\end{proof}

\begin{dfn} We denote by $\kappa_{(S,\theta)}^{G,\epsilon}$ the isomorphism class of the representation of $G(k) \times N(S,G)(k)_\theta$ on $H_c^{d_U}(Y_U^G,\bar\Q_l)_\theta$. For $\rho \in \tx{Irr}(N(S,G)(k)_\theta,\theta)$ define
\[ \kappa_{(S,\theta,\rho)}^{G,\epsilon} = \tx{Hom}_{N(S,G)(k)_\theta}(\rho^\vee, \kappa_{(S,\theta)}^{G,\epsilon}|_{\{1\} \times N(S,G)(k)_\theta}). \]
\end{dfn}

Recall from Fact \ref{fct:cohsplituniq} that the set of coherent splittings $\{\epsilon\}$ is a torsor under the finite abelian group $(N(S,G)(k)_\theta/S(k))^*$.

\begin{thm} \label{thm:dlpar}
\begin{enumerate}
	\item The map $\rho \mapsto \kappa_{(S,\theta,\rho)}^{G,\epsilon}$ is a bijection 
	\[ \tx{Irr}(N(S,G)(k)_\theta,\theta) \to [\kappa_{(S,\theta)}^G]. \]
	\item The multiplicity of $\kappa_{(S,\theta,\rho)}^{G,\epsilon}$ in $\kappa_{(S,\theta)}^G$ is equal to $\tx{dim}(\rho)$, i.e. to the multiplicity of $\theta$ in $\rho|_{S(k)}$.
	\item For $\delta \in (N(S,G)(k)_\theta/S(k))^*$ and $\phi : \tx{cok}(G^0_\tx{sc}(k) \to G(k)) \to \bar\Q_l^\times$ we have
\[ \kappa_{(S,\theta,\rho)}^{G,\delta\epsilon} = \kappa_{(S,\theta,\delta\otimes\rho)}^{G,\epsilon},\qquad \kappa_{(S,\phi\otimes\theta,\phi\otimes\rho)}^{G,\epsilon} = \phi\otimes\kappa_{(S,\theta,\rho)}^{G,\epsilon}. \]
\end{enumerate}

\end{thm}
\begin{proof}
The first equality in the third point follows from $\kappa_{(S,\theta)}^{G,\delta\epsilon}=\kappa_{(S,\theta)}^{G,\epsilon}\otimes(\mathbf{1}\boxtimes\delta)$, which is immediate from the definition of the intertwining operators $R_U^{G,\epsilon}(n)$.

The rest of the theorem will be proved in stages: first for the case $G=G^0$, then for the case $G=G^0 \cdot S$, and then for general $G$. Fix a unipotent radical $U$ of a Borel $\bar k$-subgroup of $G^0$ containing $S^0$, so that $\kappa_{(S,\theta)}^G$ is realized in $H^{d_U}_c(Y_U^G,\bar\Q_l)_\theta$. 

Assume first that $G=G^0$. The two essential inputs in this case are Theorem \ref{thm:multfree} and Proposition \ref{pro:dlcharext}, which already imply the second point. The second equality of the third point follows from Lemma \ref{lem:iopcharshift} applied to $\alpha=c_n$. To prove the first point, choose $\rho$ for which $\kappa_{(S,\theta,\rho)}^{G,\epsilon}$ is non-zero. We claim that for any $\bar s \in S_\tx{ad}(k)$ the $G(k)$-representations $\kappa_{(S,\theta,\rho)}^{G,\epsilon}\circ\tx{Ad}(\bar s)^{-1}$ and $\kappa_{(S,\theta,\rho\otimes\delta_{\bar s})}^{G,\epsilon}$ are isomorphic, where $\delta_{\bar s}$ is the image of $\bar s$ under the map $S_\tx{ad}(k) \to \Omega(S,G)(k)_\theta^*$ given by \eqref{eq:wtheta1'}. Granting this claim and using Proposition \ref{pro:dlcharext}, we see that for all extensions $\rho$ of $\theta$ the isotypic component $\kappa_{(S,\theta,\rho)}^{G,\epsilon}$ is non-zero. By Theorem \ref{thm:multfree} and Lemma \ref{lem:cliff1} the number of irreducible $G(k)$-subrepresentations of $\kappa_{(S,\theta)}^G$ is equal to the number of extensions $\rho$ of $\theta$, namely $|\Omega(S,G)(k)_\theta^*|$. This shows that each isotypic component $\kappa_{(S,\theta,\rho)}^{G,\epsilon}$ is irreducible, and the case $G=G^0$ is complete modulo the outstanding claim.

To prove the claim we consider the isomorphism $\tx{Ad}(\bar s) : Y_U \to Y_U$. It induces a vector space isomorphism
\[ H^{d_U}_c(Y_U,\bar\Q_l)_\theta \to H^{d_U}_c(Y_U,\bar\Q_l)_\theta \]
that translates the action of $g \in G(k)$ on its source to the action of $\tx{Ad}(\bar s)g$ on its target. By Lemma \ref{lem:iopconj1} the collection of intertwining operators $\Psi$ is $S_\tx{ad}(k)$-equivariant, therefore the above isomorphism translates the action of $R_U^{G,\epsilon}(n)$ on its source to the action of $R_U^{G,\epsilon}(\tx{Ad}(\bar s)n)$ on its target. Therefore the isomorphism $\tx{Ad}(\bar s)$ identifies the representations $\kappa_{(S,\theta,\rho)}^{G,\epsilon}$ and $\kappa_{(S,\theta,\rho\circ\tx{Ad}(\bar s)^{-1})}^{G,\epsilon}\circ\tx{Ad}(\bar s)$, or equivalently $\kappa_{(S,\theta,\rho)}^{G,\epsilon}\circ\tx{Ad}(\bar s)^{-1}$ and $\kappa_{(S,\theta,\rho\circ\tx{Ad}(\bar s)^{-1})}^{G,\epsilon}$. Clearly $\rho\circ\tx{Ad}(\bar s)^{-1}$ extends $\theta$, so it is given by $\rho \cdot \delta$ for a uniquely determined $\delta \in \Omega(S,G)(k)_\theta^*$, which we can then evaluate at $n \in N(S,G)(k)_\theta$ by the formula $\delta(n) = \rho((\tx{Ad}(\bar s)^{-1} n)  \cdot n^{-1}) = \rho(n \dot sn^{-1} \dot s^{-1})$ where $\dot s\in S_\tx{sc}(\bar k)$ is any lift of $\bar s$. A look at Lemma \ref{lem:dlbichar} reveals that $\delta=\delta_{\bar s}$. The case $G=G^0$ is now complete.

Assume next that $G=G^0 \cdot S$. Recall from Remark \ref{rem:kappaext} that $\kappa_{(S,\theta)}^G$ is an extension of $\kappa_{(S^0,\theta^0)}^{G^0}$, and that the action of $S(k)$ on this extension is given by $sv=\tx{Ad}(s)v\theta(s)$, where $\tx{Ad}(s)$ is an action on the realization $H^{d_U}_c(Y_U^{G^0},\bar\Q_l)_{\theta^0}$ of $\kappa_{(S^0,\theta^0)}^{G^0}$. Since $\kappa_{(S^0,\theta^0)}^{G^0}$ is multiplicity\-free, so is $\kappa_{(S,\theta)}^G$. On the other hand, since $N(S,G^0)(k)_\theta$ is a subgroup of $N(S,G^0)(k)_{\theta^0}$ containing $S^0(k)$, Proposition \ref{pro:dlcharext} implies that $\theta^0$ extends to $N(S,G^0)(k)_\theta$. Now $N(S,G)(k)=N(S,G^0)(k) \cdot S(k)$, so the extensions of $\theta$ from $S(k)$ to $N(S,G)(k)_\theta$ are in 1-1 correspondence with the extensions of $\theta^0$ from $S^0(k)$ to $N(S,G^0)(k)_\theta$: if $\rho_0 \in \tx{Irr}(N(S,G^0)(k)_{\theta},\theta^0)$, then $\rho_0 \otimes \theta$ is a character of $N(S,G^0)(k)_\theta \times S(k)$ that descends to $N(S,G^0)(k)_\theta \cdot S(k)=N(S,G)(k)_\theta$. The second point now follows.
 
To prove the first point, note that an irreducible constituent of $\kappa_{(S,\theta)}^G$ is simply the direct sum of a $G(k)$-orbit of irreducible constituents of $\kappa_{(S^0,\theta^0)}^{G^0}$. Since $G(k)=G^0(k) \cdot S(k)$, a $G(k)$-orbit is the same as an $S(k)$-orbit, in fact a $\pi_0(S)(k)$-orbit.  Using the case $G=G^0$, let $\kappa_{(S^0,\theta^0,\rho^0)}^{G^0,\epsilon}$ be an irreducible constituent of $\kappa_{(S^0,\theta^0)}^{G^0}$ corresponding to an extension $\rho^0$ of $\theta^0$ to $N(S^0,G^0)(k)_{\theta^0}$. The $\tx{Ad}(s)$-equivariance of the collection $\Psi^{G^0}$ of intertwining operators shows that the map $v \mapsto sv$ on $\kappa_{(S^0,\theta^0)}^{G^0}$ maps $\kappa_{(S^0,\theta^0,\rho^0)}^{G^0,\epsilon}$ to $\kappa_{(S^0,\theta^0,\rho^0\circ\tx{Ad}(s)^{-1})}^{G^0,\epsilon}$. In other words, the irreducible constituents of $\kappa_{(S^0,\theta^0)}^{G^0}$ whose direct sum makes up an irreducible constituent of $\kappa_{(S,\theta)}^{G}$ are those $\kappa_{(S^0,\theta^0,\rho^0)}^{G^0,\epsilon}$, where $\rho^0$ runs over an orbit in $\tx{Irr}(N(S^0,G^0)(k)_{\theta^0},\theta^0)$ for the action of $S(k)$ by conjugation. 

Recall that $N(S^0,G^0)=N(S,G^0)$. We claim that two elements of $\tx{Irr}(N(S^0,G^0)(k)_{\theta^0},\theta^0)$ are in the same $S(k)$-orbit if and only if they have the same restriction to $N(S,G^0)(k)_\theta$. Indeed, two elements $\rho_1,\rho_2 \in \tx{Irr}(N(S^0,G^0)(k)_{\theta^0},\theta^0)$, being linear characters by Proposition \ref{pro:dlcharext}, have the same restriction to $N(S,G^0)(k)_\theta$ if and only if $\rho_2=\rho_1\delta$ for a character $\delta$ of $N(S^0,G^0)(k)_{\theta^0}/N(S,G^0)(k)_\theta = \Omega(S^0,G^0)(k)_{\theta^0}/\Omega(S,G^0)(k)_\theta$. Corollary \ref{cor:bichar} implies the existence of $s \in S(k)$ such that $\delta(w^{-1})=\theta^0(wsw^{-1}s^{-1})$. Therefore $\rho_2=\rho_1\delta$ is equivalent to $\rho_2(n)=\rho_1(n)\theta^0(n^{-1}sns^{-1})=\rho_1(sns^{-1})$, proving the claim.

We have thus shown that the irreducible constituents of $\kappa_{(S,\theta)}^G$ is indexed by element $\rho \in \tx{Irr}(N(S^0,G^0)(k)_{\theta},\theta^0)$, the correspondence being that the constituent corresponding to $\rho$ is given by $\bigoplus_{\rho_0} \kappa_{(S^0,\theta^0,\rho^0)}^{G,\epsilon}$ as $\rho_0 \in \tx{Irr}(N(S^0,G^0)(k)_{\theta^0},\theta^0)$ runs over those elements whose restriction to $N(S^0,G^0)(k)_{\theta}$ is $\rho$. But Fact \ref{fct:psigind} implies that this direct sum is precisely the $\rho$-isotypic constituent for the right action of $N(S,G^0)(k)_\theta$, hence also of $N(S,G)(k)_\theta=N(S,G^0)(k)_\theta \cdot S(k)$, since the right action of $S(k)$ is via $\theta$ on all of $\kappa_{(S,\theta)}^G$. With this, the first point is proved, and the third point follows at once from its validity for $\kappa_{(S^0,\theta^0)}^{G^0}$. The case $G=G^0 \cdot S$ is thus complete.

Finally consider a general $G$. According to Lemma \ref{lem:hcbp} the first two points are equivalent to the statement that the set $\{R_U^{G,\epsilon}(n) \cdot \C|n \in N(S,G)(k)_\theta/S(k)\}$ is both linearly independent and generating in $\tx{End}_{G(k)}(H^{d_U}_c(Y_B^G,\bar\Q_l)_\theta)$.

It is now useful to introduce a variant of the intertwining operator $R_U^{G,\epsilon}(n)$. Consider for $n \in N(S,G)(k)_\theta$ the conjugation morphism $c_n=\tx{Ad}(n)$ on $G$ and define $C_U^{G,\epsilon}(n)=c_n^G \circ [\epsilon\Psi]_{U,{^nU}}^G$. Of course $C_U^{G,\epsilon}(n)=l_n \circ R_U^{G,\epsilon}(n)$, where $l_n$ is the action of $n \in N(S,G)(k)_\theta \subset G(k)$. If we let $N(S,G)(k)_\theta$ act on $H^{d_U}_c(Y_U^G,\bar\Q_l)_\theta$ via the operators $C_U^{G,\epsilon}(n)$ instead of the operators $R_U^{G,\epsilon}(n)$ we obtain the structure of a $G(k) \rtimes N(S,G)(k)_\theta$-representation.

We have the intermediate group $G'=G^0 \cdot S$. According to Corollary \ref{cor:dl_ind4} we have $H^{d_U}_c(Y_U^G,\bar\Q_l)_\theta=\tx{Ind}_{G'(k)}^{G(k)}H^{d_U}_c(Y_U^{G'},\bar\Q_l)_\theta$. The utility of the operator $C_U^{G,\epsilon}(n)$ comes from the fact that it is also induced from $G'$. Indeed, the conjugation morphism $c_n$ preserves $G^0$. Since $G/G^0$ is abelian by Assumptions \ref{asm:dl_mu}, $c_n$ also preserves $G'$ and therefore induces a morphism $c_n^{G'} : Y_U^{G'} \to Y_{^nU}^{G'}$ that is immediately seen to be the restriction of the morphism $c_n^G : Y_U^G \to Y_{^nU}^G$. On the level of cohomology this implies $c_n^G=\tx{Ind}_{G'(k)}^{G(k)}c_n^{G'}$. At the same time, according to Fact \ref{fct:psigind} we have $\Psi_{V,U}^G=\tx{Ind}_{G'(k)}^{G(k)}(\Psi_{V,U}^{G'})$ and therefore we also have $\Phi_{V,U}^G=\tx{Ind}_{G'(k)}^{G(k)}(\Phi_{V,U}^{G'})$. This implies $C_U^{G,\epsilon}(n)=\tx{Ind}_{G'(k)}^{G(k)}C_U^{G',\epsilon}(n)$, as claimed. This means that
\[ H^{d_U}_c(Y_U^G,\bar\Q_l)_\theta=\tx{Ind}_{G'(k) \rtimes N(S,G)(k)_\theta}^{G(k)\rtimes N(S,G)(k)_\theta}H^{d_U}_c(Y_U^{G'},\bar\Q_l)_\theta. \]

We can now apply Proposition \ref{pro:hcbi} and reduce the proof to the following two points:
\begin{enumerate}
	\item For every $g \in G(k)$ the representations $\kappa_{(S,\theta)}^{G'} \circ \tx{Ad}(g)$ and $\kappa_{(S,\theta)}^{G'}$ are isomorphic if $g \in G'(k) \cdot N(S,G)(k)_\theta$, and disjoint otherwise.
	\item The set $\{\C \cdot R_U^{G',\epsilon}(n)|n \in N(S,G')(k)_\theta/S(k)\}$ is linearly independent and generating.
\end{enumerate}
The first point follows from Fact \ref{fct:ftwist} and Proposition \ref{pro:inddisj} applied to the group $G'$ and the pairs $(S,\theta)$ and $(S^g,\theta^g)$. The second point follows from Lemma \ref{lem:hcbp} and the case of $G=G^0 \cdot S$ applied to $G'$. The first point of the theorem is now complete.

For the behavior under $\rho \mapsto \phi\cdot\rho$ we apply Lemma \ref{lem:iopcharshift} to $\alpha=c_n$ to conclude that the isomorphism $\phi\otimes\kappa_{(S,\theta)}^{G'} \to \kappa_{(S,\theta\cdot\phi)}^{G'}$ of $G'(k)$-representations of Lemma \ref{lem:dl_kappatwist} is also equivariant with respect to the operators $C_n^\epsilon$ and therefore is an isomorphism of representations of $G'(k)\rtimes N(S,G)(k)_\theta$. Note that there are two equivalent ways to think of $\phi\otimes\kappa_{(S,\theta)}^{G'}$ as a representation of $G'(k)\rtimes N(S,G)(k)_\theta$ -- either as the vector space underlying $\kappa_{(S,\theta)}^{G'}$, on which the action of $G'(k)$ is twisted by $\phi$, and the action of $N(S,G)(k)_\theta$ is unaltered, or as the representation $\kappa_{(S,\theta)}^{G'}$ of $G'(k)\rtimes N(S,G)(k)_\theta$ twisted by $\phi$, where $\phi$ is now viewed as a character of the group $G'(k)\rtimes N(S,G)(k)_\theta$ that is trivial on $N(S,G)(k)_\theta$. After induction we obtain the isomorphism $\phi\otimes \kappa_{(S,\theta)}^{G,\epsilon} \to \kappa_{(S,\theta\cdot\phi)}^{G,\epsilon}$ of $G(k) \rtimes N(S,G)(k)_\theta$-representations, and thus the isomorphism $(\phi \boxtimes \phi^{-1})\otimes \kappa_{(S,\theta)}^{G,\epsilon} \to \kappa_{(S,\theta\cdot\phi)}^{G,\epsilon}$ of $G(k) \times N(S,G)(k)_\theta$.
\end{proof}

\begin{cor} \label{cor:dlsurj}
Every irreducible representation of $G(k)$ whose restriction to $G^0(k)$ contains an irreducible non-singular cuspidal representation is of the form $\kappa_{(S,\theta,\rho)}^{G,\epsilon}$.
\end{cor}
\begin{proof}
Let $\tau$ be an irreducible representation of $G(k)$ whose restriction to $G^0(k)$ contains a non-singular cuspidal representation. Thus there exists a pair $(S^0,\theta^0)$ of an elliptic maximal torus $S^0 \subset G^0$ and a non-singualr character $\theta^0$ of $S^0(k)$ s.t. $\tx{Hom}_{G^0(k)}(\tau,\kappa_{(S^0,\theta^0)}^{G^0}) \neq 0$. Let $\theta$ be any extension of $\theta^0$ to $S(k)$. Let $G'=G^0 \cdot S$. Recall from Corollary \ref{cor:dl_ind4} that $\kappa_{(S,\theta)}^{G'}$ is an extension of $\kappa_{(S^0,\theta^0)}^{G^0}$ to $G'(k)$. Trivially we still have $\tx{Hom}_{G^0(k)}(\tau,\kappa_{(S,\theta)}^{G'}) \neq 0$. By Fact \ref{fct:ssindres} the restriction of $\tau$ to $G'(k)$ is semi-simple. Let $\tau'$ be an irreducible constituent of that restriction and, using Theorem \ref{thm:dlpar}, let $\rho' \in \tx{Irr}(N(S,G')(k)_\theta,\theta)$ be such that $\tx{Hom}_{G^0(k)}(\tau',\kappa_{(S,\theta,\rho')}^{G'}) \neq 0$. Since $G^0(k)\cdot Z(k)$ is of finite index in $G'(k)$, an irreducible representation of $G'(k)$ restricts semi-simply to $G^0(k)\cdot Z(k)$ by Fact \ref{fct:ssindres}, and then further restricts semi-simply to $G^0(k)$ since $Z(k)$ is central and acts by a character on each irreducible representation. Apply Lemma \ref{lem:cliff2} to the exact sequence
\[ 1 \to G^0(k) \to G'(k) \to \pi_0(S)(k) \to 1 \]
to obtain a character $\phi : \pi_0(S)(k) \to \bar\Q_l^\times$ such that $\tau'=\phi \otimes \kappa_{(S,\theta,\rho')}^{G'}=\kappa_{(S,\phi\theta,\phi\rho')}^{G'}$, the second equality by Theorem \ref{thm:dlpar}. It follows from Corollary \ref{cor:dl_ind4} that
\[ \tx{Hom}_{G(k)}(\tau,\kappa_{(S,\theta)}^G)=\tx{Hom}_{G'(k)}(\tx{Res}^{G(k)}_{G'(k)}\tau,\kappa_{(S,\theta)}^{G'}) \neq \{0\}.\]
\end{proof}

We finish this section we a few remarks on the case $G=G^0$.

\begin{lem} \label{lem:dletatriv}
Assume that $G=G^0$. Then Assumption \ref{asm:dl_normintex} holds.
\end{lem}
\begin{proof}
For $n \in N(S,G)(k)_\theta$ consider the operator $R_U^{G,1}(n)=\Psi^G_{U,{^nU}}\circ r_{n^{-1}}$. Using that $\Psi$ is a  $\Gamma$-equivariant collection of $G(k)$-equivariant operators, we compute
\[ R_U^{G,1}(n) \circ R_U^{G,1}(m) = \eta_{\Psi,U}(1,n,nm)\cdot R_U^{G,1}(nm). \]
Choose an irreducible constituent $\pi \in [\kappa_{(S,\theta)}^G]$. According to Theorem \ref{thm:multfree} for each $n \in N(S,G)(k)_\theta$ the operator $R_U^{G,1}(n)$ preserves $\pi$ and hence acts on it by a scalar $\xi_n \in \bar\Q_l^\times$. We have $\xi \in C^1(N(S,G)(k)_\theta,\bar\Q_l^\times)$ and for $s \in S(k)$ we have $\xi_s=\theta(s)^{-1}$. By Proposition \ref{pro:dlcharext} there exists an extension $\tilde\theta$ of $\theta$ to $N(S,G)(k)_\theta$. We set $\tilde \xi_n := \tilde\theta(n) \cdot \xi_n$. Then $\tilde \xi \in C^1(\Omega(S,G)(k)_\theta,\bar\Q_l^\times)$ and using inhomogeneous notation we see that $\partial \tilde \xi(n,m) = \partial \xi(n,m) = \eta_{\Psi,U}(n,m)$ holds for all $n,m \in N(S,G)(k)_\theta$, hence $\partial \tilde \xi = \eta_{\Psi,U}$ holds in $Z^2(\Omega(S,G)(k)_\theta,\bar\Q_l^\times)$.
\end{proof}

\begin{lem} Assume that $G=G^0$. The bijection $\rho \mapsto \kappa_{(S,\theta,\rho)}^{G,\epsilon}$ of Theorem \ref{thm:dlpar} is equivariant with respect to the action of $\Omega(S,G)(k)_\theta^*$ on the left hand side by multiplication and the action on the right hand side given by Proposition \ref{pro:dlp}.
\end{lem}
\begin{proof}
This was proved in the course of proving Theorem \ref{thm:dlpar}: it is the claim that for $\bar s \in S_\tx{ad}(k)$ the $G(k)$-representations $\kappa_{(S,\theta,\rho)}^{G,\epsilon}\circ\tx{Ad}(\bar s)^{-1}$ and $\kappa_{(S,\theta,\rho\otimes\delta_{\bar s})}^{G,\epsilon}$ are isomorphic, where $\delta_{\bar s}$ is the image of $\bar s$ under the map $S_\tx{ad}(k) \to \Omega(S,G)(k)_\theta^*$ given by \eqref{eq:wtheta1'}.
\end{proof}

\section{Non-singular Deligne-Lusztig packets over local fields} \label{sec:nsdl-dz}

Let $G$ be a connected reductive group defined over a non-archimedean local field $F$, $S \subset G$ an elliptic maximally unramified maximal torus, $S' \subset S$ its maximal unramified subtorus. We consider the set $R_\tx{res}(S',G)$ of restrictions to $S'$ of the absolute roots $R(S,G)$. Since $S'$ is a maximally split torus of $G$ over $F^u$, this is in fact a (possibly non-reduced) root system.

Let $x \in \mc{B}(G,F)$ be the point associated to $S$. It is a vertex \cite[Lemma 3.4.3]{KalRSP}. There exists a unique smooth integral model of $G$ whose group of $O_{F^u}$-points equals the stabilizer $G(F^u)_x$. This model is locally of finite type. Let $\ms{G}_x$ be the quotient of the special fiber of this model modulo its (connected) unipotent radical. Then $\ms{G}_x$ is a smooth $k_F$-group scheme. Its neutral connected component $\ms{G}_x^\circ$ is a connected reductive $k_F$-group -- it is the reductive quotient of the special fiber of the parahoric group scheme of $G$ associated to the vertex $x$. We have $\ms{G}_x(k_F)=G(F)_x/G(F)_{x,0+}$ and $\ms{G}_x^\circ(k_F)=G(F)_{x,0}/G(F)_{x,0+}$.

The point $x$ lies in $\mc{B}(S,F)$ and we have the corresponding $k_F$-group schemes $\ms{S}$ and $\ms{S}^\circ$ satisfying $\ms{S}(k_F)=S(F)/S(F)_{0+}$ and $\ms{S}^\circ(k_F)=S(F)_0/S(F)_{0+}$. Note however that $\ms{S}(\bar k_F)$ does not equal $S(F^u)/S(F^u)_{0+}$, but rather $S(F^u)_c/S(F^u)_{0+}$, where $S(F^u)_c$ is the preiamge of $S_\tx{ad}(F^u)_b$. We recall here that since $S$ is maximally unramified, $S_\tx{ad}$ has induced ramification, and hence $S_\tx{ad}(F^u)_b=S_\tx{ad}(F^u)_0$, cf. \cite[Fact 3.1.2]{KalRSP}.

\subsection{Non-singular characters}

Let $F'/F$ be an unramified extension splitting $S'$. 

\begin{dfn} \label{dfn:nsc} 
A character $\theta : S(F) \to \C^\times$ will be called
\begin{enumerate}
	\item $F$-\emph{non-singular}, if for every $\alpha_\tx{res} \in R_\tx{res}(S',G)$ the character
\[ \theta \circ N_{F'/F} \circ \alpha_\tx{res}^\vee : F'^\times \to \C^\times \]
has non-trivial restriction to $O_{F'}^\times$;
	\item $k_F$-\emph{non-singular}, if for every $\bar\alpha \in R(\ms{S}^\circ,\ms{G}_x^\circ) \subset R_\tx{res}(S',G)$ the character
\[ \theta \circ N_{F'/F} \circ \bar\alpha^\vee : F'^\times \to \C^\times \]
has non-trivial restriction to $O_{F'}^\times$.
\end{enumerate}
\end{dfn}

\begin{rem} \label{rem:ns}
The choice of $F'$ is irrelevant, because for any finite unramified extension $F''/F'$ the norm map $N_{F''/F'} : O_{F''}^\times \to O_{F'}^\times$ is surjective. 
\end{rem}

\begin{rem}
We have not assumed in the definition that $\theta$ is of depth zero. Our main applications of that definition will be to the case that $\theta$ is of depth zero, or slightly more generally, a product of a depth-zero character of $S(F)$ with a character of $G(F)$ that may have positive depth. When $\theta$ is of depth zero, it factors through a character of $\ms{S}(k_F)=S(F)/S(F)_{0+}$. We denote its restriction to $\ms{S}^\circ(k_F)=S(F)_0/S(F)_{0+}=S'(F)_0/S'(F)_{0+}$ by $\theta^\circ$. By \cite[Lemma 3.4.14]{KalRSP} the character $\theta$ is $k_F$-non-singular if and only if $\theta^\circ$ is non-singular with respect to $\ms{G}_x^\circ$ in the sense of \cite[Definition 5.15]{DL76}.
\end{rem} 

\begin{rem} If $(S_1,\theta_1)$ and $(S_2,\theta_2)$ are stably conjugate pairs, then $\theta_1$ is $F$-non-singular if and only if $\theta_2$ is. On the other hand, the $k_F$-non-singularity of $\theta_1$ does not a-priori imply anything about the $k_F$-non-singularity of $\theta_2$.
\end{rem}

\begin{fct} \label{fct:regns}
Let $\theta : S(F) \to \C^\times$ be a depth zero character.
\begin{enumerate}
	\item If $\theta$ is $F$-non-singular then it is $k_F$-non-singular.
	\item If the vertex $x$ is absolutely special and $\theta$ is $k_F$-non-singular, then it is $F$-non-singular.
	\item If $\theta$ is regular in the sense of \cite[Definition 3.4.16]{KalRSP}, then $\theta$ is $F$-non-singular.
\end{enumerate}
\end{fct}
\begin{proof}
The root system $R(\ms{S}^\circ,\ms{G}_x^\circ)$ is a subset of $R_\tx{res}(S',G)$ and according to \cite[Lemma 3.4.14]{KalRSP} $F$-non-singularity implies $k_F$-non-singularity.

Assume now that $x$ is absolutely special and that $\theta$ is $k_F$-non-singular. Then $R(\ms{S}^\circ,\ms{G}_x^\circ)$ is the set of reduced roots in $R_\tx{res}(S',G)$. Thus given $\alpha_\tx{res} \in R(S',G)$ either $\beta_\tx{res}=\alpha_\tx{res}$ or $\beta_\tx{res}=\frac{1}{2}\alpha_\tx{res}$ lies in $R(\ms{S}^\circ,\ms{G}_x^\circ)$. By assumption $\theta \circ N_{F'/F} \circ \beta_\tx{res}^\vee$ has non-trivial restriction to $O_{F'}^\times$. Since $\beta_\tx{res}^\vee=\alpha_\tx{res}^\vee$ or $\beta_\tx{res}^\vee=2\alpha_\tx{res}^\vee$ this implies that $\theta \circ N_{F'/F} \circ \alpha_\tx{res}^\vee$ has non-trivial restriction to $O_{F'}^\times$.

Assume now that $\theta$ is not $F$-non-singular. We want to show that it cannot be regular. The torus $S$ transfers to the quasi-split inner form of $G$ \cite[Lemma 3.2.2]{KalRSP}, so we may assume that $G$ is quasi-split. We may further replace $S$ by a stable conjugate to ensure that the vertex $x$ is absolutely special, according to \cite[Lemma 3.4.12]{KalRSP}. By the previous point, $\theta$ is not $k_F$-non-singular. According to Remark \ref{rem:ns} and \cite[Corollary 5.18]{DL76} there exists $w \in \Omega(\ms{S}^\circ,\ms{G}_x^\circ)(k_F)$ stabilizing $\theta^\circ$. Now \cite[Lemma 3.4.10]{KalRSP} precludes the regularity of $\theta$.
\end{proof}

\subsection{The verification of Assumptions \ref{asm:dl_mu} and \ref{asm:dl_normintex}} \label{sub:ass_check}

Let $\theta : S(F) \to \C^\times$ be a $k_F$-non-singular depth zero character. We would like to use the results of the previous section, in particular Theorem \ref{thm:dlpar}, to study the representation $\kappa_{(\ms{S},\theta)}^{\ms{G}_x}$ of Definition \ref{dfn:kappast}, and combine this information with the results of Moy-Prasad \cite{MP96}. For this, we must verify Assumptions \ref{asm:dl_mu} and \ref{asm:dl_normintex}.

Assumption \ref{asm:dl_mu} is easy to check. The component group $\pi_0(\ms{G}_x)$ is described by the Kottwitz isomorphism as a subgroup of $\pi_1(G)_I$, and by construction $\pi_1(G)$ is a finitely generated abelian group. We let $\ms{Z} \subset \ms{G}_x$ be the image in $G(F^u)_x/G(F^u)_{x,0+}$ of $Z_G(F^u)$, where $Z_G$ is the center of $G$. Then $[\ms{Z} \cdot \ms{G}_x^\circ : \ms{G}_x]=[G(F^u)_{x,0} \cdot Z_G(F^u):G(F^u)_x]<\infty$. We must also check that $\ms{S}=\ms{S}^\circ \cdot \ms{Z}$ as per Notation \ref{ntt:dl_s}. Let $\ms{T}$ be reductive quotient of the special fiber of the the connected Neron model of $S_\tx{ad}$. The map $S(F^u)_c \to S_\tx{ad}(F^u)_b=S_\tx{ad}(F^u)_0$ induces a map $\ms{S} \to \ms{T}$. Since $S_\tx{sc}(F^u)_0/S_\tx{sc}(F^u)_{0+} \to S_\tx{ad}(F^u)_0/S_\tx{ad}(F^u)_{0+}$ is an isogeny of tori over $\bar k_F$, it is surjective, and we conclude that the restriction of $\ms{S} \to \ms{T}$ to $\ms{S}^\circ$ is already surjective. Since $\ms{Z}=\tx{ker}(\ms{S} \to \ms{T})$ we have $\ms{S}=\ms{S}^\circ \cdot \ms{Z}$, as required.

As we have already noted, the equality $\ms{S}=\ms{S}^\circ \cdot \ms{Z}$ is on the level of algebraic groups and may fail on the level of rational points. If $S$ is unramified then the equality holds on the level of rational points as well, and in fact $S(F)=Z_G(F) \cdot S(F)_0$, see e.g. \cite[Lemma 7.1.1]{KalECI}. Therefore $\Omega(S,G)(F)_\theta=\Omega(S,G)(F)_{\theta^\circ}$. In the ramified case the equality $S(F)=Z_G(F) \cdot S(F)_0$ fails, see \cite[\S3.4.3]{KalRSP}.

We now begin to verify Assumption \ref{asm:dl_normintex}.

\begin{lem} \label{lem:weylab2} The group $\Omega(S,G)(F)_{\theta^\circ}$ is abelian. If $G$ is absolutely simple and simply connected, then $\Omega(S,G)(F)_{\theta^\circ}$ is cyclic except when $G$ is split of type $D_{2n}$, in which case the possibilities are $\{1\}$, $\Z/2\Z$, and $(\Z/2\Z)^2$.
\end{lem}
\begin{proof}
According to \cite[Lemma 3.4.12]{KalRSP} we may assume that $G$ is quasi-split and the point $x \in \mc{B}(G,F)$ associated to $S$ is absolutely special. Then $\Omega(S,G)(F)_{\theta^\circ} = \Omega(\ms{S}^\circ,\ms{G}_x^\circ)(k_F)_{\theta^\circ}$, the latter being abelian by Corollary \ref{cor:weylab1}.

Assume now that $G$ is absolutely simple and simply connected. Then $\ms{G}_x^\circ$ is absolutely simple and semi-simple. By Corollary \ref{cor:weylab1} we have $\Omega(\ms{S}^\circ,\ms{G}_x^\circ)(k_F)_{\theta^\circ} \subset \tx{cok}(\ms{S}^\circ(k_F) \to \ms{S}^\circ_\tx{ad}(k_F)) = H^1(k_F,Z(\ms{G}_x^\circ))$. This group is cyclic as soon as $Z(\ms{G}_x^\circ)$ is, which is always the case except possibly when $\ms{G}_x^\circ$ is of type $D_{2n}$. This happens if and only if $G$ itself is unramified of type $D_{2n}$. In that case $\ms{G}_x^\circ$ is simply connected and its center is $\mu_2 \times \mu_2$. If $\ms{G}_x^\circ$ is not split, then $H^1(k_F,Z(\ms{G}_x^\circ))=\Z/2\Z$, so we may assume that $\ms{G}_x^\circ$ is split, which is the case if and only if $G$ is split. In that case, take $\ms{G}_x^*$ to be Lusztig's dual group of $\ms{G}_x^\circ$ and $s^* \in \ms{S}^* \subset \ms{G}_x^*$ the semi-simple element of the dual torus of $\ms{S}^\circ$ corresponding to $\theta^\circ$. Then $\Omega(\ms{S}^\circ,\ms{G}_x^\circ)(k_F)_{\theta^\circ}=\Omega(\ms{S}^*,\ms{G}_x^*)(k_F)_{s^*}$. The latter can be either of $\{1\}$, $\Z/2\Z$, or $(\Z/2\Z)^2$, and all possibilities are realized, as one observes using \cite[Proposition 2.1]{Ree10}.
\end{proof}

\begin{lem} The inclusion $N(S,G)(F) \to G(F)_x$ induces an isomorphism 
\[ N(S,G)(F)/S(F) \to N(\ms{S},\ms{G}_x)(k)/\ms{S}(k). \]
\end{lem}
\begin{proof}
An element of $G(F)$ that normalizes $S$ normalizes the apartment of $S$ over $F^u$ and acts on it $\Gamma$-equivariantly, hence fixes the point $x$. This gives the map $N(S,G)(F) \to N(\ms{S},\ms{G}_x)(k)$. The preimage of $\ms{S}$ under the map $G(F^u)_x \to \ms{G}_x(\bar k_F)$ is $G(F^u)_{x,0+} \cdot S(F^u)_c$. Since $G(F^u)_{x,0+} \cap N(S,G)(F^u)=S(F^u)_{0+}$ we see that the map $N(S,G)(F)/S(F) \to N(\ms{S},\ms{G}_x)(k)/\ms{S}(k)$ is injective. Its surjectivity follows from \cite[Lemma 3.4.5]{KalRSP}.
\end{proof}

\begin{rem} The Weyl group $\Omega(S,G)(F)$ and $\Omega(\ms{S},\ms{G}_x)(k_F)$ are related as follows. The identity $\ms{S}=\ms{S}^\circ \cdot \ms{Z}$ implies
\[ N(\ms{S},\ms{G}_x)/\ms{S} = N(\ms{S}^\circ,\ms{G}_x)/\ms{S}^\circ\cdot\ms{Z} = N(\ms{\bar S}^\circ,\ms{\bar G}_x)/\ms{\bar S}^\circ, \]
where ``bar'' denotes taking the quotient modulo $\ms{Z}$. Lang's theorem implies $\Omega(\ms{S},\ms{G}_x)(k_F)=N(\ms{\bar S}^\circ,\ms{\bar G}_x)(k_F)/\ms{\bar S}^\circ(k_F)$ and the previous lemma implies that this equals $N(S_\tx{ad},G_\tx{ad})(F)/S_\tx{ad}(F)$, which is a possibly proper subgroup of $\Omega(S,G)(F)$.
	
\end{rem}

\begin{pro} \label{pro:dznormiopex} The class $[\eta]$ in Assumption \ref{asm:dl_normintex} trivial.
\end{pro}
\begin{proof}
First note that $Z(F) \subset \Gamma$ acts trivially, so within this proof we can replace $\Gamma$ by $N(S,G)(F)_{\theta^\circ}/Z(F)$, as this does not affect $\bar\Gamma$. Next we notice that we are free to enlarge $\bar\Gamma$ if we like. It will be convenient to take for $\Gamma$ the stabilizer of $\theta_\tx{sc}=\theta|_{S(F)_\tx{sc}}$ in $N(S,G_\tx{ad})(F)$. Then $\bar\Gamma$ becomes the quotient of $[N(S,G_\tx{ad})(F)_{\theta_\tx{sc}}/S_\tx{ad}(F)]$ by the stabilizer of $\ms{U}$. Since $\Omega(S,G)(F)_{\theta_\tx{sc}}$ is abelian by Lemma \ref{lem:weylab2} applied to $G_\tx{sc}$, this quotient is a group.

Let $\ms{\tilde G}_x^\circ$ be the reductive quotient of the parahoric subgroup of $G_\tx{sc}$ associated to the vertex $x$. Then $\ms{\tilde G}_x^\circ \to \ms{G}_x^\circ$ is a morphism with abelian kernel and cokernel. By Lemma \ref{lem:iopindep} the class in $H^2(\bar\Gamma,\bar\Q_l^\times)$ is unchanged if we replace $G$ by $G_\tx{sc}$, so we may assume from now on that $G$ is simply connected. Note that now $\theta=\theta^\circ$.

By Corollary \ref{cor:etaprod} we may further assume that $G$ is $F$-simple. Thus $G=\tx{Res}_{E/F}H$ for some absolutely simple simply connected group $H$ defined over a finite extension $E$ of $F$. We have $\mc{B}(G,F)=\mc{B}(H,E)$. Moreover, $\ms{G}_x^\circ = \tx{Res}_{k_E/k_F}\ms{H}_x^\circ$ by Appendix \ref{app:parahoric}. Let us represent $\ms{G}_x^\circ$ as the product $\ms{H}_x^\circ \times \dots \times \ms{H}_x^\circ$ of $k=[k_E:k_F]$-many factors, with Frobenius acting by $F(h_1,\dots,h_k) \mapsto (F^kh_k,h_1,\dots,h_{k-1})$. Then $\ms{S}^\circ=\ms{T}' \times \dots \times \ms{T}'$ for an elliptic maximal torus $\ms{T}' \subset \ms{H}_x^\circ$. By Lemma \ref{lem:normiop} the class we are considering is independent of the choice of $\ms{U}$, so we may take $\ms{U}$ to have the form $\ms{V} \times \dots \times \ms{V}$ for some unipotent radical $\ms{V}$ of a Borel subgroup of $\ms{H}_x^\circ$ containing $\ms{T}'$. The diagonal embedding $\ms{H}_x^\circ \to \ms{G}_x^\circ$ provides an isomorphism of varieties $Y_{\ms{V}} \to Y_{\ms{U}}$. Under this isomorphism the geometric intertwining operators $\Psi$ for $\ms{G}_x^\circ$ provide a collection of $\Gamma$-equivariant operators for $\ms{H}_x^\circ$. We may therefore compute the class in $H^2(\bar\Gamma,\bar\Q_l^\times)$ using $\ms{H}_x^\circ$.

We are thus looking at an absolutely simple simply connected group $H$, a maximally unramified anisotropic maximal torus $T \subset H$ with vertex $x \in \mc{B}(H,E)$, and a $k_E$-non-singular 
character $\theta : T(E) \to \C^\times$. We have $\Gamma=N(T,H_\tx{ad})(E)_{\theta}$ and $\bar\Gamma$ is a subquotient of $\Omega(T,H)(E)_{\theta}$. This latter group is cyclic for all possible $H$ except for $H$ being of split type $D_{2n}$, in which case it can be one of $\{1\}$, $\Z/2\Z$, or $(\Z/2\Z)^2$ by Lemma \ref{lem:weylab2}. When $\bar\Gamma$ is cyclic $H^2(\bar\Gamma,\bar\Q_l^\times)$ vanishes, so we are left do deal with the case when $H$ is of split type $D_{2n}$, and $\Omega(T,H)(E)_{\theta}=(\Z/2\Z)^2$.

If $x$ is hyperspecial, then $\Omega(T,H)(E)_\theta=\Omega(\ms{T}',\ms{H}_x^\circ)(k_E)$ and we can apply Lemma \ref{lem:dletatriv} to conclude the triviality of the class in $H^2(\bar\Gamma,\bar\Q_l^\times)$. Assume now that $x$ is not hyperspecial. The root system of $\ms{H}_x^\circ$ is the product of two systems of type $D$. If an element of $\Omega(T,H)(E)_\theta$ swaps the two copies, then we can choose the Borel subgroup containing $\ms{T}$ to be invariant under this element, forcing the image of this element in $\bar\Gamma$ to be trivial. Then $\bar\Gamma$ is cyclic and again $H^2(\bar\Gamma,\bar\Q_l^\times)$ is trivial. We can thus assume that $\Omega(T,H)(E)_\theta$ preserves each of the two irreducible factors of the root system of $\ms{H}_x^\circ$. We apply Lemma \ref{lem:iopindep} to replace $\ms{H}_x^\circ$ with its simply connected cover and write it as a product of its two absolutely irreducible factors, and then apply Corollary \ref{cor:etaprod} to reduce to studying each factor separately.

There exist choices $\Delta$ and $\Delta'$ of simple roots for $R(T,H)$ with the following properties. There is a basis $e_1,\dots,e_{2n}$ of $X^*(T)_\R$ such that $\Delta=\{e_1-e_2,\dots,e_{2n-1}-e_{2n},e_{2n-1}+e_{2n}\}$ and the group $\Omega(T,H)(E)_\theta$ is given by $\<w_1,w_2\>$, where $w_1=\epsilon_1\epsilon_{2n}$, with $\epsilon_i(e_j)=(-1)^{\delta_{i,j}}e_j$, and $w_2=(-1)m$, where $m(e_i)=e_{2n+1-i}$. There exists a second basis $e_1',\dots,e_{2n}'$ of $X^*(T)_\R$ such that $\Delta'=\{e'_1-e'_2,\dots,e'_{2n-1}-e'_{2n},e'_{2n-1}+e'_{2n}\}$ and the set $\{-e'_1-e'_2,e'_1-e'_2,e'_2-e'_3,\dots,e'_{k-1}-e'_k\} \cup \{e'_{k+1}-e'_{k+2},\dots,e'_{2n-1}-e'_{2n},e'_{2n-1}+e'_{2n}\}$ is a set of simple roots for $R(\ms{T},\ms{H}_x^\circ)$. Let $w \in \Omega(T,H)$ be the unique element that sends $\Delta$ to $\Delta'$. It is a signed permutation of $\{e_i\}$ and thus sends $e_i$ to $(-1)^{s_i}e_{\sigma(i)}$. Then $w_1=\epsilon'_i\epsilon'_j$ where $1=\sigma(i)$ and $2=\sigma(j)$. The transposition swapping $e'_i$ and $e'_j$ is a factor of the permutation part of $w_2$. Since $w_2$ preserves both irreducible factors of the root system $R(\ms{T},\ms{H}_x^\circ)$, we must have that either $i,j \leq k$ or $i,j > k$. We thus see that $w_1$ acts trivially on one of the factors, and belongs to the Weyl group of the other factor. Consider now $w_2$. Its permutation part is a product of $n$ disjoint transpositions. Since it preserves both factors, these factors are of type $D_{2a}$ and $D_{2b}$ respectively for some $a+b=n$. The restriction of $w_2$ to $D_{2a}$ has a permutation part given by a product of $a$ disjoint transpositions, and the restriction of $w_2$ to $D_{2b}$ has a permutation part given by a product of $b$ disjoint transpositions. We conclude that $w_2$ acts on each factor by a Weyl element. We can now apply Lemma \ref{lem:dletatriv} again to each of the factors to conclude the triviality of the class $[\eta]$.
\end{proof}

\subsection{The internal structure of depth-zero Deligne-Lusztig packets} \label{sub:nsdl-dz}

We are now in position to apply the results of Section \ref{sec:nsdl-fin} to the study of representations of $G(F)$, and begin with the case of depth zero.

Consider a tuple $(S,\theta,\rho,\epsilon)$, where $S \subset G$ is a maximally unramified elliptic maximal torus, $\theta : S(F) \to \C^\times$ is a $k_F$-non-singular character of depth zero in the sense of Definition \ref{dfn:nsc}, $\rho$ is an irreducible representation of $N(S,G)(F)_\theta$ whose restriction to $S(F)$ contains $\theta$, and $\epsilon$ is a coherent splitting of the family of 2-cocycle $\{\eta_{\Psi,\ms{U}}\}$ as in \S\ref{sub:natiop}, where $\ms{U}$ ranges over the unipotent radicals of the Borel subgroups of $\ms{G}_x^\circ$ containing $\ms{S}^\circ$. Let $\kappa_{(S,\theta)}$ and $\kappa_{(S,\theta,\rho)}^{\epsilon}$ be the inflations to $G(F)_x$ of the representations $\kappa_{(\ms{S},\theta)}^{\ms{G}_x}$ and $\kappa_{(\ms{S},\theta,\rho)}^{\ms{G}_x,\epsilon}$ of Definition \ref{dfn:kappast}.

We shall consider two such tuples $(S_i,\theta_i,\rho_i,\epsilon_i)$ for $i=1,2$ equivalent if there exist $g \in G(F)$ and $\delta \in (N(S_1,G)(F)_{\theta_1}/S_1(F))^*$ such that $(S_2,\theta_2,\rho_2,\epsilon_2)=\tx{Ad}(g)(S_1,\theta_1,\rho_1\otimes\delta,\delta^{-1}\epsilon_1)$.

\begin{lem} \label{lem:gxrep}
Fix a vertex $x \in \mc{B}(G,F)$. The map $(S,\theta,\rho,\epsilon) \mapsto \kappa_{(S,\theta,\rho)}^\epsilon$ is a bijection between the set of equivalence classes of tuples $(S,\theta,\rho,\epsilon)$ s.t. $x$ is the vertex for $S$, and the set of irreducible representations of $G(F)_x/G(F)_{x,0+}$ whose restriction to $G(F)_{x,0}$ contains a non-singular cuspidal representation of $\ms{G}_x^\circ(k_F)$.
\end{lem}
\begin{proof}
Theorem \ref{thm:dlpar} and Proposition \ref{pro:inddisj} imply that the map is well-defined and injective (note that two maximally unramified elliptic maximal tori with vertex $x$ are conjugate in $G(F)$ if and only if they are conjugate in $G(F)_x$). Corollary \ref{cor:dlsurj} implies that the map is surjective, for given an ellitpic maximal torus $\ms{S}^\circ \subset \ms{G}_x^\circ$ we can find using \cite[Lemma 3.4.4]{KalRSP} a maximally unramified elliptic maximal torus $S \subset G$ with vertex $x$ whose image in $\ms{G}_x^\circ$ is $\ms{S}^\circ$, and then the image of $S(F^u)_c$ in $\ms{G}_x$ equals $\ms{S}$ as we have verified in \S\ref{sub:ass_check}.
\end{proof}

\begin{pro} \label{pro:dzclass} The representation
\[ \pi^\epsilon_{(S,\theta,\rho)}=\tx{c-Ind}_{G(F)_x}^{G(F)}\kappa^\epsilon_{(S,\theta,\rho)} \]
is irreducible and supercuspidal. Two tuples lead to isomorphic representations if and only if they are equivalent. If $\phi : G(F) \to \C^\times$ is a character of depth zero and trivial on $G_\tx{sc}(F)$, then $\chi\otimes\pi^\epsilon_{(S,\theta,\rho)}=\pi^\epsilon_{(S,\chi\cdot\theta,\chi\otimes\rho)}$. If $\delta \in (N(S,G)(F)_\theta/S(F))^*$ then $\pi^{\delta\epsilon}_{(S,\theta,\rho)}=\pi^\epsilon_{(S,\theta,\delta\otimes\rho)}$.
\end{pro}
\begin{proof}
That $\pi^\epsilon_{(S,\theta,\rho)}$ is irreducible and supercuspidal follows from \cite[Proposition 6.6]{MP96}. That equivalent triples lead to the same representation follows from Lemma \ref{lem:gxrep}. Assume conversely that two tuples $(S_i,\theta_i,\rho_i,\epsilon_i)$ for $i=1,2$ lead to isomorphic representations. By \cite[Theorem 3.5]{MP96} there exists $g \in G(F)$ s.t. if $x_1$ and $x_2$ are the vertices for $S_1$ and $S_2$, then $gx_1=x_2$ and $\tx{Ad}(g)\kappa^{\epsilon_1}_{(S_1,\theta_1,\rho_1)}=\kappa^{\epsilon_2}_{(S_2,\theta_2,\rho_2)}$. We may conjugate $(S_1,\theta_1,\rho_1,\epsilon_1)$ by $g$ to assume $x_1=x_2=:x$ and $\kappa^{\epsilon_1}_{(S_1,\theta_1,\rho_1)}=\kappa^{\epsilon_2}_{(S_2,\theta_2,\rho_2)}$, after which Lemma \ref{lem:gxrep} completes the proof. The behavior under twisting by $\chi$ and $\delta$ follows directly from Theorem \ref{thm:dlpar}.
\end{proof}

Define further
\begin{equation} \label{eq:dzpi} \pi_{(S,\theta)} := \tx{c-Ind}_{G(F)_x}^{G(F)}\kappa_{(S,\theta)}. \end{equation}
This is a supercuspidal representation of $G(F)$. When $\theta$ is regular, $\pi_{(S,\theta)}$ is irreducible, and was the subject of study in \cite{KalRSP}. Indeed, in loc. cit. we considered a representation $\kappa_{(S,\theta)}$ of $S(F)G(F)_{x,0}$ that was a certain extension of the inflation to $G(F)_{x,0}$ of the irreducible Deligne-Lusztig induction $\mc{R}_{\ms{S}^\circ}^{\ms{G}_x^\circ}(\theta^\circ)$. As explained in Remark \ref{rem:kappaext}, this extension process produces the same representation as the inflation to $S(F)G(F)_{x,0}$ of the representation $\kappa_{(\ms{S},\theta)}^{\ms{G}_x'}$ of Definition \ref{dfn:kappast} applied to the group $\ms{G}_x'=\ms{G}_x^\circ \cdot \ms{S}$. Corollary \ref{cor:dl_ind4} implies that inducing this $\kappa_{(S,\theta)}$ from $S(F)G(F)_{x,0}$ to $G(F)_x$ produces what we have called here $\kappa_{(S,\theta)}$. Hence the representation $\pi_{(S,\theta)}$ of \eqref{eq:dzpi} is the same as the representation $\pi_{(S,\theta)}$ of \cite[Lemma 3.4.20]{KalRSP}, when $\theta$ is regular.

When $\theta$ is $k_F$-non-singular, but not necessarily regular, $\pi_{(S,\theta)}$ may be reducible. We shall write $[\pi_{(S,\theta)}]$ for the set of irreducible constituents of $\pi_{(S,\theta)}$ and refer to this set as the non-singular Deligne-Lusztig packet associated to $(S,\theta)$.
From Theorem \ref{thm:dlpar} and Proposition \ref{pro:dzclass} we immediately obtain:

\begin{cor} \label{cor:dzclass1}
Let $\epsilon$ be a coherent splitting for $\{\eta_{\Psi,\ms{U}}\}$.
\begin{enumerate}
\item The irreducible constituents of $\pi_{(S,\theta)}$ are precisely the representations $\pi^\epsilon_{(S,\theta,\rho)}$, for irreducible smooth representations $\rho$ of $N(S,G)(F)_\theta$ whose restriction to $S(F)$ contains $\theta$.
\item Two such representations $\pi^\epsilon_{(S,\theta,\rho_1)}$ and $\pi^\epsilon_{(S,\theta,\rho_2)}$ are isomorphic if and only if $\rho_1 \cong \rho_2$.
\item The multiplicity of $\pi^\epsilon_{(S,\theta,\rho)}$ in $\pi_{(S,\theta)}$ equals $\tx{dim}\ \rho$.
\item The sets $[\pi_{(S_i,\theta_i)}]$ for two pairs $(S_i,\theta_i)$ are either equal or disjoint. They are equal if and only if the pairs are $G(F)$-conjugate.
\end{enumerate}
\end{cor}

\subsection{The internal structure of positive depth Deligne-Lusztig packets} \label{sub:nsdl-pd}

Let $G$ be a connected reductive group defined over $F$ and split over a tame extension of $F$. We assume that the residual characteristic $p$ of $F$ is odd, is not a bad prime for $G$ in the sense of \cite[\S4.3]{SS70}, and does not divide the order of the fundamental group of $G_\tx{der}$. If $M \subset G$ is a Levi subgroup, then $p$ satisfies the same assumptions relative to $M$.

\begin{dfn} \label{dfn:tnsep} Let $S \subset G$ be a maximal torus and $\theta : S(F) \to \C^\times$ a character. We shall call the pair $(S,\theta)$ \emph{tame $k_F$-non-singular elliptic} (resp. \emph{tame $F$-non-singular elliptic}) if
\begin{enumerate}
	\item $S$ is elliptic and its splitting extension $E/F$ is tame;
	\item Inside the connected reductive subgroup $G^0 \subset G$ with maximal torus $S$ and root system
	\[ R_{0+} = \{\alpha \in R(S,G)| \theta(N_{E/F}(\alpha^\vee(E_{0+}^\times)))=1\}, \]
	the torus $S$ is maximally unramified.
	\item The character $\theta$ is $k_F$-non-singular (resp. $F$-non-singular) with respect to $G^0$ in the sense of Definition \ref{dfn:nsc}.
\end{enumerate}
\end{dfn}

\begin{rem} A tame regular elliptic pair, in the sense of \cite[Definition 3.7.5]{KalRSP}, is a special case of a tame non-singular elliptic pair, due to Fact \ref{fct:regns}.
\end{rem}

\begin{rem} \label{rem:g0}
The subgroup $G^0$ is a tame twisted Levi subgroup, according to \cite[Lemma 3.6.1]{KalRSP}.
\end{rem}

\begin{rem} When $G=\tx{GL}_N$, all vertices are special, in fact absolutely special, and Fact \ref{fct:regns} shows that the notions of $F$-non-singular and $k_F$-non-singular coincide. The argument in the proof of \cite[Lemma 3.7.8]{KalRSP} shows furthermore that when $p \nmid N$ a tame non-singular elliptic pair is admissible in the sense of Howe. Thus for $G=\tx{GL}_N$, $p \nmid N$, the notions of non-singular, regular, extra regular, and admissible, are all equivalent. 
\end{rem}

Given a tame $k_F$-non-singular elliptic pair $(S,\theta)$ we apply \cite[Proposition 3.6.7]{KalRSP} and obtain a Howe factorization $(\phi_{-1},\dots,\phi_d)$ for $(S,\theta)$. Then $S \subset G^0$ is a maximally unramified maximal torus. Moreover \cite[Fact 3.6.4]{KalRSP} tells us that $\theta|_{S^0_\tx{sc}(F)}=\phi_{-1}|_{S^0_\tx{sc}(F)}$, from which we see that $\phi_{-1}$ is a $k_F$-non-singular character of $S(F)$ with respect to $G^0$. The failure of $\phi_{-1}$ to be regular mirrors the failure of $\theta$ to be regular, in the following sense.

\begin{lem} \label{lem:weylred} The natural inclusion $\Omega(S,G^0)(F) \to \Omega(S,G)(F)$ gives the identifications
\[ \Omega(S,G^0)(F)_{\phi_{-1}} = \Omega(S,G^0)(F)_\theta = \Omega(S,G)(F)_\theta. \]
The natural inclusion $N(S,G^0)(F) \to N(S,G)(F)$ gives the identifications
\[ N(S,G^0)(F)_{\phi_{-1}} = N(S,G^0)(F)_\theta = N(S,G)(F)_\theta. \]
\end{lem}
\begin{proof}
This follows from \cite[Lemma 3.6.5]{KalRSP}.
\end{proof}

Associated to the tame $k_F$-non-singular elliptic pair $(S,\phi_{-1})$ of $G^0$ we have the family of 2-cocycles $\eta_\Psi=\{\eta_{\Psi,\ms{U}}\}$ of Lemma \ref{lem:normiop}. This family depends only on $(S,\theta)$, but not on the Howe factorization. Indeed, by \cite[Lemma 3.6.6]{KalRSP} any two factorizations differ by a refactorization. So if $(\dot\phi_{-1},\dots,\dot\phi_d)$ is another Howe factorization, then $\dot\phi_{-1}=\phi_{-1}\cdot \chi_0$, where $\chi_0$ is a character of $G^0(F)$ of depth zero, and trivial on $G^0_\tx{sc}(F)$, and Corollary \ref{cor:etacharshift} implies the claim.

Consider a tuple $(S,\theta,\rho,\epsilon)$, where $(S,\theta)$ is a tame $k_F$-non-singular elliptic pair, $\rho$ is an irreducible smooth representation of $N(S,G)(F)_\theta$ whose restriction to $S(F)$ contains $\theta$, and $\epsilon$ is a coherent splitting for the family of 2-cocycles $\eta_\Psi$. We shall consider two such tuples $(S_i,\theta_i,\rho_i,\epsilon_i)$, $i=1,2$, equivalent if there exists $g \in G(F)$ and $\delta \in [N(S,G)(F)_\theta/S(F)]^*$ s.t. $(S_2,\theta_2,\rho_2,\epsilon_2)= \tx{Ad}(g)(S_1,\theta_1,\rho_1\otimes\delta,\delta^{-1}\cdot\epsilon_1)$.

Put $\delta_0 := \prod_{i=0}^d \phi_i^{-1} : G^0(F) \to \C^\times$, so that $\phi_{-1}=\delta_0\theta$. Then $\rho \mapsto \delta_0\otimes\rho=:\rho_{-1}$ is a bijection between the smooth irreducible representations of $N(S,G)(F)_\theta=N(S,G^0)(F)_{\phi_{-1}}$ whose restriction to $S(F)$ contains $\theta$, and those whose restriction contains $\phi_{-1}$. For such $\rho$, we have the irreducible depth-zero supercuspidal representation $\pi^\epsilon_{(G^0,S,\phi_{-1},\rho_{-1})}$ of $G^0(F)$ obtained in Proposition \ref{pro:dzclass}, and
\[ ((G^0 \subset G^1 \dots \subset G^d),\pi^\epsilon_{(G^0,S,\phi_{-1},\rho_{-1})},(\phi_0,\dots,\phi_d)), \]
is a normalized reduced generic cuspidal $G$-datum in the sense of \cite[Definition 3.7.1]{KalRSP}, leading to a supercuspidal representation of $G(F)$, which we shall denote by $\pi^\epsilon_{(S,\theta,\rho)}$. We will use here the twisted Yu construction of \cite{FKS}. This has the same effect as using the original Yu construction applied to the character $\theta\cdot\epsilon$, where $\epsilon : S(F) \to \{\pm 1\}$ is the product of the characters $\epsilon^{G^i/G^{i-1}}_x$ of \cite[Theorem 3.4]{FKS}.

\begin{pro} \label{pro:pdclass}
\begin{enumerate}
\item The representation $\pi^\epsilon_{(S,\theta,\rho)}$ depends only on $(S,\theta,\rho,\epsilon)$, and is independent of the choice of Howe factorization.
\item Two tuples $(S_i,\theta_i,\rho_i,\epsilon_i)$ produce isomorphic representations if and only they are equivalent.
\item If $\phi : G(F) \to \C^\times$ is a character trivial on $G_\tx{sc}(F)$, then $\chi\otimes\pi^\epsilon_{(S,\theta,\rho)}=\pi^\epsilon_{(S,\chi\cdot\theta,\chi\otimes\rho)}$.
\end{enumerate}
\end{pro}
\begin{proof}
Another factorization $(\dot\phi_{-1},\dots,\dot\phi_d)$ is a refactorization of $(\phi_{-1},\dots,\phi_d)$ according to \cite[Lemma 3.6.6]{KalRSP}. Using the notation of that Lemma, write $\chi_i=\prod_{j=i}^d\phi_j\dot\phi_j^{-1}=\prod_{j=-1}^{i-1}\phi_j^{-1}\dot\phi_j$, and in particular $\dot\phi_{-1}=\phi_{-1}\chi_0$. We have $\dot\delta_0=\chi_0\delta_0$ and hence $\dot\rho_{-1}=\chi_0\otimes\rho_{-1}$. By Proposition \ref{pro:dzclass} we obtain the equality $\pi^\epsilon_{(G^0,S,\dot\phi_{-1},\dot\rho_{-1})} = \chi_0 \otimes \pi^\epsilon_{(G^0,S,\phi_{-1},\rho_{-1})}$. With this, \cite[Corollary 3.5.5]{KalRSP}, which is a mild strengthening of \cite[Theorem 6.6]{HM08}, imply that the normalized generic cuspidal $G$-data for the two Howe factorizations produce the same representation of $G(F)$.

It is clear that conjugate tuples produce the same representation. Replacing $(S,\theta,\rho,\epsilon)$ by $(S,\theta,\rho\otimes\delta,\delta^{-1}\epsilon)$ replaces the depth-zero tuple $(S,\phi_{-1},\rho_{-1},\epsilon)$ by $(S,\phi_{-1},\rho_{-1}\otimes\delta,\delta^{-1}\cdot\epsilon)$. By Proposition \ref{pro:dzclass} the corresponding representation of $G^0(F)$ is unchanged, and then so is $\pi$ itself.

Now assume that two tuples $(S_i,\theta_i,\rho_i,\epsilon_i)$, $i=1,2$, produce isomorphic representations. Let $(\vec{G}_i,\vec{\phi}_i)$ be the tuples consisting of twisted Levi tower and Howe factorization of $\theta_i$, respectively. Then \cite[Theorem 6.6]{HM08} implies the existence of $g \in G(F)$ s.t. $\tx{Ad}(g)\vec{G}_2=\vec{G}_1$, $\tx{Ad}(g)(\phi_{2,0}\dots\phi_{2,d})$  is a refactorization of $(\phi_{1,0}\dots\phi_{1,d})$, and $\tx{Ad}(g)[\pi^{\epsilon_2}_{(G^0,S,\phi_{2,-1},\rho_{2,-1})}\otimes\delta_{2,0}^{-1}]=[\pi^{\epsilon_1}_{(G^0,S,\phi_{1,-1},\rho_{1,-1})}\otimes\delta_{1,0}^{-1}]$, where as before $\delta_{i,0}^{-1}$ is the product of the restrictions to $G^0_i(F)$ of $\phi_{i,0}\dots\phi_{i,d}$. We conjugate $(S_i,\theta_i,\rho_i,\epsilon_i)$ by $g$ to assume $g=1$ and then have $\pi^{\epsilon_2}_{(G^0,S,\phi_{2,-1},\rho_{2,-1})}\otimes\delta_{2,0}^{-1}\delta_{1,0}=\pi^{\epsilon_1}_{(G^0,S,\phi_{1,-1},\rho_{1,-1})}$. By \cite[Lemma 3.4.28]{KalRSP} the depth of $\delta_{2,0}^{-1}\delta_{1,0}$ is zero, so we may apply Proposition \ref{pro:dzclass} and see that the two depth-zero tuples for $G^0$ given by $(S_2,\theta_2\delta_{1,0},\rho_2\delta_{1,0},\epsilon_2)$ and $(S_1,\theta_1\delta_{1,0},\rho_1\delta_{1,0},\epsilon_1)$ are equivalent. But then so are the original tuples for $G$.

Finally, $(\phi_{-1},\phi_0,\dots,\phi_{d-1},\chi\cdot\phi_d)$ is a Howe factorization of $\chi\cdot\theta$.
\end{proof}

We define, just as in the depth-zero case, also
\begin{equation} \label{eq:pdpi}
\pi_{(S,\theta)}
\end{equation}
to be the supercuspidal representation produced by Yu's construction applied to the datum $((G^0 \subset G^1 \dots \subset G^d),\pi_{(G^0,S,\phi_{-1})},(\phi_0,\dots,\phi_d))$, where $\pi_{(G^0,S,\phi_{-1})}$ is the representation \eqref{eq:dzpi} for the group $G^0$ and the pair $(S,\phi_{-1})$. Again this representation may be reducible. We shall refer to this set of irreducible constituents as the non-singular Deligne-Lusztig packet associated to $(S,\theta)$ and write $[\pi_{(S,\theta)}]$ for it. From Proposition \ref{pro:pdclass} and Corollary \ref{cor:dzclass1} we obtain the following:

\begin{cor} \label{cor:pdclass1}
\begin{enumerate}
	\item The representation $\pi_{(S,\theta)}$ depends only on $(S,\theta)$, but not on the choice of Howe factorization.
	\item The irreducible constituents of $\pi_{(S,\theta)}$ are $\pi^\epsilon_{(S,\theta,\rho)}$ for varying smooth irreducible representations $\rho$ of $N(S,G)(F)_\theta$ whose restriction to $S(F)$ is $\theta$-isotypic.
	\item The multiplicity of $\pi^\epsilon_{(S,\theta,\rho)}$ in $\pi_{(S,\theta)}$ is equal to the dimension of $\rho$.
	\item The sets $[\pi_{(S_i,\theta_i)}]$ for two pairs $(S_i,\theta_i)$ are either equal or disjoint. They are equal if and only if the pairs are $G(F)$-conjugate.
	\item The assignment $\delta_0^{-1}\otimes\pi^\epsilon_{(S,\phi_{-1},\rho_{-1})} \mapsto \pi^\epsilon_{(S,\theta,\rho)}$ is a bijection $[\pi_{(G^0,S,\theta)}] \to [\pi_{(G,S,\theta)}]$ independent of any choices.
\end{enumerate}
\end{cor}

\subsection{Remarks on the character formula} \label{sub:char}

The representation $\pi_{(S,\theta)}$ constructed in the previous subsection is a direct sum of finitely many supercuspidal representations. Despite the fact that it is not irreducible, the material in \cite[\S4]{KalRSP} applies to it. In particular we have the formula of \cite[Corollary 4.10.1]{KalRSP} for the character values of $\pi_{(S,\theta)}$ at shallow elements of $S(F)$. However, since we are using the twisted Yu construction as in \cite{FKS}, the character formula becomes simpler. It is stated in \cite[Theorem 9.2.1]{FKS}, and we will recall it here.

Let $R(S,G)$ be the absolute root system of the maximal torus $S$. Let $\Lambda : F \to \C^\times$ be a character that is non-trivial on $O_F$ but trivial on $\mf{p}_F$. Let $\Lambda^0$ be the character of $k_F$ whose inflation to $O_F$ equals the restriction of $\Lambda$.

For each symmetric $\alpha \in R(S,G)$ consider the characters $\Lambda \circ \tx{tr}_{F_\alpha/F} : F_\alpha \rw \C^\times$ and $\theta \circ N_{F_\alpha/F} \circ \alpha^\vee : F_\alpha^\times \rw \C^\times$. The first one has depth zero, and we denote by $-r_\alpha$ the depth of the second. Define $\bar a_\alpha \in [F_\alpha]_{r_\alpha}/[F_\alpha]_{r_\alpha+}$ by the formula
\[ \theta \circ N_{F_\alpha/F} \circ \alpha^\vee(X+1) = \Lambda\circ\tx{tr}_{F_\alpha/F}(\bar a_\alpha X), \]
where $X$ is a variable in $[F_\alpha]_{-r_\alpha}/[F_\alpha]_{-r_\alpha+}$. Then $(r_\alpha,\bar a_\alpha)_\alpha$ is a set of mod-$a$-data in the sense of \cite[Definition 4.6.8]{KalRSP}. 

Recall that tame twisted Levi tower $G^0 \subset \dots \subset G^d$ associated to $(S,\theta)$. Given $\alpha \in R(S,G^{i+1}) \sm R(S,G^i)$ let $\alpha_i$ be the restriction of $\alpha$ to $Z(G^i)^\circ$. We will specify a character $\chi''_{\alpha_i} : F_{\alpha_i}^\times \to \C^\times$, by setting it trivial if $\alpha_i$ is asymmetric, the unramified quadratic character of $\alpha_i$ is unramified symmetric, and the unique tamely ramified character that extends the inflation to $O_{{\alpha_i}}^\times$ of the quadratic character of $k_{{\alpha_i}}^\times$ and has the property
\[ \chi_{\alpha_i}(2\ell(\alpha)a_\alpha) = (-1)^{f_{\alpha_i}+1}\mf{G}_{k_{\alpha_i}}(\Lambda^0). \]
Here $f_{\alpha_i}$ is the degree of the residue field extension $k_{\alpha_i}/k$, and for any finite extension $\ell/k$ we define the Gauss sum 
\[ \mf{G}_\ell(\Lambda^0) = q_l^{-1/2}\sum_{x \in \ell^\times}\tx{sgn}_{\ell^\times}(x)\Lambda^0(\tx{tr}_{\ell/k}(x)) = q_\ell^{-1/2}\sum_{x \in \ell} \Lambda^0(\tx{tr}_{\ell/k}(x^2)). \]
Finally define $\chi''_\alpha = \chi''_{\alpha_i} \circ N_{F_{\alpha}/F_{\alpha_i}}$. Then $(\chi''_\alpha)_\alpha$ is a set of tamely ramified $\chi$-data. Note that it may not be minimally ramified in the sense of \cite[Definition 4.6.1]{KalRSP}. We have the function
\[ \Delta_{II}^\tx{abs}[\bar a,\chi''] : S(F) \to \C^\times,\qquad \gamma \mapsto \prod_{\substack{\alpha \in R(S,G)/\Gamma\\\alpha(\gamma) \neq 1}} \chi_\alpha''\left(\frac{\alpha(\gamma)-1}{\bar a_\alpha}\right) \]
defined in \cite[Definition 4.6.2]{KalRSP}. Then if $\gamma \in S(F)$ is regular and shallow, \cite[Theorem 9.2.1]{FKS} states that the character of $\pi_{(S,\theta)}$ at $\gamma$ is given by
\begin{equation} \label{eq:char} 
e(G)|D_G(\gamma)|^{-\frac{1}{2}}\epsilon_L(X^*(T_0)_\C-X^*(S)_\C,\Lambda)\sum_{w \in N(S,G)(F)/S(F)}\Delta_{II}^\tx{abs}[\bar a,\chi''](\gamma^w)\theta(\gamma^w), \end{equation}
where $e(G)$ is the Kottwitz sign of $G$, $D_G(\gamma)$ is the Weyl discriminant, $T_0$ is the minimal Levi subgroup of the quasi-split inner form of $G$, $\epsilon_L$ is the Langlands normalization of the local $\epsilon$-factor. 

Note that this formula differs from \cite[Corollary 4.10.1]{KalRSP} in two essential ways. First, the auxiliary characters $\epsilon^\tx{ram}$ and $\epsilon_{f,\tx{ram}}$ are missing. Second, the $\chi$-data $\chi''$ is different from the $\chi$-data $\chi'$ used in loc. cit. We will see in the construction of $L$-packets that the $\chi$-data $\chi''$ reflects the functoriality implied by the inductive nature of Yu's construction, and is therefore better suited for the study of the local Langlands correspondence.

\section{Supercuspidal $L$-packets} \label{sec:pack}

\subsection{Factorization of parameters} \label{sub:lparam}

Let $G$ be a quasi-split connected reductive group defined over $F$ and split over a tame extension of $F$, $\hat G$ its complex dual group, $^LG$ its $L$-group. We assume that the residual characteristic $p$ of $F$ is odd, is not a bad prime for $G$ in the sense of \cite[\S4.3]{SS70}, and does not divide the order of the fundamental group of $G_\tx{sc}$. If $M \subset G$ is a Levi subgroup then $p$ is not a bad prime for $M$ and does not divide the order of the fundamental group of $M_\tx{der}$. Then the same properties hold for $\hat G$ in place of $G$.

\begin{dfn} A \emph{supercuspidal Langlands parameter} for $G$ is a discrete Langlands parameter $W_F \to {^LG}$.
\end{dfn}
In other words, it is a discrete parameter $W_F \times \tx{SL}_2(\C) \to {^LG}$ whose restriction to $\tx{SL}_2(\C)$ is trivial. It is expected that these parameters correspond precisely to those discrete series $L$-packets of $G$ that consists entirely of supercuspidal representations. This expectation was formulated in \cite[\S3.5]{DR09} and is in hindsight a special case of a more precise conjecture \cite{AMS}.

\begin{dfn} A supercuspidal parameter is called \emph{torally wild} if $\varphi(P_F)$ is contained in a maximal torus of $\hat G$.
\end{dfn}

\begin{lem} \label{lem:scpar}
Let $\varphi : W_F \to {^LG}$ be a supercuspidal parameter.
\begin{enumerate}
	\item If $p$ does not divide the order of the Weyl group of $G$, then $\varphi$ is torally wild.
	\item If $\varphi$ is torally wild, then $\tx{Cent}(\varphi(I_F),\hat G)^\circ$ is a torus.
\end{enumerate}
\end{lem}
\begin{proof}
The image $\varphi(P_F) \subset \hat G$ is a finite $p$-group, hence nilpotent, hence supersolvable. As a supersolvable group of semi-simple automorphisms of $\hat G$ it normalizes a maximal torus $\hat T \subset \hat G$, by \cite[\S II,Theorem 5.16]{SS70}. By assumption $p\nmid |\Omega(\hat T,\hat G)|$, so the image of $\varphi(P_F)$ in $\Omega(\hat T,\hat G)$ is trivial, so $\varphi(P_F)\subset \hat T$.

Let $\hat M$ be the centralizer of $\varphi(P_F)$, a Levi subgroup of $\hat G$ by \cite[Lemma 5.2.2]{KalRSP}. Let $\hat L \subset \hat M$ be the connected centralizer of $\varphi(I_F)$. It is a connected reductive group and is a torus if and only if $\hat L/Z(\hat G)^\circ$ is a torus.
We may thus replace $\hat G$ by $\hat G/Z(\hat G)^\circ$ and assume that $\hat G$ is semi-simple. This has the effect that $\hat L^{\varphi(\tx{Frob})}\subset\tx{Cent}(\varphi(W_F),\hat G)$ is finite, where $\tx{Frob} \in W_F$ is any Frobenius element. Now $\varphi(\tx{Frob})$ is a semi-simple automorphism of $\hat L$, which we decompose as a product $\tx{Ad}(l)\theta$ of an automorphism $\theta$ of $\hat L$ that preserves a pinning $(\hat T,\hat B,\{X_\alpha\})$ of $\hat L$ and an inner automorphism by an element $l$ of $\hat L$. The map
\[ \hat L \to \hat L,\qquad x \mapsto x^{-1}l\theta(x) \]
gives an isomorphism from the coset space $\hat L/\tx{Cent}(l\theta,\hat L)$ to the $\theta$-twisted conjugacy class of $l$. This $\theta$-twisted conjugacy class is an irreducible closed subvariety of $\hat L$ and the finiteness of $\tx{Cent}(l\theta,\hat L)=\hat L^{\varphi(Frob)}$ implies that its dimension is equal to that of $\hat L$. Since $\hat L$ is an irreducible variety this means that the $\theta$-twisted conjugacy class of $l$ is equal to $\hat L$. In other words, $\hat L$ is a single $\theta$-twisted conjugacy class. This is only possible of $\hat L$ is a torus, for otherwise there is a 1-1 correspondence between $\theta$-twisted conjugacy classes in $\hat L$ and $\Omega(\hat T,\hat L)^\theta$-orbits in the group $\hat T_\theta$ of $\theta$-coinvariants in $\hat T$, the latter being a non-trivial algebraic torus, see \cite[Lemma 3.2.A]{KS99}.
\end{proof}

We will now introduce the concept of torally wild $L$-packet data and show that there is a natural 1-1 correspondence between $\hat G$-conjugacy classes of torally wild Langlands parameters $W_F \to {^LG}$ and equivalence classes of such data. The data is closely related to the regular supercuspidal $L$-packet data from \cite[\S5.2]{KalRSP}, with one subtle difference.

\begin{dfn} \label{dfn:nspd}
A \emph{torally wild supercuspidal $L$-packet datum} is a tuple $(S,\hat j,\chi_0,\theta)$, where 
\begin{enumerate}
	\item $S$ is a torus of dimension equal to the absolute rank of $G$, defined over $F$ and split over a tame extension of $F$;
	\item $\hat j : \hat S \rw \hat G$ is an embedding of complex reductive groups whose $\hat G$-conjugacy class is $\Gamma$-stable; 
	\item $\chi_0 = (\chi_{\alpha_0})_{\alpha_0}$ is tamely ramified $\chi$-data for $R(S^0,G)$, as explained below;
	\item and $\theta : S(F) \rw \C^\times$ is a character. 
\end{enumerate}
subject to the condition that $(S,\theta)$ is a tame $F$-non-singular elliptic pair in the sense of Definition \ref{dfn:tnsep}.
\end{dfn}

We need to explain the notation in the third point. As discussed in \cite[\S5.1]{KalRSP} we obtain from $\hat j$ a $\Gamma$-invariant root system $R(S,G) \subset X^*(S)$. We can then define the subsystem $R_{0+} \subset R(S,G)$ as in Definition \ref{dfn:tnsep}. Let $S^0 \subset S$ be the connected component of the intersection of the kernels of all elements of $R_{0+}$ and $R(S^0,G)$ be the image of $R(S,G) \sm R_{0+}$ under the restriction map $X^*(S) \to X^*(S^0)$.

\begin{rem} \label{rem:infchi}
The difference between this definition and \cite[Definition 5.2.4]{KalRSP} is, besides requiring that $(S,\theta)$ be non-singular rather than regular, is the usage of $\chi$-data for $R(S^0,G)$ in place of $R(S,G)$. This is done to accommodate the functorial transfer from the group $G^0$ with root system $R_{0+}$ to $G$. More precisely, from $\chi_0$ one obtains $\chi$-data for $R(S,G)$, which we will denote by $\chi$, as follows. For $\alpha \in R(S,G) \sm R_{0+}$ we let $\chi_\alpha=\chi_{\alpha_0} \circ N_{F_\alpha/F_{\alpha_0}}$; any $\alpha \in R_{0+}$ is either asymmetric or unramified symmetric, and we let $\chi_\alpha$ be trivial or the unramified quadratic character, respectively. According to \cite[Proposition 6.9]{KalDC}, the $L$-embedding $^Lj_{S,G} : {^LS} \to {^LG}$ obtained from $\chi$ then factors as $^Lj_{S,G} = {^Lj}_{M,G} \circ {^Lj}_{S,M}$, where $^Lj_{M,G} : {^LM} \to {^LG}$ is obtained from the $\chi$-data $\chi_0$ as in \cite[\S6.1]{KalDC}.	
\end{rem}

\begin{dfn} \label{dfn:nspdi}
A \emph{morphism} $(S,\hat j,\chi_0,\theta) \rw (S',\hat j',\chi_0',\theta')$ \emph{of torally wild supercuspidal $L$-packet data}  is a triple $(\iota,g,\zeta_0)$, where 
\begin{enumerate}
	\item $\iota : S \rw S'$ is an isomorphism of $F$-tori;
	\item $g \in \hat G$;
	\item and $\zeta_0 = (\zeta_{\alpha'_0})_{\alpha'_0}$ is a set of $\zeta$-data for $R(S'^0,G)$ in the sense of \cite[Definition 4.6.4]{KalRSP}.
\end{enumerate}
We require that $\hat j \circ \hat \iota = \tx{Ad}(g) \circ \hat j'$, that $\chi_{\alpha_0'\circ \iota} = \chi'_{\alpha_0'} \cdot \zeta_{\alpha_0'}$, and that $\zeta_{S'}^{-1} \cdot \theta'\circ\iota=\theta$. Here we take $\zeta=\tx{inf}\zeta_0$ and take $\zeta_{S'}$ to be the character of $S'(F)$ corresponding to $\zeta$ as in \cite[Definition 4.6.5]{KalRSP}. Composition of morphisms is defined in the obvious way. 
\end{dfn}

\begin{rem}
Every morphism is an isomorphism. When $\theta$ is regular \cite[Lemma 5.2.6]{KalRSP} shows that $s \mapsto (1,\hat j(s),1)$ is an isomorphism from $\hat S$ to the group of automorphisms of $(S,\hat j,\chi_0,\theta)$. When $\theta$ is not regular this is not true, but the proof of that lemma shows that the cokernel of that map is identified with $\Omega(S,G)(F)_\theta$.	
\end{rem}

\begin{pro} \label{pro:lpdat} There is a natural 1-1 correspondence between $\hat G$-conjugacy classes of torally wild Langlands parameters for $G$ and isomorphism classes of torally wild $L$-packet data.
\end{pro}

The rest of this subsection is devoted to the proof of this proposition. The arguments are an amplification of those in the proof of \cite[Proposition 5.2.7]{KalRSP}. We will present them here in an abbreviated form and a slightly different structure, in the hope that this will help shed a better light on them.

First, we give the two inverse constructions. Starting with a torally wild $L$-packet datum $(S,\hat j,\chi_0,\theta)$ we extend $j$ to an $L$-embedding $^Lj : {^LS} \to {^LG}$ using the $\chi$-data for $R(S,G)$ obtained from $\chi_0$ as in Remark \ref{rem:infchi}, and let $\varphi={^Lj}\circ\varphi_S$, where $\varphi_S : W_F \to {^LS}$ is the Langlands parameter for the character $\theta$. In this way we obtain from the tuple $(S,\hat j,\chi_0,\theta)$ a Langlands parameter $\varphi$.

Conversely, given a torally wild parameter $\varphi : W_F \to {^LG}$ we apply \cite[Lemma 5.2.2]{KalRSP} and Lemma \ref{lem:scpar} to obtain the Levi subgroup $\hat M \subset \hat G$ and a maximal torus $\hat T \subset \hat M$, both normalized by $\varphi(W_F)$. Conjugating $\varphi$ in $\hat G$ if necessary we may arrange that $\hat T$ is part of a $\Gamma$-invariant Borel pair of $\hat G$. Then $\varphi : W_F \to N(\hat T,\hat G) \rtimes W_F$. The action of $W_F$ on $\hat T$ via $\tx{Ad}(\varphi(-))$ extends to $\Gamma_F$. We denote by $\hat S$ the corresponding $\Gamma_F$-module structure on $\hat T$, and by $\hat j : \hat S \to \hat G$ the tautological embedding $\hat T \to \hat G$. Let $S$ be the algebraic torus defined over $F$ and dual to $\hat S$. Write $R_{0+}=R(\hat S,\hat M)$ and let $S^0 \subset S$ be defined with respect to this $R_{0+}$. Choose tame $\chi$-data $\chi_0$ for $R(S^0,G)$, obtain from it $\chi$-data for $R(S,G)$ as in Remark \ref{rem:infchi} and use it to extend $\hat j$ to an $L$-embedding ${^Lj_{\chi_0}} : {^LS} \to {^LG}$. The parameter $\varphi$ factors through this embedding as $\varphi={^Lj_{\chi_0}} \circ \varphi_{S,\chi_0}$ for $\varphi_{S,\chi_0} : W_F \to {^LS}$. We let $\theta_{\chi_0} : S(F) \to \C^\times$ be the corresponding character. In this way we obtain from $\varphi$ the tuple $(S,\hat j,\chi_0,\theta_{\chi_0})$.

This concludes the description of the two constructions. The proofs that the isomorphism class of the tuple $(S,\hat j,\chi_0,\theta_{\chi_0})$ produced from $\varphi$ depends only on the $\hat G$-conjugacy class of $\varphi$, and conversely that the $\hat G$-conjugacy class of the parameter $\varphi$ produced from a tuple $(S,\hat j,\chi_0,\theta)$ depends only on the isomorphism class of that tuple, are very similar to the ones given in the proof of \cite[Proposition 5.2.7]{KalRSP}. They are routine and we will not repeat them. Moreover, the fact that the two constructions are inverse to each other is clear.

What remains to be checked is that the tuple $(S,\hat j,\chi_0,\theta_{\chi_0})$ produced from a torally wild $\varphi$ is a torally wild $L$-packet datum, and conversely that the parameter $\varphi$ produced from a torally wild $L$-packet datum $(S,\hat j,\chi_0,\theta)$ is torally wild.

We begin by noting that, given a torally wild parameter $\varphi$, the definition of $\hat M$ implies that $\varphi(P_F) \subset Z(\hat M) \subset \hat T$, so the $\Gamma$-module $\hat S$ is tame. Moreover, \cite[Lemma 5.2.2]{KalRSP} implies that $\varphi(I_F)$ preserves a Borel subgroup of $\hat M$ containing $\hat T$, so the action of $I_F$ on $R(\hat S,\hat M)$ preserves a positive chamber. 

\begin{lem} \label{lem:r0id}
Under the identification $R(\hat S,\hat G)=R^\vee(S,G)$ the root system $R(\hat S,\hat M)$ is identified with the coroot system of the root system $R_{0+}$ of Definition \ref{dfn:tnsep}.
\end{lem}
\begin{proof}
We have $R(\hat S,\hat M) = \{\hat\alpha \in R(\hat S,\hat G)| \hat\alpha(\varphi(P_F))=1\}$. For any $\hat\alpha \in R(\hat S,\hat G)$ let $\alpha^\vee \in R^\vee(S,G)$ be the corresponding cocharacter. Letting $E/F$ be the tame Galois extension splitting $S$, the parameter of the character $\theta\circ N_{E/F}\circ \alpha^\vee$ is equal to the restriction to $W_E$ of $\hat\alpha\circ\varphi_S$. The tameness of the $\chi$-data implies that $\varphi|_{P_E}=\varphi_S|_{P_E}$. Since $P_F=P_E$ we see using \cite[Theorem 7.10]{Yu09} that $R(\hat S,\hat M)$ is the subset of $R^\vee(S,G)$ consisting of those $\alpha^\vee$ for which $\theta\circ N_{E/F}\circ\alpha^\vee$ restricts trivially to $E_{0+}^\times$, as claimed.
\end{proof}

\begin{lem} \label{lem:lpns} Let $(S,\hat j,\chi_0,\theta)$ be a tuple as in Definition \ref{dfn:nspd}, but without assuming that $(S,\theta)$ is a tame non-singular elliptic pair. Instead we only assume that $S$ is tame and maximally unramified in $G^0$. Let $\varphi = {^Lj_{\chi_0}} \circ \varphi_S$. Then $\theta$ is $F$-non-singular with respect to $G^0$ if and only if $\hat M^{\varphi(I_F),\circ}$ is a torus.
\end{lem}

Granting this lemma, we complete the proof of Proposition \ref{pro:lpdat} as follows. Since $\varphi$ is torally wild Lemma \ref{lem:scpar} implies that $\hat M^{\varphi(I_F),\circ}=\tx{Cent}(\varphi(I_F),\hat G)^\circ$ is a torus, so $\theta$ is $F$-non-singular with respect to $G^0$ by Lemma \ref{lem:lpns}. Furthermore, $\hat j$ identifies $\hat S^{\Gamma,\circ}$ with $\hat T^{\varphi(W_F),\circ} \subset \hat M^{\varphi(W_F),\circ}=\tx{Cent}(\varphi(W_F),\hat G)^\circ$. The discreteness of $\varphi$ implies that $\hat S^\Gamma/Z(\hat G)^\Gamma$ is finite, thus $S/Z(G)$ is anisotropic.

Conversely, starting with a torally wild $L$-packet datum $(S,\hat j,\chi_0,\theta)$ we have $\tx{Cent}(\varphi(I_F),\hat G)^\circ = \hat M^{\varphi(I_F),\circ}$, which is a torus by Lemma \ref{lem:lpns}. By \cite[Lemma 5.2.2]{KalRSP} it normalizes $\hat T$, so by rigidity of tori it centralizes $\hat T$, but since $\hat T$ is maximal it must then lie inside of $\hat T$. Thus $\tx{Cent}(\varphi(W_F),\hat G)^\circ \subset \hat T^{\varphi(W_F),\circ} = \hat S^{\Gamma,\circ}$. But $S/Z(G)$ is anisotropic, so $\hat S^\Gamma/Z(\hat G)^\Gamma$ is finite, so $\varphi$ is discrete.

The proof of Proposition \ref{pro:lpdat} is thus reduced to the proof of Lemma \ref{lem:lpns}. As a preparation for that, we take a closer look at the construction of the $L$-embedding ${^Lj_{\chi_0}} : {^LS} \to {^LG}$ of \cite[\S2.5,\S2.6]{LS87}. Recall that its $\hat G$-conjugacy class is uniquely determined by the $\chi$-data. Let $(\hat T,\hat B,\{X_{\hat\alpha}\}_{\hat\alpha \in \Delta^\vee})$ be a $\Gamma$-invariant pinning of $\hat G$. Assume that $\varphi(W_F)$ normalizes $\hat T$, as \cite[Lemma 5.2.2]{KalRSP} allows. The equation $\varphi = {^Lj_{\chi_0}} \circ \varphi_{S,\chi_0}$ specifies $^Lj_{\chi_0}$ further up to $\hat T$-conjugacy. 

When $G=G^0$ then the datum $\chi_0$ is empty and the $\chi$-datum for $R(S,G)$ of Remark \ref{rem:infchi} is the unique unramified datum, in which $\chi_\alpha$ is the unramified quadratic character when $\alpha$ is symmetric, and trivial when $\alpha$ is asymmetric. We shall denote by $^Lj_0$ the corresponding $L$-embedding.

\begin{lem} \label{lem:ljx} Assume that $G=G^0$. There exists a representative of the $\hat T$-conjugacy class of $^Lj_0$ so that for all $x \in I_F$
\[ {^Lj_0}(1 \rtimes x) = 1 \rtimes x. \]
\end{lem}
\begin{proof} In order to obtain a representative of the $\hat T$-conjugacy class we follow \cite[\S2.5]{LS87} and make the following choices: First we choose one representative $\hat\alpha \in R(\hat S,\hat G)_\tx{sym}$ within each $\Gamma$-orbit. We make sure $\hat\alpha>0$. For each such chosen $\hat\alpha$ we choose a set of representatives $w_1,\dots,w_n \in W_F$ for the quotient $\Gamma_{\pm\hat\alpha} \lmod \Gamma$, making sure that $w_i^{-1}\hat\alpha>0$. Choose also $v_1 \in W_{\pm\hat\alpha} \sm W_{\hat\alpha}$. Set $v_0=1$. Note that these choices make the gauge $p$ of \cite[\S2.5]{LS87} be given by $p(\beta)=1 \Leftrightarrow \beta>0$. 

We now obtain the representative ${^Lj_0}(s \rtimes w) = \hat j(s) \cdot r_p(\sigma) \cdot n_w \rtimes w$. Here $n_w \in N(\hat T,\hat G)$ is the Tits lift of $\omega_w \in \Omega(\hat T,\hat G)$, and $\omega_w \rtimes \sigma$ is the action on $\hat T$ by $\tx{Ad}(\varphi(w))$. Furthermore, $r_p : W_F \to \hat T$ is defined by
\[ r_p(w) = \prod_{\hat\alpha} \prod_{i=1}^n \hat\alpha^\vee(\chi_{\hat\alpha}(v_0(u_i(w)))). \]
Here $u_i(w) \in W_{\pm\hat\alpha}$ is defined by $w_iw=u_i(w)w_{i'}$, with $i' \in \{1,\dots,n\}$ the unique possible index, and $v_0(u) \in W_{\hat\alpha}$ for $u \in W_{\pm\hat\alpha}$ by $u=v_0(u)v_{i'}$, where $i' \in
\{0,1\}$ is again the unique possible index.

As is shown in \cite[\S2.5]{LS87} making different choices changes the 1-cochain $r_p$ up to 1-coboundaries, and hence $^Lj_0$ up to $\hat T$-conjugacy. We will now show how to make the choices so as to obtain $r_p(x)=1$ for $x \in I_F$. The assumption $G=G^0$ implies $\omega_x=1$ and hence $n_x=1$, so the lemma will be proved.

What we want is for the contribution of the unramified symmetric roots to vanish when $w=x$. Thus let $\hat\alpha$ be an unramified symmetric root. Let $I_{\hat\alpha}$ and $I_{\pm\hat\alpha}$ be the intersections of $I_F$ with $\Gamma_{\hat\alpha}$ and $\Gamma_{\pm\hat\alpha}$ respectively. We have $I_{\hat\alpha}=I_{\pm\hat\alpha}$ as $\hat\alpha$ is symmetric unramified. We choose representatives $\tau_j \in W_F$ for the coset space $I_F \cdot \Gamma_{\pm\hat\alpha} \lmod \Gamma$, again maintaining $\tau_i^{-1}\hat\alpha>0$, as well as representatives $\delta_j \in I_F$ of the coset space $I_{\pm\hat\alpha} \lmod I_F$. Then $\{\delta_j\tau_i\}$ is a set of representatives for $\Gamma_{\pm\hat\alpha} \lmod \Gamma$. 

We claim that $(\delta_j\tau_i)^{-1}\hat\alpha>0$. Indeed, this equals $(\tau_i^{-1}\delta_j^{-1}\tau_i)\tau_i^{-1}\hat\alpha$. By construction $\tau_i^{-1}\hat\alpha>0$. Moreover, $\tau_i^{-1}\delta_j^{-1}\tau_i \in I_F$ and the assumption $G=G^0$ means that the action of $I_F$ on $R(\hat S,\hat G)$ preserves the set of positive roots.

The claim we just proved means that we can take $\{\delta_j\tau_i\}$ as the set $w_1,\dots,w_n$ of representatives above. For $w=x \in I_F$ we then have $u_{ij}(x)=\delta_j \tau_i x \tau_{i'}^{-1}\delta_{j'}^{-1}$. One now observes that $i'=i$, so $u_{ij}(x)$ is an element of $\Gamma_{\pm\hat\alpha} \cap I_F = I_{\pm\hat\alpha}=I_{\hat\alpha}$. Hence $v_0(u_{ij}(x))=u_{ij}(x) \in I_{\hat\alpha}$. But $\chi_{\hat\alpha}$ is unramified, so $\chi_{\hat\alpha}(v_0(u_{ij}(x)))=1$.
\end{proof}

\begin{rem} It may be tempting to drop the assumption $G=G^0$ in the above lemma and assert that one can arrange the choices so that the cochain $r_p$ only receives contributions from the ramified symmetric roots. That is, there is a representative of that $\hat T$-conjugacy class of $^Lj_\chi$ so that for all $x \in I_F$
\[ {^Lj_\chi}(1 \rtimes x) = \left(\prod_{\hat\alpha \in R(\hat S,\hat G)_{sym,ram}/\Gamma}\hat\alpha^\vee(z_i(x))\right) n_x \rtimes x, \]
where $z_i(x)$ are complex numbers and $n_x$ is the Tits lift of $\omega_x$ and $\omega_x \rtimes x$ is the action of $\tx{Ad}(\varphi(x))$ on $\hat T$. Note however that the proof will not go through, because there is no guarantee that $(\tau_i^{-1}\delta_j^{-1}\tau_i)$ will send the positive root $\tau_i^{-1}\hat\alpha$ to a positive root. And indeed, this generalization is false. This is the Galois\-theoretic expression of the fact mentioned in \S\ref{sub:char} that the character $\theta'$ need not be non-singular even if $\theta$ is.
\end{rem}

\begin{proof}[Proof of Lemma \ref{lem:lpns}]
We first consider the special case that $\hat G=\hat M$. We have $\varphi : W_F \to N(\hat T,\hat G) \rtimes W_F$. The action of $\tx{Ad}(\varphi(I_F))$ preserves a Borel subgroup of $\hat G$ containing $\hat T$. Upon further conjugating $\varphi$ we may arrange that this Borel subgroup is the chosen one $\hat B$. This implies $\varphi(I_F) \subset \hat T \rtimes I_F$. In particular, all symmetric roots in $R(S,G)$ are unramified and our $\chi$-data consists of unramified quadratic characters.

According to Lemma \ref{lem:ljx} we can arrange that for all $x \in I_F$ we have $^Lj_\chi(1 \rtimes x)=1 \rtimes x$. This means that if $\varphi_S(x)= s \rtimes x$ then $\varphi(x)=\hat j(x) \rtimes x$. We now specify $x \in I_F$ to be a lift of a topological generator of $I_F/P_F$ and let $t \in \hat T$ be determined by $\varphi(x)= t \rtimes x$. Thus $\varphi_S(x)=s \rtimes x$ and $t=\hat j(s)$.

Write again $\hat L = \tx{Cent}(\varphi(I_F),\hat G)^\circ$. Then $\hat L$ is a connected reductive group with maximal torus $\hat T^{x,\circ}$. To determine its root system, we following \cite[\S1.3]{KS99} and consider the relative root system $R_\tx{res}(\hat T^{x,\circ},\hat G)$. We subdivide its elements into types R1/R2/R3 as follows: $\hat\alpha_\tx{res} \in R_\tx{res}(\hat T^{x,\circ},\hat G)$ is of type R1 if it is neither divisible nor multipliable, of type R2 if it is multipliable, and of type R3 if it is divisible. Types R2 and R3 occur only if $\hat G$ has a component of Dynkin type $A_{2n}$ and a power of $\varphi(x)$ preserves and acts non-trivially on this component. In that case, they occur together: the restriction of $\hat\alpha \in R(\hat T,\hat G)$ to $\hat T^{x,\circ}$ is of type R2 if and only if the smallest $l \in \N$ such that $x^l\hat\alpha=\hat\alpha$ is even and $\hat\beta = \hat\alpha+x^{l/2}\hat\alpha$ is also a root. Then the restriction of $\hat\beta$ to $T^{x,\circ}$ is of type R3, and every relative root of type R3 occurs this way.

An element $\hat\alpha_\tx{res} \in R_\tx{res}(\hat T^{x,\circ},\hat G)$ belongs to the root system $R(\hat T^{x,\circ},\hat L)$ if and only if either $\hat\alpha_\tx{res}$ is of type R1 or R2 and $N\hat\alpha(t)=1$, or $\hat\alpha_\tx{res}$ is of type R3 and $N\hat\alpha(t)=-1$. Here $\hat\alpha \in R(\hat T,\hat G)$ is any root restricting to $\hat\alpha_\tx{res}$ and $N\hat\alpha$ is the sum of the members of the $x$-orbit of $\hat\alpha$.

We now consider dually $R(S,G)$ and the relative root system $R_\tx{res}(S',G)$. The bijection $R(S,G) \to R(\hat S,\hat G)$ given by
$\alpha \mapsto \alpha^\vee=\hat\alpha$ induces a type-preserving bijection $R_\tx{res}(S',G) \to R_\tx{res}(\hat S^{I,\circ},\hat G)$. If $\hat\alpha_\tx{res}$ is of type R1 or R3, the coroot of $\alpha_\tx{res}=(\hat\alpha^\vee)_\tx{res}$ is $N\alpha^\vee=N\hat\alpha$. And if $\hat\alpha_\tx{res}$ is of type R2 the corresponding coroot is $2N\alpha^\vee=2N\hat\alpha$.

We now relate this to Definition \ref{dfn:nsc}. Let $F'/F$ be an unramified extension splitting $S'$. The Langlands parameter of the character $\theta\circ N_{F'/F}$ of $S'(F')$ is the composition
\[ \xymatrix{
	W_{F'}\ar[r]&W_F\ar[r]^-{\varphi_S}&\hat S \rtimes W_F\ar[r]&\hat S_{I_F} \rtimes W_F
}
\]
where the first map is the natural inclusion and the last map is the natural projection. For $\alpha_\tx{res} \in R_\tx{res}(S',G)$ the dual of the $F'$-rational homomorphism $\alpha_\tx{res}^\vee : \mb{G}_m \to S'$ is the homomorphism $\hat S_{I_F} \times W_{F'} \to \hat \C^\times$ that is trivial on $W_{F'}$ and given by the factorization of $kN\hat\alpha : \hat S \to \C^\times$ to $\hat S_{I_F}$, where $k=1$ if $\alpha_\tx{res}$ is of type R1 or R3, and $k=2$ if $\alpha_\tx{res}$ is of type R2. Thus the Langlands parameter of $\theta\circ N_{F'/F} \circ \alpha_\tx{res}^\vee$ is $(kN\hat\alpha)\circ \varphi_S|_{W_{F'}}$ and the character $\theta\circ N_{F'/F} \circ \alpha_\tx{res}^\vee$ is trivial on $O_{F'}^\times$ if and only if its parameter is has trivial restriction to $I_{F'}=I_F$.

Let $\hat\alpha_\tx{res}$ be of type R1. Then $\hat\alpha_\tx{res}$ occurs in the root system of $\hat L$ if and only if $N\hat\alpha(t)=1$, which is equivalent to the triviality of $(N\hat\alpha)\circ \varphi_S|_{I_F}$.

Let $\hat\alpha_\tx{res}$ be of type R2. Choose a lift $\hat\alpha \in R(\hat S,\hat G)$ and let $\hat\beta=\hat\alpha+x^{l/2}\hat\beta$ be as above, so that $\hat\beta_\tx{res}$ is of type R3. Note that $N\hat\alpha=N\hat\beta$. Then $\hat\alpha_\tx{res}$ occurs in the root system of $\hat L$ if and only if $N\hat\alpha(t)=1$ and $\hat\beta_\tx{res}$ occurs in that root system if and only if $N\hat\beta(t)=-1$. Now $\alpha_\tx{res}^\vee=2N\hat\alpha$ and $\beta_\tx{res}^\vee=N\hat\beta$. Thus $\hat\alpha_\tx{res}$ occurs in the root system of $\hat L$ if and only if $\theta\circ N_{F'/F}\circ\beta_\tx{res}^\vee$ has trivial restriction to $O_{F'}^\times$, while $\hat\beta_\tx{res}$ occurs in that root system if and only if $\theta\circ N_{F'/F}\circ\alpha_\tx{res}^\vee$ has trivial restriction to $O_{F'}^\times$. This completes the proof in the case $\hat M=\hat G$.

We now turn to the general case. By construction $\varphi(W_F)$ normalizes $\hat M$, $\hat T$, and in addition $\varphi(I_F)$ normalizes a Borel subgroup of $\hat M$ containing $\hat T$, which we can arrange to be $\hat B \cap \hat M$. The pinning of $\hat M$ inherited from the chosen pinning of $\hat G$ gives a section $\tx{Out}(\hat M) \to \tx{Aut}(\hat M)$. We compose $\varphi : W_F \to N(\hat M,{^LG}) \to \tx{Aut}(\hat M) \to \tx{Out}(\hat M)$ with this section and obtain a new homomorphism $W_F \to \tx{Aut}(\hat M)$. It extends to $\Gamma_F$ and induces an action of $\Gamma_F$ on $\hat M$ preserving the pinning. Let $M$ be the quasi-split $F$-group whose dual group is $\hat M$ with this pinned $\Gamma$-action. 

The $\chi$-data for $R(S^0,G)$ leads to an embedding ${^Lj_M} : {^LM} \to {^LG}$ by the construction of \cite[\S6.1]{KalDC}. The image of this $L$-embedding contains the image of $\varphi$ which leads to a factorization $\varphi = {^Lj_M} \circ \varphi_M$ for a parameter $\varphi_M : W_F \to {^LM}$. The natural inclusion restricts to an isomorphism $\tx{Cent}(\varphi_M,\hat M) \to \tx{Cent}(\varphi,\hat G)$. We conclude that $\varphi_M$ is a torally wild supercuspidal parameter. We apply the established special case $\hat M=\hat G$ to the parameter $\varphi_M$. Thus we have the $L$-embedding $^Lj_{S,M} : {^LS} \to {^LM}$, obtained from unramified $\chi$-data, since $S$ is maximally unramified with respect to $M$, and the factorization $\varphi_M = {^Lj_{S,M}} \circ \varphi_{S,M}$ for a parameter $\varphi_{S,M} : W_F \to {^LS}$. The character $\theta_M : S(F) \to \C^\times$ corresponding to $\varphi_{S,M}$ is $F$-non-singular by the previously handled case.

According to \cite[\S6.2]{KalDC}, the composition ${^Lj_{S,M}} \circ {^Lj_M}$ is equal to the $L$\-embedding $^Lj_{S,G} : {^LS} \to {^LG}$ obtained by making $(\chi_{\alpha_0})$ into $\chi$-data for $R(S,G) \sm R(S,G^0)$ and complementing it with unramified $\chi$-data for $R(S,G^0)$. Therefore the parameters $\varphi_S$ and $\varphi_{S,M}$ are $\hat S$-conjugate, so $\theta=\theta_M$.
\end{proof}

\subsection{Construction of the $L$-packet} \label{sub:lpackconst}

Let $(S,\hat j,\chi_0,\theta)$ be a torally wild $L$-packet datum. We write down a formula for a function $\Theta : S(F)_\tx{reg} \to \C$ just as in \cite[(5.3.1)]{KalRSP}: We choose a non-trivial character $\Lambda : F \to \C^\times$ and for each $\alpha \in R(S,G)$ we define $\bar a_\alpha \in [F_\alpha]_{(r_{\Lambda,\alpha}-r_{\theta,\alpha})}/[F_\alpha]_{(r_{\Lambda,\alpha}-r_{\theta,\alpha})+}$, by the formula
\[ \theta \circ N_{F_\alpha/F} \circ \alpha^\vee(X+1) = \Lambda\circ\tx{tr}_{F_\alpha/F}(\bar a_\alpha X), \]
where $r_{\Lambda,\alpha}$ and $r_{\theta,\alpha}$ are the depths of the characters $\Lambda \circ \tx{tr}_{F_\alpha/F} : F_\alpha \rw \C^\times$ and $\theta \circ N_{F_\alpha/F} \circ \alpha^\vee : F_\alpha^\times \rw \C^\times$ respectively, and $X$ is a variable in $[F_\alpha]_{r_{\theta,\alpha}}/[F_\alpha]_{r_{\theta,\alpha}+}$. Then
\begin{equation} \label{eq:charconj} \Theta(\gamma) := \epsilon_L(X^*(T)_\C - X^*(S)_\C,\Lambda)\Delta_{II}^\tx{abs}[\bar a,\chi](\gamma)\theta(\gamma).\end{equation}
By \cite[Lemma 5.3.1]{KalRSP}, this function depends only on the isomorphism class of $(S,\hat j,\chi_0,\theta)$.

As explained in \cite[\S5.1]{KalRSP},  the map $\hat j$ gives rise to a unique stable conjugacy class of embeddings $S \to G$, called admissible, which identify $S$ with a maximal torus of $G$ defined over $F$. For every inner twist $\xi : G \to G'$ we obtain by composition a stable conjugacy class of embeddings $S \to G'$ with the same property, also called admissible.

To each admissible embedding $j : S \to G'$ we shall now assign a non-singular Deligne-Lusztig packet $[\pi_{(jS,\theta_j)}]$ by following the guiding principle that the formula for Harish-Chandra character of the (possibly reducible) supercuspidal representation $\pi_{(jS,\theta_j)}$ defined by \eqref{eq:pdpi} is given on shallow elements $\gamma' \in jS(F)$ by the function
\begin{equation} \label{eq:charwish} e(G')|D_{G'}(\gamma')|^{-\frac{1}{2}}\sum_{w \in \Omega(jS(F),G'(F))} \Theta(j^{-1}(\gamma'^w)). \end{equation}

This can be done explicitly as follows. Let $(\chi''_\alpha)_\alpha$ be the $\chi$-data for $R(S,G)$ computed in terms of $\theta$ as reviewed in \S\ref{sub:char}. Since both $\chi$ and $\chi''$ agree on $R(S,G^0)$ and on $R(S,G) \sm R(S,G^0)$ are inflated (cf. \cite[Definition 5.14]{KalDC}) from $R(Z(G^0)^\circ,G)$, we can modify $(S,\hat j,\chi_0,\theta)$ within its isomorphism class to ensure $\chi=\chi''$. Note that since this modification changes $\theta$ only by a depth-zero character, while $\chi''$ is computed in terms of $S(F)_{0+}$, this is a well-defined procedure. We now take $\theta_j=j_*\theta$. It now follows from \eqref{eq:char} that the character of $\pi_{(jS,j_*\theta)}$ on shallow elements of $jS(F)$ is given by \eqref{eq:charwish}.

The resulting non-singular Deligne-Lusztig packet $[\pi_{(jS,\theta_j)}]$ is uniquely determined by the isomorphism class of $(S,\hat j,\chi_0,\theta)$ and the admissible embedding $j : S \to G'$.  We define the $L$-packet $\Pi_\varphi(G')$ as the disjoint union of the non-singular Deligne-Lusztig packets $[\pi_{(jS,\theta_j)}]$ for all $G'(F)$-conjugacy classes of admissible embeddings $j$.

We will now put together the individual $L$-packets $\Pi_\varphi(G')$ into a compound $L$-packet $\Pi_\varphi$ encompassing all rigid inner forms of $G$. The construction is the same as in \cite[\S5.3]{KalRSP}.
We introduce the notion of a non-singular Deligne-Lusztig packet datum. It is a tuple $(S,\hat j,\chi,\theta,(G',\xi,z),j)$, where $(S,\hat j,\chi,\theta)$ is a torally wild $L$-packet datum, $(G',\xi,z)$ is a rigid inner twist of $G$ in the sense of \cite[\S5.1]{KalRI}, and $j : S \rw G'$ is an admissible embedding defined over $F$. We organize these data into a category, where a morphism 
\[ (S_1,\hat j_1,\chi_1,\theta_1,(G'_1,\xi_1,z_1),j_1) \rw (S_2,\hat j_2,\chi_2,\theta_2,(G'_2,\xi_2,z_2),j_2) \] 
is given by $(\iota,g,\zeta,f)$, where $(\iota,g,\zeta)$ is an isomorphism of the underlying regular torally wild $L$-packet data, $f : (G'_1,\xi_1,z_1) \rw (G_2',\xi_2,z_2)$ is an isomorphism of rigid inner twists, and $j_2\circ \iota = f \circ j_1$. There is an obvious forgetful functor from the category of non-singular Deligne-Lusztig packet data to the category of torally wild $L$-packet data. If we fix a torally wild $L$-packet datum $(S,\hat j,\chi,\theta)$, the set of isomorphism classes of non-singular Deligne-Lusztig packet data mapping to it is a torsor under $H^1(u \rw W,Z(G) \rw S)$. This torsor is given by the relation
\begin{equation} \label{eq:jtors} x \cdot (G_1',\xi_1,z_1,j_1) = (G_2',\xi_2,z_2,j_2) \Leftrightarrow x = \tx{inv}(j_1,j_2),
\end{equation}
see \cite[\S5.1]{KalRI}.

To each non-singular Deligne-Lusztig packet datum $(S,\hat j,\chi_0,\theta,(G',\xi,z),j)$ we associate the corresponding non-singular Deligne-Lusztig packet $[\pi_{(jS,\theta_j)}]$ on the group $G'(F)$. The compound packet $\Pi_\varphi$ is then defined as the union of the sets $\{(G',\xi,z,\pi)| \pi \in [\pi_{(jS,\theta_j)}]\}$, as $(S,\hat j,\chi_0,\theta,(G',\xi,z),j)$ runs over the isomorphism classes of non-singular Deligne-Lusztig packet data that map to the isomorphism class of the torally wild packet datum $(S,\hat j,\chi_0,\theta)$.

It is possible that two non-isomorphic non-singular Deligne-Lusztig packet data
$(S,\hat j,\chi_0,\theta,(G',\xi,z),j_i)$, $i=1,2$ give the same non-singular Deligne-Lusztig packet. By Corollary \ref{cor:pdclass1} this happens if and only if there is $w \in \Omega(S,G)(F)_\theta$ such that $(S,\hat j,\chi_0,\theta,(G',\xi,z),j_2\circ w)$ is isomorphic to $(S,\hat j,\chi_0,\theta,(G',\xi,z),j_1)$.

\begin{lem} \label{lem:dz_genconst}
Assume $\varphi$ has depth zero. For any given Whittaker datum $\mf{w}$ of $G$ there exists a unique $\mf{w}$-generic member of $\Pi_\varphi$.
\end{lem}
\begin{proof}
By Lemma \ref{lem:dzgenchar}, $\mf{w}$ determines a absolutely special vertex $x \in \mc{B}(G,F)$, unique up to $G(F)$-conjugacy, s.t. $\psi$ has depth zero at $x$ for some $(B,\psi) \in \mf{w}$. According to Proposition \ref{pro:dzgen} a depth-zero supercuspidal representation is $\mf{w}$-generic if and only if it is induced from an irreducible representation of $G(F)_x$ containing a $\psi_x$-generic cuspidal representation of $G(F)_{x,0}$. By \cite[Lemmas 3.4.12]{KalRSP} there exists precisely one admissible embedding $j : S \to G$ up to $G(F)$-conjugacy such that the vertex $x$ corresponds to the maximal torus $j(S) \subset G$. A $\mf{w}$-generic member of $\Pi_\varphi$ can thus only come from the non-singular Deligne-Lusztig packet $[\pi_{(jS,\theta_j)}]$. By \cite[Proposition 3.10]{DLM92}, there exists a unique $\psi_x$-generic irreducible component of the Deligne-Lusztig character $\kappa_{(jS,\theta^\circ_j)}$, and hence a unique irreducible representation of $G(F)_x$ containing it upon restriction. Its compact induction to $G(F)$ is then the unique $\mf{w}$-generic element of $[\pi_{(jS,\theta_j)}]$.
\end{proof}

\begin{rem} We expect that the case of positive depth can be reduced to the case of depth zero using the local character expansions of \cite{Spice17}.
\end{rem}

\subsection{Study of the centralizer $S_\varphi$} \label{sub:sphi}

Before we can construct a bijection between the compound $L$-packet $\Pi_\varphi$ constructed in \S\ref{sub:lpackconst} and $\tx{Irr}(\pi_0(S_\varphi^+))$, we need to have a better understanding of the group $S_\varphi$ and its cover $S_\varphi^+$. Contrary to the case of regular supercuspidal parameters, where $S_\varphi$ is always abelian and in fact canonically isomorphic to $\hat S^\Gamma$, this is no longer true for arbitrary supercuspidal parameters, even for classical Dynkin types, as the following example shows.

\begin{exa} We consider the split group $G=\tx{Spin}_9$, which is of type $B_4$. Its dual group is $\hat G=\tx{PSp}_8(\C)$. We shall produce a discrete parameter of depth zero $\varphi : W_F/P_F \to \hat G$ such that
\[ S_\varphi \cong [(\mu_4)^4_2/\mu_2] \rtimes (\Z/2\Z), \]
where $(\mu_4)^4_2$ is the subgroup of $\mu_4^4$ consisting of those $(t_1,t_2,t_3,t_4)$ satisfying $t_1^2=t_2^2=t_3^2=t_4^2$, $\mu_2$ is embedded diagonally in that subgroup, and $\Z/2\Z$ acts on it by sending $(t_1,t_2,t_3,t_4)$ to $(t_4,t_3,t_2,t_1)$. For this, we assume that $q=|k_F|>8$, choose a primitive $(q+1)$-root of unity $\zeta \in \C^\times$, and consider the matrices
\[
\begin{bmatrix}&&&&&&&1\\ &&&&&&-1&\\ &&&&&1&&\\ &&&&-1&&&\\ &&&1&&&&\\ &&-1&&&&&\\ &1&&&&&&\\ -1&&&&&&&
\end{bmatrix}\ ,\
\begin{bmatrix} &&&-1&&&&\\ &&-1&&&&&\\ &1&&&&&&\\ 1&&&&&&&\\ &&&&&&&-1\\ &&&&&&-1&\\ &&&&&1&&\\ &&&&1&&&\\
\end{bmatrix}
\]
which we call $j$ and $n$, respectively, and let $t$ be the diagonal matrix with diagonal entries $(\zeta,\zeta^2,-\zeta^2,-\zeta,-\zeta^{-1},-\zeta^{-2},\zeta^{-2},\zeta^{-1})$. We realize $\tx{Sp}_8$ as the subgroup of $\tx{GL}_8$ that preserves the symplectic form given by $j$, i.e. the subgroup of matrices $g$ satisfying $j^{-1}\cdot g^T\cdot j=g^{-1}$. Then the matrices $j,n,t$ all belong to $\tx{Sp}_8(\C)$. The element $t$ is regular semi-simple. Its image in $\hat G$ is not strongly regular, because it commutes with $n$. In fact, $n$ generates the stabilizer of $t$ in the Weyl group of the diagonal maximal torus in $\hat G$. In $\hat G$ we have $j\cdot t \cdot j^{-1}=t^q$, $jn = nj$, $nt= tn$, and $n^2=1$. As before we let $x \in I_F/P_F$ be a topological generator, and choose a Frobenius element $y \in W_F/P_F$. We define $\varphi(x)=t$ and $\varphi(y)=j$. Then $\hat S^\Gamma = \hat T^j$ is the 2-torsion subgroup of $\hat T$ and is thus canonically isomorphic to $(\mu_4)^4_2/\mu_2$. The element $n$ also belongs to $S_\varphi$. It projects onto a generator of $\Omega(S,G)(F)_\theta \cong \Z/2\Z$ and acts on $\hat S^\Gamma$ as stated. \hfill$\sslash$
\end{exa}

We consider the following functors from the category of torally wild $L$-packet data to the category of groups:
\begin{enumerate}
	\item $(S,\hat j,\chi,\theta) \mapsto S_\varphi$, where $\varphi := {^Lj} \circ \varphi_S$, ${^Lj} : {^LS} \to {^LG}$ is the extension of $\hat j$ given by $\chi$, well-defined up to conjugation by $\hat T$, and $\varphi_S : W_F \to {^LS}$ is the parameter of $\theta$. It sends the morphism $(\iota,g,\zeta) : (S_1,\hat j_1,\chi_1,\theta_1) \to (S_2,\hat j_2,\chi_2,\theta_2)$ to the morphism $\tx{Ad}(g) : S_\varphi \to S_{\tx{Ad}(g)\varphi}$.
	\item $(S,\hat j,\chi,\theta) \mapsto \hat S^\Gamma$. It sends the morphism $(\iota,g,\zeta)$ to the morphism $\hat\iota^{-1} : \hat S_1^\Gamma \to \hat S_2^\Gamma$.
	\item $(S,\hat j,\chi,\theta) \mapsto \Omega(S,G)(F)_\theta$. It sends the morphism $(\iota,g,\zeta)$ to the morphism $\Omega(S_1,G)(F)_{\theta_1} \to \Omega(S_2,G)(F)_{\theta_2}$ induced by $\iota$.
\end{enumerate}

\begin{pro} \label{pro:sphiseq} There is a functorial exact sequence
\begin{equation} \label{eq:sphi} 1 \to \hat S^\Gamma \to S_\varphi \to \Omega(S,G)(F)_\theta \to 1. \end{equation}
\end{pro}

We begin with a preparatory lemma, considering a more general situation where $j : S \to G$ is an embedding defined over $F$ of a torus $S$ into $G$ as a maximal torus, $\theta : S(F) \to \C^\times$ is a character, $\varphi_S : W_F \to {^LS}$ its parameter, $\chi$ a set of $\chi$-data for $R(S,G)$, $^Lj : {^LS} \to {^LG}$ the corresponding $L$-embedding, with image $\hat T := {^Lj}(\hat S)$, and $\varphi = {^Lj} \circ \varphi$. We have the exact sequence $1 \to \hat S \to N(\hat T,\hat G) \to \Omega(S,G) \to 1$ in which the first map and third maps are given by the identifications $\hat S \to \hat T$ and $\Omega(S,G)=\Omega(\hat S,\hat G) \to \Omega(\hat T,\hat G)$ induced by $^Lj$.

\begin{lem} \label{lem:c1} In the exact sequence for $W_F$-cohomology
\[ 1 \to \hat S^{\Gamma} \to N(\hat T,\hat G)^{\varphi(W_F)} \to \Omega(S,G)(F) \to H^1(W_F,\hat S) \]
an element $w \in \Omega(S,G)(F)$ is mapped to the parameter of the character $w\theta/\theta$, provided the $\chi$-data is $w$-invariant.
\end{lem}
\begin{proof}
The $w$-invariance of the $\chi$-data and \cite[(2.6.2)]{LS87} imply the existence of a lift $n \in N(\hat T,\hat G)$ such that $\tx{Ad}(n)\circ{^Lj} = {^Lj} \circ w$. The connecting homomorphism sends $w$ to the class of the 1-cocycle of $W_F$ valued in $\hat T$ for the action of $W_F$ via $\tx{Ad}(\varphi(-))$ given by
\[ x \mapsto \varphi(x)^{-1}n\varphi(x)n^{-1}={^Lj}(\varphi_S(x)^{-1} \cdot {^w\varphi_S(x)}). \]
Via the identification of $\hat T$ with $\hat S$ this 1-cocycle becomes $x \mapsto \varphi_S(x)^{-1}\cdot {^w\varphi_S(x)}$, which is the parameter for $^w\theta/\theta$.
\end{proof}

\begin{proof}[Proof of Proposition \ref{pro:sphiseq}]
Let $\hat M=\tx{Cent}(\varphi(P_F),\hat G)$. According to \cite[Lemma 5.2.2]{KalRSP} we have $\tx{Cent}(\varphi(W_F),\hat G) \subset N(\hat T,\hat M) \subset N(\hat T,\hat G)$. We thus consider the exact sequence
\[ 1 \to \hat T \to N(\hat T,\hat G) \to \Omega(\hat T,\hat G) \to 1 \]
with action of $W_F$ via $\tx{Ad}(\varphi(-))$. Since $N(\hat T,\hat G)^{\varphi(W_F)}$ is contained in $N(\hat T,\hat M)$, its image in $\Omega(\hat T,\hat G)$ is contained in $\Omega(\hat T,\hat M)$. Any $w \in \Omega(\hat S,\hat M)^\Gamma$ preserves the $\chi$-data in the datum $(S,\hat j,\chi,\theta)$ so we may apply Lemma \ref{lem:c1} and see that the image of $N(\hat T,\hat G)^{\varphi(W_F)}$ in $\Omega(S,G)(F)$ is precisely $\Omega(S,G^0)(F)_\theta$. Now apply Lemma \ref{lem:weylred}.

The functoriality of the exact sequence follows by a straightforward unwinding of the definitions.
\end{proof}

We set $\bar S=S/Z$, $\bar G=G/Z$ and obtain the covers $\hat{\bar S} \to \hat S$ and $\hat{\bar G} \to \hat G$. Let $[\hat{\bar S}]^+$ be the preimage of $\hat S^\Gamma$ and $S_\varphi^+ \subset \hat{\bar G}$ be the preimage of $S_\varphi$. Both of these are functors in $(S,\hat j,\chi,\theta)$.

\begin{cor} \label{cor:sphipseq} We have the functorial exact sequence
\begin{equation} \label{eq:sphi+} 1 \to \pi_0([\hat{\bar S}]^+) \to \pi_0(S_\varphi^+) \to \Omega(S,G)(F)_\theta \to 1. \end{equation}
\end{cor}

It is tempting to expect that this extension, or at least the simpler extension \eqref{eq:sphi}, has the multiplicity 1 property in the sense of Definition \ref{dfn:mult1}. While this does hold in many cases, it turns out that it doesn't always hold, as we now discuss.

\begin{lem} \label{lem:sphim1} Assume that $\varphi$ is of depth zero. The extension \eqref{eq:sphi+} has multiplicity 1 in the following cases.
\begin{enumerate}
	\item $G$ is simply connected.
	\item $G$ is unramified.
\end{enumerate}
\end{lem}
\begin{proof}
If $G$ is simply connected we can write it as a product of $F$-simple factors, and assume that $G$ is $F$-simple. Then it is of the form $\tx{Res}_{E/F}H$ for an absolutely simple simply connected group $H$ defined over a finite tamely ramified extension $E/F$. We may thus assume that $G$ is absolutely simple. The claim now follows from Lemma \ref{lem:weylab2} and Lemma \ref{lem:cliff1} in all cases except when $H$ is split of type $D_{2n}$ and $\Omega(S,G)(F)_\theta \cong (\Z/2\Z)^2$. In the latter case we let $s \in \hat G$ be the image of a topological generator of $I_F/P_F$ under $\varphi$, and $f \in \hat G$ the image of a Frobenius element. The extension \eqref{eq:sphi+} is then the push-out of the extension of Lemma \ref{lem:d2ng} along the inclusion $\hat S_\tx{sc}^\Gamma \to \hat S_\tx{sc}^+$, and the claim follows from that Lemma.

Assume now that $G$ is unramified. We shall repeatedly modify the extension \eqref{eq:sphi+} and use Corollary \ref{cor:m1stab}, without explicitly referring to it. For example, we may replace the kernel $Z$ of $G \to \bar G$ by a larger one, because the extension for the smaller $Z$ is a push-out of the extension for the larger $Z$. We may thus assume $Z(G_\tx{der}) \subset Z$, so that $\bar G = G_\tx{ad} \times Z(\bar G)$ and $\hat{\bar G}=\hat G_\tx{sc} \times Z(\hat{\bar G})^\circ$. Since every irreducible representation of $\pi_0(S_\varphi^+)$ transforms under $\pi_0(Z(\hat{\bar G})^+)$ by a character, it is enough to fix a character $\zeta$ of $\pi_0(Z(\hat{\bar G})^+)$ and consider only those irreducible representations of $\pi_0(S_\varphi^+)$ that transform by that character. Applying the bijection \cite[(6.7)]{KalRIBG} we may assume that this character is trivial on the kernel of the morphism $Z(\hat{\bar G})^+ \to Z(\hat G_\tx{sc})$ induced by projecting onto the first factor of $\hat{\bar G}=\hat G_\tx{sc} \times Z(\hat{\bar G})^\circ$: Indeed, we have the diagram
\[ \xymatrix{
Z(\hat G_\tx{sc})^\Gamma\ar@{^(->}[r]\ar@{^(->}[rd]&Z(\hat G_\tx{sc})\\
&Z(\hat{\bar G})^+\ar[u]
}
\]
We may thus restrict $\zeta$ to $Z(\hat G_\tx{sc})$, extend to $Z(\hat G_\tx{sc})$, and via the vertical arrow obtain another character $\zeta'$ of $Z(\hat{\bar G})^+$. By construction this character is trivial on $Z(\hat{\bar G})^{\circ,\Gamma}$. Upon further enlarging $Z$ we may assume that $\zeta$ is trivial on $Z(\hat{\bar G})^{\circ,\Gamma}$ (see erratum to \cite{KalRIBG}). Thus the difference between $\zeta$ and $\zeta'$ is trivial on $Z(\hat{\bar G})^{\Gamma}$ and thus factors through the differential $d : Z(\hat{\bar G})^+ \to Z^1(\Gamma,\hat Z)$ and can be extended to a character $\eta$ of $Z^1(\Gamma,\hat Z)$. This differential is the restriction of the differential $d : \pi_0(S_\varphi^+) \to Z^1(\Gamma,\hat Z)$. We can pull back $\eta$ to a character of $\pi_0(S_\varphi^+)$. Tensoring with this character gives a bijection between the irreducible representations transforming under $\pi_0(Z(\hat{\bar G})^+)$ by $\zeta$ and those transforming by $\zeta'$, and this bijection preserves the property of having multiplicity 1 upon restriction to $\pi_0([\hat{\bar S}]^+)$.

We have thus arranged that $\zeta$ is trivial on the kernel of the morphism $Z(\hat{\bar G})^+ \to Z(\hat G_\tx{sc})$. Let $\zeta_\tx{sc}$ be an extension of this character to $Z(\hat G_\tx{sc})$. Write $\tx{Irr}(\pi_0(S_\varphi^+),\zeta)$ for the set of irreducible representations of $\pi_0(S_\varphi^+)$ that transform under the central group $\pi_0(Z(\hat{\bar G})^+)$ by $\zeta$. If $\rho \in \tx{Irr}(\pi_0(S_\varphi^+),\zeta)$, then the representation $\rho\boxtimes\zeta_\tx{sc}$ of $\pi_0(S_\varphi^+) \times Z(\hat G_\tx{sc})$ descends to the push-out $S_\varphi^+ \times_{Z(\hat{\bar G})^+} Z(\hat G_\tx{sc})$. The restriction of $\rho$ to $\pi_0([\hat{\bar S}]^+)$ has multiplicity 1 if and only if the restriction of $\rho\boxtimes\zeta_\tx{sc}$ to $\pi_0([\hat{\bar S}]^+ \times_{Z(\hat{\bar G})^+} Z(\hat G_\tx{sc}))$ has multiplicity $1$.

Let $S_\varphi^\tx{sc}$ be the preimage in $\hat G_\tx{sc}$ of $S_\varphi/Z(\hat G)^\Gamma \subset \hat G_\tx{ad}$. Recall the bijection $S_\varphi^\tx{sc} \to S_\varphi^+ \times_{Z(\hat{\bar G})^+} Z(\hat G_\tx{sc})$ from \cite[(4.6)]{KalGRI}. It sends $s_\tx{sc} \in S_\varphi^\tx{sc}$ to $(s_\tx{sc}\dot y',\dot y'',(\dot y')^{-1})$, where $y \in Z(\hat G)$ is chosen so that $s_\tx{der}y \in S_\varphi$, where $s_\tx{der} \in \hat G_\tx{der}$ is the image of $s_\tx{sc}$ under $\hat G_\tx{sc} \to \hat G_\tx{der}$, $y' \in Z(\hat G_\tx{der})$ and $y'' \in Z(\hat G)^\circ$ are chosen so that $y=y'y''$, and $\dot y' \in Z(\hat G_\tx{sc})$ and $\dot y'' \in Z(\hat{\bar G})^\circ$ are lifts of $y'$ and $y''$. Then $(s_\tx{sc}\dot y', \dot y'') \in \hat G_\tx{sc} \times Z(\hat{\bar G})^\circ = \hat{\bar G}$ belongs to $S_\varphi^+$. Let $\hat S_\tx{sc}^+$ denote the preimage of $\hat S^\Gamma/Z(\hat G)^\Gamma$. In the same way we obtain the isomorphism $\hat S_\tx{sc}^+ \to [\hat{\bar S}]^+ \times_{Z(\hat{\bar G})^+} Z(\hat G_\tx{sc})$. We have thus reduced the problem to showing that the extension
\[ 1 \to \hat S_\tx{sc}^+ \to S_\varphi^\tx{sc} \to \Omega(S,G)(F)_\theta \to 1 \]
has multiplicity $1$. All groups in this extension are finite.

Write again $s \in \hat G$ and $f \in \hat G \rtimes \tx{Frob}$ for the images of a topological generator of $I_F/P_F$ and a Frobenius element in $W_F/P_F$ under $\varphi$. Then $S_\varphi^\tx{sc}=\{\dot x \in \hat G_\tx{sc}|\exists z \in Z(\hat G) : \tx{Ad}(s)(xz)=xz,\tx{Ad}(f)(xz)=xz\}$, where $x \in \hat G$ is the image of $\dot x$. 

The element $s$ is regular semi-simple and the identity component of its centralizer is $\hat T=\hat j(\hat S)$. Elements of $S_\varphi^\tx{sc}$ normalize $\hat T$ and therefore lie in $N(\hat T_\tx{sc},\hat G_\tx{sc})=:\hat N_\tx{sc}$. Let $\hat N_\tx{sc}^+=\{\dot x \in \hat N_\tx{sc}|\tx{Ad}(s)(x)=x\}$. We have the extension $1 \to \hat T_\tx{sc} \to \hat N_\tx{sc}^+ \to \Omega_s \to 1$, where $\Omega_s$ is the stabilizer in $\Omega(\hat T,\hat G)$ of $s \in \hat T$. Since $(1-f) : \hat T_\tx{sc} \to \hat T_\tx{sc}$ has finite kernel, it is surjective, and hence taking $\tx{Ad}(f)$-fixed points gives an exact sequence $1 \to \hat T_\tx{sc}^f \to \hat N_\tx{sc}^{+,f} \to \Omega_s^f \to 1$.

We claim that $\Omega_s^f = \Omega(S,G)(F)_\theta$. Indeed, taking $\tx{Ad}(f)$-fixed points in the exact sequence $1 \to \hat T \to N(\hat T,\hat G)^s \to \Omega_s \to 1$ and using the surjectivity of $1-f : \hat T \to \hat T$ we see that $\Omega_s^f$ is the image of $N(\hat T,\hat G)^{s,f}$ under $N(\hat T,\hat G) \to \Omega(\hat T,\hat G)$, and that in turn equals $\Omega(S,G)(F)_\theta$.

Pushing out $1 \to \hat T_\tx{sc}^f \to \hat N_\tx{sc}^{+,f} \to \Omega_s^f \to 1$ along the inclusion $\hat T_\tx{sc}^f \to \hat T_\tx{sc}^+:=\{\dot x \in \hat T_\tx{sc}|\exists z \in Z(\hat G): \tx{Ad}(f)(xz)=xz\} = \hat S_\tx{sc}^+$ we obtain the extension
\[ 1 \to \hat S_\tx{sc}^+ \to S_\varphi^+ \to \Omega(S,G)(F)_\theta \to 1, \]
of which we want to show that it satisfies multiplicity $1$. It is thus enough to show this for the extension $1 \to \hat T_\tx{sc}^f \to \hat N_\tx{sc}^{+,f} \to \Omega_s^f \to 1$.

Let $\hat N_\tx{sc}^\dagger$ be the preimage of the centralizer of $s$ in $\hat N_\tx{ad}$ and let $\Omega_\dagger$ be the image of $\hat N_\tx{sc}^\dagger$ in $\Omega(\hat T,\hat G)$. Then $\hat N_\tx{sc}^+$ is the preimage in $\hat N_\tx{sc}^\dagger$ of $\Omega_s \subset \Omega_\dagger$. Taking $\tx{Ad}(f)$-fixed points we obtain the extension $1 \to \hat T_\tx{sc}^f \to \hat N_\tx{sc}^{\dagger,f} \to \Omega_\dagger^f \to 1$. It is enough to show the multiplicity $1$ property for this extension, because pulling back along the inclusion $\Omega_s^f \to \Omega_\dagger^f$ we obtain the extension $1 \to \hat T_\tx{sc}^f \to \hat N_\tx{sc}^{+,f} \to \Omega_s^f \to 1$.

We now break $\hat G_\tx{sc}$ into a product of simple factors. The extension $1 \to \hat T_\tx{sc} \to \hat N_\tx{sc}^\dagger \to \Omega_\dagger \to 1$ breaks accordingly. The action of $\tx{Ad}(f)$ permutes the simple factors, and the extension $1 \to \hat T_\tx{sc}^f \to \hat N_\tx{sc}^{\dagger,f} \to \Omega_\dagger^f \to 1$ breaks up according to orbits of simple factors under $\tx{Ad}(f)$. By Corollary \ref{cor:m1stab} we may assume that there is only one orbit. Shapiro's lemma then reduces to the case where the orbit is a singleton. Lemma \ref{lem:weylab2} completes the proof in all cases except when $G$ is of split type $D_{2n}$ and $\Omega_s=\Omega_\dagger=(\Z/2\Z)^2$, in which case we appeal to Lemma \ref{lem:d2ng}.
\end{proof}

In the following example we will show that the extension \eqref{eq:sphi} can fail to have multiplicity $1$, even for parameters of depth zero, when the group in question is ramified. As discussed in the proof of Proposition \ref{pro:lpdat}, there is an isomorphism $S_\varphi(G)=S_\varphi(G^0)$, where $G^0 \subset G$ is a tame twisted Levi, for which $\varphi$ is essentially of depth zero. Even if $G$ is taken to be unramified, or simply connected, $G^0$ will be neither of these. Thus, despite Lemma \ref{lem:sphim1}, one cannot expect the multiplicity 1 property for the extension \eqref{eq:sphi} for general parameters, even after placing restrictions on $G$.

\begin{exa}
Let $\Gamma=(\Z/2\Z)\times (\Z/2\Z)$ be a quotient of $W_F/P_F$, with $(1,0)$ being the image of a topological generator of $I_F/P_F$, and $(0,1)$ being the image of a Frobenius element. Thus $\Gamma$ is the Galois group of the biquadratic extension $F(\varpi\eta)/F$, where $\varpi$ is a square root of a uniformizer $\pi \in F$ and $\eta$ is a square root of a root of unity in $O_F^\times$ of order $q-1$. We have $(1,0)\varpi=-\varpi$, $(1,0)\eta=\eta$, $(0,1)\varpi=\varpi$, $(0,1)\eta=-\eta$.

 Consider the following complex algebraic groups with $\Gamma$-action: An algebraic torus $Z_1=\C^\times \times \C^\times$ with $(1,0)(z_1,z_2)=(z_1^{-1},z_2^{-1})$ and $(0,1)(z_1,z_2)=(z_1z_2^4,z_2^{-1})$. An algebraic torus $Z_2=\C^\times$ with $(1,0)z=z^{-1}$ and $(0,1)z=z$. The group $G_1'=\tx{SL}_4(\C)$ with both $(1,0)$ and $(0,1)$ acting as $\theta$, where $\theta$ is the pinned non-trivial outer automorphism of $\tx{SL}_4(\C)$ relative to the standard pinning. The group $G_2'=\tx{SL}_4(\C)$ with $(1,0)$ acting as $\theta$ and $(0,1)$ acting as the identity. We embed $\mu_4 \to Z_1$ via $z \mapsto (1,z)$. Let $G_1=(G_1' \times Z_1)/\mu_4$ and $G_2=(G_2' \times Z_2)/\mu_4$, where in both cases $\mu_4$ is embedded anti-diagonally. 

Thus $G_2$ is the dual group of the unitary group $\tx{U}_4(F(\varpi\eta)/F)$, while $G_1$ is the dual group of a reductive group whose derived subgroup is $\tx{SU}_4(F(\varpi)/F)$ and whose base-change to $F(\eta)$ is the product $\tx{U}_4(F(\varpi,\eta)/F(\eta)) \times \tx{U}_1(F(\varpi,\eta)/F(\eta))$.

We embed $\mu_4 \to \mu_4 \times \mu_4$ by $z \mapsto (z,z^2)$ and then further $\mu_4 \times \mu_4 \to Z_1 \times Z_2$ by $(z_1,z_2) \mapsto (1,z_1,z_2)$ and form $G=(G_1 \times G_2)/\mu_4$. Consider the elements
\[ s = \left(
\begin{bmatrix}
\zeta_4\\&\zeta_6\\&&\zeta_6^{-1}\\&&&\zeta_4^{-1}
\end{bmatrix},
\begin{bmatrix}
\zeta_4\\&\zeta_6\\&&\zeta_6^{-1}\\&&&\zeta_4^{-1}
\end{bmatrix}
\right) \rtimes (1,0)
\]
and
\[f = \left(
\begin{bmatrix}
&&&\zeta_4\\&&-\zeta_6\\&\zeta_6^{-1}\\-\zeta_4^{-1}
\end{bmatrix},
\begin{bmatrix}
&&&1\\&&-1\\&1\\-1
\end{bmatrix}
\right) \rtimes (0,1) \]
of $(G_1' \times G_2') \rtimes \Gamma$, where $\zeta_k=\exp(2\pi i/k)$. We have $fsf^{-1}=s^{11}$ and we set $p=q=11$. These two elements together give a depth zero supercuspidal parameter for the algebraic group over $\Q_p$ with dual group $G$.

We now compute the mutual centralizer of $s$ and $f$ in $G$ and its image in $G_\tx{ad}$. Let $T_\tx{ad}$ denote the standard diagonal torus in $\tx{PGL}_4(\C) \times \tx{PGL}_4(\C)$. Then
\[ T_\tx{ad}^s = \left\{\begin{bmatrix} a\\&b\\&&b^{-1}\\&&&a^{-1}\end{bmatrix}\right\} \times \left\{\begin{bmatrix} c\\&d\\&&d^{-1}\\&&&c^{-1}\end{bmatrix}\right\}, \]
where $(a,b)$ and $(c,d)$ run over $(\C^\times \times \C^\times)/\mu_2$, with $\mu_2$ embedded diagonally.
The preimage in $G_\tx{sc}$ is
\[ = \left\{x\begin{bmatrix} a\\&b\\&&b^{-1}\\&&&a^{-1}\end{bmatrix}\right\} \times \left\{y\begin{bmatrix} c\\&d\\&&d^{-1}\\&&&c^{-1}\end{bmatrix}\right\}, \]
where now $(x,a,b)$ and $(y,c,d)$ run over $(\mu_4 \times \C^\times \times \C^\times)/\mu_2$, with $\mu_2$ embedded diagonally. Applying $(s-1)$ to such an element gives $(x^{-2},y^{-2}) \in \mu_4 \times \mu_4 = Z(G_1') \times Z(G_2')$, while applying $(f-1)$ gives
\begin{equation} \label{eq:fc} \left(x^{-2}\begin{bmatrix} a^{-2}\\&b^{-2}\\&&b^2\\&&&a^2\end{bmatrix},\begin{bmatrix} c^{-2}\\&d^{-2}\\&&d^2\\&&&c^2\end{bmatrix}\right). \end{equation}
The centralizer of $s$ in $G_\tx{ad}$ has four cosets under the centralizer in $T_\tx{ad}$ and they are represented by the matrices $(1,n)$, $(n,1)$, and $(n,n)$, where
\[ n = \begin{bmatrix}
&&&\zeta_8^{-1}\\&-\zeta_8\\&&-\zeta_8\\-\zeta_8^{-1}
\end{bmatrix} \in \tx{SL}_4(\C), \]
where as before $\zeta_k=\exp(2\pi i/k)$. We compute $(s-1)(n,1)=(\zeta_4^{-1},1)$, and $(s-1)(1,n)=(1,\zeta_4^{-1})$, $(f-1)(n,1)=(\zeta_4^{-1},1)$, $(f-1)(1,n)=(1,1)$.

We can now compute the image in $G_\tx{ad}$ of $G^{s,f}$. First consider an element of $T_\tx{sc}$ with image in $T_\tx{ad}^s$. It is given by $(x,a,b),(y,c,d)$. It will belong to the image of $T^{s,f}$ in $T_\tx{ad}$ if and only if there exists $(z_1,z_2) \in Z_1 \times Z_2$ s.t. $(s-1)(x,a,b,y,c,d)=(s-1)(z_1,z_2,z_3)$ and $(f-1)(x,a,b,y,c,d)=(f-1)(z_1,z_2,z_3)$, where we have used $Z_1 \times Z_2 = \C^\times \times \C^\times \times \C^\times$, and the equalities are to hold in $G$. By definition $(s-1)(z_1,z_2,z_3)=(z_1^{-2},z_2^{-2},z_3^{-2})$ and $(f-1)(z_1,z_2,z_3)=(z_2^4,z_2^{-2},1)$.

Looking at $(f-1)$ we see that the term \eqref{eq:fc} must belong to the center of $G_\tx{sc}$, which forces $a^2=a^{-2}=b^2=b^{-2}$, i.e. $a,b \in \mu_4$ and $a^2=b^2$, and the same for the pair $(c,d)$. Looking at both $(s-1)$ and $(f-1)$ we see that we must find $(z_1,z_2,z_3) \in \C^\times$ such that the equalities $(1,x^{-2},y^{-2})=(z_1^{-2},z_2^{-2},z_3^{-2})$ and $(1,x^{-2}a^{-2},c^{-2})=(z_2^4,z_2^{-2},1)$ hold in the quotient of $\C^\times \times \C^\times \times \C^\times$ by the subgroup $\{(1,z,z^2)|z \in \C^\times\}$. Using that subgroup and the fact that $a \in \mu_4$ we can rewrite the second equation as $(1,x^{-2},c^{-2})=(z_2^4,z_2^{-2},1)$. This forces $z_2 \in \mu_4$ and $x^{-2}c = \pm z_2^{-2}$, implying $c \in \mu_2$ and hence $d \in \mu_2$. Conversely, for any tuple $(x,a,b,y,c,d)$ satisfying $a,b \in \mu_4$, $c,d \in \mu_2$, $a^2=b^2$, we can find $(z_1,z_2,z_3)$ that satisfy these equations. We conclude that the image of $T^{s,f}$ in $T_\tx{ad}$ is given by
\[ \begin{bmatrix} 1\\&\epsilon\\&&\eta\\&&&\epsilon\eta
\end{bmatrix},
\begin{bmatrix}1\\&\delta\\&&\delta\\&&&1
\end{bmatrix}
\]
for $\epsilon,\eta,\delta \in \mu_2$.

We now come to the non-trivial $T^{s,f}$-cosets in $G^{s,f}$ and their image in $G_\tx{ad}$. Since there are precisely four cosets of $T_\tx{ad}^s$ in $G_\tx{ad}^s$, there can be at most four $T^{s,f}$-cosets in $G^{s,f}$. The element $(1,n)$ of $G_\tx{ad}$ is fixed by both $s$ and $f$, and it is the image of the element $(1,\zeta_8^{-1}n) \in G^{s,f}$. Let $t$ be the diagonal matrix with entries $(\zeta_4,\zeta_4,-\zeta_4,-\zeta_4)$. Then $(s-1)(n,t)=(1,\zeta_4^{-1},1)$ and $(f-1)(n,t)=(1,\zeta_4^{-1},-1)$. Modulo $(1,z,z^2)$ these elements become $(1,1,-1)$ and $(1,1,1)$ respectively, and thus equal to $(z_1^{-2},z_2^{-2},z_3^{-2})$ and $(z_2^4,z_2^{-2},1)$ if we take $z_1=z_2=1$ and $z_3=\zeta_4$. We conclude that $(n,t)$ also belongs to the image of $G^{s,f}$ in $G_\tx{ad}$. We conclude that there are exactly four cosets of $T^{s,f}$ in $G^{s,f}$. The image of $G^{s,f}$ in $G_\tx{ad}$ is thus an extension
\begin{equation} \label{eq:sphiex} 1 \to T^{s,f}/Z(G)^{s,f} \to G^{s,f}/Z(G)^{s,f} \to (\Z/2\Z)^2 \to 0, \end{equation}
and the elements $(1,n)$ and $(n,t)$ map to a basis of $(\Z/2\Z)^2$. We compute their commutator and find that it is given by $\epsilon=\eta=1$ and $\delta=-1$. Since the actions of $(1,0)$ and $(0,1)$ send $(\epsilon,\eta,\delta)$ to $(\eta,\epsilon,\delta)$ and $(\epsilon,\eta,\delta)$ respectively, we see that $(1,1,-1)$ does not vanish in the quotient of coinvariants in $T^{s,f}/Z(G)^{s,f}$ for the action of $(\Z/2\Z)^2$. Lemma \ref{lem:cliff1} implies that the extension \eqref{eq:sphiex} does not have the multiplicity 1 property.
\end{exa}

\subsection{Internal structure I: Reduction to depth zero DL-packets} \label{sub:is1}

Having gained some understanding of the structure of $\pi_0(S_\varphi^+)$, we turn to establishing a bijection between $\tx{Irr}(\pi_0(S_\varphi^+))$ and the compound $L$-packet $\Pi_\varphi$ constructed in \S\ref{sub:lpackconst}. Such a bijection is expected to depend only on the choice of a Whittaker datum $\mf{w}$ for the quasi-split group $G$, and to satisfy stability and endoscopic character identities, as described in \cite[Conjecture G]{KalSimons}. Our construction in this paper will be less precise -- we will make some auxiliary choices and show that they lead to a bijection 
\begin{equation} \label{eq:intstr}
\tx{Irr}(\pi_0(S_\varphi^+)) \to \Pi_\varphi.	
\end{equation}
Then we will sketch an argument showing that this bijection satisfies stability as well as endoscopic transfer for $s \in [\hat{\bar S}]^+ \subset S_\varphi^+$. The discussion of how the auxiliary choices involved in the construction of the bijection relate to the choice of a Whittaker datum, the details of the argument of endoscopic transfer, and its extension to all $s \in S_\varphi^+$, will be given in a forthcoming paper.

Before we begin, we summarize here the construction of the compound $L$-packet $\Pi_\varphi$. The parameter $\varphi$ corresponds to an isomorphism class of torally wild $L$-packet data $(S,\hat j,\chi_0,\theta)$ by Proposition \ref{pro:lpdat}. The compound $L$-packet $\Pi_\varphi$ is a disjoint union of non-singular Deligne-Lusztig packets. There is a surjective map from the set of those isomorphism classes of non-singular Deligne-Lusztig packet data that map to the isomorphism class of $(S,\hat j,\chi_0,\theta)$ to the set of non-singular Deligne-Lusztig packets inside of $\Pi_\varphi$. The former are of the form $(S,\hat j,\chi_0,\theta,G',\xi,z,j)$ and are a torsor under $H^1(u \to W,Z(G) \to S)$, see \eqref{eq:jtors}. The surjection does not depend on any choices, and maps two such data $(S,\hat j,\chi_0,\theta,(G',\xi,z),j_i)$, $i=1,2$, to the same non-singular Deligne-Lusztig packet if and only if there is $w \in \Omega(S,G)(F)_\theta$ such that $(S,\hat j,\chi_0,\theta,(G',\xi,z),j_2\circ w)$ is isomorphic to $(S,\hat j,\chi_0,\theta,(G',\xi,z),j_1)$.

The first auxiliary choice that we make is an admissible embedding $j_0 : S \to G$ having the following property:  The pair $(S,\theta)$ and the embedding $j_0$ determine a tame twisted Levi $G^0 \subset G$ as in Definition \ref{dfn:tnsep} and Remark \ref{rem:g0}. Recall that $G^0$ contains $j_0(S)$ and its root system is $\{\alpha \in R(S,G)|\theta\circ N_{E/F}\circ \alpha^\vee(E_{0+}^\times)=1\}$, where $E/F$ is the splitting field of $S$. The maximal torus $j_0(S)$ of $G^0$ is maximally unramified. It follows that the point in $\mc{B}(G^0,F)$ determined by it by Prasad's theorem \cite{Pr01} is a vertex, \cite[Lemma 3.4.3]{KalRSP}. The property of $j_0$ that we require is that $G^0$ is quasi-split and this vertex is absolutely special. Such an admissible embedding exists by \cite[Lemma 3.4.12]{KalRSP}.

If the parameter $\varphi$ is essentially of depth zero, by which we mean that $\varphi(P_F) \subset Z(\hat G)$, then Lemma \ref{lem:dz_genconst} implies that $j_0$ is uniquely determined by the Whittaker datum $\mf{w}$. We expect that this continues to be true in the general case, and will follow from the asymptotic expansions of \cite{Spice17}.

The choice of $j_0$ establishes a bijection between $H^1(u \to W,Z(G) \to S)$ and the set of isomorphism classes of non-singular Deligne-Lusztig packet data mapping to the isomorphism class of $(S,\hat j,\chi_0,\theta)$, and thus a surjection from $H^1(u \to W,Z(G) \to S)$ to the set of non-singular Deligne-Lusztig packets inside of $\Pi_\varphi$. Lemma \ref{lem:c2} and \cite[Lemma 3.4.10]{KalRSP} imply that $\eta_1,\eta_2 \in H^1(u \to W,Z(G) \to S)$ map to the same non-singular Deligne-Lusztig packet if and only if there exists $w \in \Omega(S,G)(F)_\theta$ s.t. $\eta_2=w\eta_1$. Thus we obtain a bijection between $H^1(u \to W,Z(G) \to S)/\Omega(S,G)(F)_\theta$ and the set of non-singular Deligne-Lusztig packets inside of $\Pi_\varphi$. Let us write $[\pi_\eta]$ for the Deligne-Lusztig packet corresponding to $\eta \in H^1(u \to W,Z(G) \to S)$. Explicitly, it is the packet $[\pi_j]$ for the unique datum $(S,\hat j,\chi_0,\theta,(G',\xi,z),j)$ s.t. $\tx{inv}(j_0,j)=\eta$.

Corollary \ref{cor:sphipseq} gives us the extension \eqref{eq:sphi+}
\[ 1 \to \pi_0([\hat {\bar S}]^+) \to \pi_0(S_\varphi^+) \to \Omega(S,G)(F)_\theta \to 1, \]
functorially assigned to a torally wild $L$-packet datum $(S,\hat j,\chi_0,\theta)$. Restriction of representations gives a surjection $\tx{Irr}(\pi_0(S_\varphi^+)) \to \pi_0([\hat {\bar S}]^+)^*/\Omega(S,G)(F)_\theta$. Combining this with the functorial isomorphism $H^1(u \to W,Z(G) \to S) \to \pi_0([\hat {\bar S}]^+)^*$ of \cite[Corollary 5.4]{KalRI} we see that the construction of the bijection $\tx{Irr}(\pi_0(S_\varphi^+)) \to \Pi_\varphi$ reduces to the construction of a bijection $\tx{Irr}(\pi_0(S_\varphi^+),\eta) \to [\pi_\eta]$, where $\tx{Irr}(\pi_0(S_\varphi^+),\eta)$ is the set of irreducible representations of $\pi_0(S_\varphi^+)$ whose restriction to $\pi_0([\hat{\bar S}]^+)$ contains the character $\eta$.

We shall now reduce the construction of the bijection $\tx{Irr}(\pi_0(S_\varphi^+),\eta) \to [\pi_\eta]$ to the case where $\varphi$ is essentially of depth zero. It would be convenient to fix a finite $Z \subset Z(G)$ and form $S_\varphi^+$ with respect to that $Z$, rather than the full $Z(G)$.

We have fixed the embedding $j_0 : S \to G^0 \to G$, with $G^0 \subset G$ quasi-split. Let $(S,\hat j,\chi_0,\theta,(G',\xi,z),j)$ be the unique non-singular Deligne-Lusztig packet datum s.t. $\tx{inv}(j_0,j)=\eta$. Then $\pi_\eta$ is the supercuspidal representation associated to the tame elliptic $k_F$-non-singular pair $(jS,\theta_j)$, where we recall that $\theta_j$ is the character $j_*\theta \cdot j_*\zeta_S$, $\zeta_S$ is the character of $S(F)$ associated to the $\zeta$-data $\zeta_\alpha=\chi_\alpha \cdot (\chi'_\alpha)^{-1}$ and $\chi'_\alpha$ is computed in terms of $\theta$.

Let $G'^0 \subset G'$ be the twisted Levi subgroup with maximal torus $jS$ and root system $R_{0+}$. It is an inner form of $G^0$. In fact, the inner twist $\xi : G \to G'$ restricts to an inner twist $\xi : G^0 \to G'^0$ and $(G'^0,\xi,z)$ becomes a rigid inner twist of $G^0$. 

Recall the subgroup $\hat M=\tx{Cent}(\varphi(P_F),\hat G)$. In Lemma \ref{lem:r0id} we showed that it is a dual group to $G^0$. Let $^Lj_{G^0,G} : {^LG^0} \to {^LG}$ be the $L$-embedding associated to $\chi_0$ and let $\varphi_{G^0} : W_F \to {^LG^0}$ be s.t. $\varphi={^Lj_{G^0,G}} \circ \varphi_{G^0}$. Then $S_\varphi \subset \hat G$ lies in $\hat M$ and equals $S_{\varphi_{G^0}} \subset \hat M$. In particular, $\varphi_{G^0}$ is discrete and hence supercuspidal. Moreover $\varphi_{G^0}$ is essentially of depth zero, by construction. The identification of $S_{\varphi_{G^0}}$ with $S_\varphi$ extends to an identification of $S_{\varphi_{G^0}}^+$ with $S_\varphi^+$, where both are taken with respect to the fixed finite $Z \subset Z(G)$. For any $\eta \in H^1(u \to W,Z \to S)$ we thus have an identification between $\tx{Irr}(\pi_0(S_{\varphi_{G^0}}^+),\eta)$ and $\tx{Irr}(\pi_0(S_{\varphi}^+),\eta)$.

Let $^Lj_{S,G^0} : {^LS} \to {^LG^0}$ be the $L$-embedding associated to unramified $\chi$-data. Let $\varphi_{S,G^0} : W_F \to {^LS}$ be the parameter s.t. $\varphi_{G^0} = {^Lj_{S,G^0}} \circ \varphi_{S,G^0}$. Let $\theta_{G^0}$ be the character of $S(F)$ corresponding to $\varphi_{S,G^0}$. Then $(S,\hat j,\emptyset,\theta_{G^0})$ is a torally wild $L$-packet datum associated to $\varphi_{G^0}$ and $(S,\hat j,\emptyset,\theta,(G'^0,\xi,z),j)$ is the unique non-singular Deligne-Lusztig packet datum with $\tx{inv}(j_0,j)=\eta$. The supercuspidal representation $\pi_\eta^{
G^0}$ of $G^0(F)$ associated to this datum is $\pi_{(jS,\theta_{G^0,j})}^{G^0}$, where $\theta_{G^0,j} = j_*\theta_{G^0}$.

Recall that $\theta$ is the character associated to the parameter $\varphi_S$ obtained by $\varphi={^Lj_{S,G}}\circ \varphi_S$, where ${^Lj_{S,G}}$ is constructed by lifting to $R(S,G) \sm R(S,G^0)$ the $\chi$-data $\chi_0$ and then combining it with unramified $\chi$-data for $R(S,G^0)$. According to \cite[\S6.2]{KalDC} we have ${^Lj_{S,G}}={^Lj_{G^0,G}}\circ{^Lj_{S,G^0}}$. Therefore $\theta_{G^0}=\theta$.

Corollary \ref{cor:pdclass1} gives a bijection of Deligne-Lusztig packets
\[ [\pi_{(jS,\theta_j)}^{G^0}] \to [\pi_{(jS,\theta_j)}^G]  \]
independent of any choices. Note that we are not using the shift by $\delta$ since we are allowing the source packet to be essentially of depth zero, rather than truly of depth zero. The above bijection is compatible with the identification $N(jS,G')(F)_{\theta_j}=N(jS,G'^0)(F)_{\theta_j}$ of \cite[Lemma 3.6.5]{KalRSP}. We have thus reduced the construction of the bijection \eqref{eq:intstr} to the case when $\varphi$ is essentially of depth zero, and we turn to this case next.

\subsection{Internal structure II: Depth zero DL-packets} \label{sub:is2}

We now assume that $\varphi$ is essentially of depth zero, i.e. $\varphi(P_F) \subset Z(\hat G)$. This implies $\hat M=\hat G$ and dually $G^0=G$. For $\eta \in H^1(u \to W,Z \to G)$ our goal is to construct a bijection $\tx{Irr}(\pi_0(S_\varphi^+),\eta) \to [\pi_\eta]$.

Let $\Omega(S,G)(F)_{\theta,\eta}$ be the mutual stabilizer in $\Omega(S,G)(F)$ of $\theta$ and $\eta$. We pull back the extension \eqref{eq:sphi+} along the inclusion $\Omega(S,G)(F)_{\theta,\eta} \to \Omega(S,G)(F)_\theta$ and obtain
\[ 1 \to \pi_0([\hat {\bar S}]^+) \to \pi_0(S_\varphi^+)_\eta \to \Omega(S,G)(F)_{\theta,\eta} \to 1. \]
We then push out along the character $\eta : \pi_0([\hat{\bar S}]^+) \to \C^\times$ and obtain an extension
\begin{equation} \label{eq:sp1}
1 \to \C^\times \to \Box_1 \to \Omega(S,G)(F)_{\theta,\eta} \to 1. \end{equation}
By Lemma \ref{lem:reppush} the set $\tx{Irr}(\pi_0(S_\varphi^+),\eta)$ is in canonical bijection with the set of $\tx{id}$-isotypic irreducible representations of $\Box_1$.

Write $(S,\hat j,\emptyset,\theta,(G',\xi,z),j)$ for the datum, unique up to isomorphism, such that $\tx{inv}(j_0,j)=\eta$. We have the extension
\[ 1 \to jS(F) \to N(jS,G')(F)_\theta \to N(jS,G')(F)_\theta/jS(F) \to 1. \]
By Lemma \ref{lem:c2} and \cite[Lemma 3.4.10]{KalRSP} the embedding $j$ gives an isomorphism $\Omega(S,G)(F)_{\theta,\eta} \to N(jS,G')(F)_\theta/jS(F)$. Thus, pushing out the above extension by $\theta\circ j^{-1}$ we obtain an extension
\begin{equation} \label{eq:sp2}
1 \to \C^\times \to \Box_2 \to \Omega(S,G)(F)_{\theta,\eta} \to 1. \end{equation}
By Lemma \ref{lem:reppush} and Corollary \ref{cor:pdclass1} the set of $\tx{id}$-isotypic irreducible representations of $\Box_2$ is in bijection with $[\pi_\eta]$. This bijection depends on the choice of normalization $\epsilon$ of the geometric intertwining operators. We are thus left with proving the following:

\begin{pro} \label{pro:exiso} The extensions \eqref{eq:sp1} and \eqref{eq:sp2} are isomorphic.
\end{pro}

The rest of this section is devoted to the proof of Proposition \ref{pro:exiso}. It will follow from an application of \cite[Proposition 6.2]{KalLLCD} once we have established the appropriate framework.

Fix a $\Gamma$-invariant pinning $(\hat T,\hat B,\{X_{\alpha^\vee}\})$ of $\hat G$ and modify the torally wild $L$-packet datum $(S,\hat j,\emptyset,\theta)$ within its isomorphism class so that $\hat j(\hat S)=\hat T$. Let $^Lj : {^LS} \to N(\hat T,\hat G) \rtimes W_F$ be the extension of $\hat j$ determined by unramified $\chi$ data; it is well-defined up to $\hat T$-conjugacy. Let $\varphi_S : W_F \to {^LS}$ be determined by $\varphi={^Lj} \circ \varphi_S$.

We obtain two actions of $W_F$ on $N(\hat T,\hat G)$, one by $\tx{Ad}(\varphi(w))$, and another by $\tx{Ad}({^Lj}(1 \rtimes w))$. The first action is the twist of the second action by $\hat j \circ \varphi_{S,0}$, where $\varphi_{S,0} \in Z^1(W_F,\hat S)$, is determined by $\varphi_S(w)=\varphi_{S,0}(w) \rtimes w$. In particular, both actions induce the same action on $\hat T$ and $\Omega(\hat T,\hat G)$, and this is the action that makes $\hat j : \hat S \to \hat T$ equivariant. But on $N(\hat T,\hat G)$ the two actions differ. The group of fixed points in $N(\hat T,\hat G)$ for the first action is $S_\varphi$.

\begin{fct} \label{fct:ga2}
The second action extends to $\Gamma$.
\end{fct}
\begin{proof}
It is enough to find a finite field extension $L/F$ s.t. $^Lj(1 \rtimes w)=1$ for $w \in W_L$, but if $L$ contains the splitting field of $S$ then the value $^Lj(1 \rtimes w)$ is given by the formula for $r_q(x)$ on \cite[p.237]{LS87}, with $x \in L^\times$ being the image of $w$ under the local reciprocity map $W_L^\tx{ab} \to L^\times$. Since all our $\chi$-data have finite order there exists $L$ for which this formula evaluates to $1$ for all $x \in L^\times$.	
\end{proof}

\begin{lem} \label{lem:sgamma}
Consider the exact sequence
\[ 1 \to \hat T_\tx{sc} \to N(\hat T_\tx{sc},\hat G_\tx{sc}) \to \Omega(\hat T,\hat G) \to 1 \]
with the action of $\Gamma$ given by $\tx{Ad}({^Lj}(1 \rtimes w))$, $w \in W_F$. Taking invariants we obtain the exact sequence
\[ 1 \to \hat S_\tx{sc}^\Gamma \to N(\hat T_\tx{sc},\hat G_\tx{sc})^\Gamma \to \Omega(S,G)(F) \to 1. \]
\end{lem}
\begin{proof}
That the first and third term have this shape follows from the fact that this action makes $\hat j : \hat S \to \hat T$ equivariant, as discussed above.

To show that an element $\mu \in \Omega(S,G)(F)$ lifts,  define $\hat j'=\hat j \circ \mu^{-1}$. We extend $\hat j'$ to $^Lj' : {^LS_\tx{ad}} \to N(\hat T_\tx{sc},\hat G_\tx{sc}) \rtimes W_F$ using the $\chi$-data for $R(S,G)$ from the torally wild $L$-packet datum $(S,\hat j,\chi,\theta)$. Note that the transport of this $\chi$-data to $\hat T$ via $\hat j$ is the same as the transport via $\hat j'$, because this $\chi$-data is $\Omega(S,G)(F)$-invariant. In particular, the restrictions to $W_F$ of $^Lj$ and $^Lj'$ agree. On the other hand, \cite[(2.6.2)]{LS87} says $^Lj' = \tx{Ad}(n_\mu)^{-1}\circ {^Lj}$, for a suitable $n_\mu \in N(\hat T_\tx{sc},\hat G_\tx{sc})$ lifting $\mu$. Thus, $n_\mu$ commutes with the restriction of $^Lj$ to $W_F$, i.e. $n_\mu \in N(\hat T_\tx{sc},\hat G_\tx{sc})^\Gamma$.
\end{proof}

Define $\mc{T}$ and $\mc{\bar T}$ to be the complex algebraic groups $N(\hat T,\hat G)$ and $N(\hat{\bar T},\hat{\bar G})$ equipped with the $\Gamma$-action that extends the action $\tx{Ad}({^Lj}(1 \rtimes w))$ for $w \in W_F$. Define $A$ to be the finite group $\Omega(S,G)$ with the natural $\Gamma$-action. Then $\mc{T}$ is an extension of $\hat S$ by $A$ that remains exact after taking $\Gamma$-fixed points by Lemma \ref{lem:sgamma}. We have $S_\varphi=\mc{T}^{\varphi_S(W_F)}$.

Let $x \in Z^1(u \to W,Z \to S)$ represent the class $\eta=\tx{inv}(j_0,j)$. Then $A^{[x]}=\Omega(S,G)(F)_\eta=N(jS,G')(F)/jS(F)$, $A^{[\varphi_S]}=\Omega(S,G)(F)_\theta$, and $A^{[x],[\varphi_S]}=[N(jS,G')(F)/jS(F)]_\theta$.

We write $\mc{E}_{[x]}^{0}$ for the extension obtained by taking the preimage $\mc{\bar T}^+$ of $\mc{T}^\Gamma$ in $\mc{\bar T}$, pulling back along the inclusion $A^{[x]} \to A^\Gamma$, and pushing out along $[x] : [\hat{\bar S}]^+ \to \C^\times$. We write $\mc{E}_{[x]}^{\varphi_S}$ for the extension obtained by pulling back the extension $S_\varphi^+$ along the inclusion $A^{[x],[\varphi_S]} \to A^{[\varphi_S]}$ and pushing out along $[x] : [\hat{\bar S}]^+ \to \C^\times$, i.e. the extension \eqref{eq:sp1}.

We now let $\tilde T$ be the algebraic group $N(j_0S,G)$ with its natural $F$-structure. It is an extension of $S$ by $A$. It remains exact after taking $\Gamma$-fixed points by \cite[Lemma 3.4.10]{KalRSP}. Write $\mc{E}_{[\varphi_S]}^0$ for the extension obtained by pulling back $\tilde T(F)$ along the inclusion $A^{[\varphi_S]} \to A^\Gamma$ and pushing out along $\theta$. The group $N(jS,G')$ is the inner twist $\tilde T_x$ of $\tilde T$ by the 1-cocycle $x$ in the sense of \cite{KalLLCD}. Write $\mc{E}_{[\varphi_S]}^{x}$ for the extension obtained by pulling back $\tilde T_x(F)$ along the inclusion $A^{[x],[\varphi_S]} \to A^{[x]}$ and pushing out along $\theta$, i.e. the extension \eqref{eq:sp2}.

Finally, let $\mc{E}_{[x]}^{0,[\varphi_S]}$ be obtained by pulling back $\mc{E}_{[x]}^{0}$ along the inclusion $A^{[x],[\varphi_S]} \to A^{[x]}$, and $\mc{E}_{[\varphi_S]}^{0,[x]}$ be obtained by pulling back $\mc{E}_{[\varphi_S]}^{0}$ along the inclusion $A^{[x],[\varphi_S]} \to A^{[\varphi_S]}$. Then \cite[Proposition 8.2]{KalLLCD} states that an isomorphism $\zeta : \mc{E}_{[\varphi_S]}^{0,[x]} \to \mc{E}_{[x]}^{0,[\varphi_S]}$ determines an isomorphism $\xi : \mc{E}_{[\varphi_S]}^{x} \to \mc{E}_{[x]}^{\varphi_S}$.

It therefore remains to establish the isomorphism $\zeta$, which we shall do by showing that both extensions are split. The choice of an isomorphism is then given by choices of splittings.

Write $N(\hat T_\tx{sc},\hat G_\tx{sc})^\Gamma_\theta$ for the preimage in $N(\hat T_\tx{sc},\hat G_\tx{sc})^\Gamma$ of $\Omega(S,G)(F)_\theta$, and analogously $N(\hat T_\tx{sc},\hat G_\tx{sc})^\Gamma_{\theta,\eta}$.

\begin{lem} \label{lem:gm1}
The extension
\[ 1 \to \hat S_\tx{sc}^\Gamma \to N(\hat T_\tx{sc},\hat G_\tx{sc})^\Gamma_{\theta} \to \Omega(S,G)(F)_{\theta} \to 1 \]
has multiplicity $1$ in the sense of Definition \ref{dfn:mult1}. \end{lem}

\begin{proof}
This is an extension of one finite abelian group by another, so applying Lemma \ref{lem:cliff1} it is enough to show that it has trivial commutator in the sense of Definition \ref{dfn:commtriv}.
Let $\theta_\tx{sc}$ be the restriction of $\theta$ to $S_\tx{sc}(F)$. The extension we are considering is the restriction of the extension
\begin{equation} \label{eq:2ct2} 1 \to \hat S_\tx{sc}^\Gamma \to N(\hat T_\tx{sc},\hat G_\tx{sc})^\Gamma_{\theta_\tx{sc}} \to \Omega(S,G)(F)_{\theta_\tx{sc}} \to 1, \end{equation}
so it is enough to show that this extension has trivial commutator. This extension breaks up according to the $F$-simple factors of $G$, and its commutator breaks up accordingly, so we may assume that $G$ is $F$-simple. Since $\theta_\tx{sc}$ is still non-singular, $\Omega(S,G)(F)_{\theta_\tx{sc}}$ is cyclic, and hence the commutator is trivial, except when $G$ is the restriction of scalars of a group of split type $D_{2n}$ by Lemma \ref{lem:weylab2}, in which case $\Omega(S,G)(F)_{\theta_\tx{sc}}$ may be isomorphic to $(\Z/2\Z)^2$.

In that latter case we argue as follows. The formation of $^Lj$ respects restriction of scalars, so we may assume that $G$ is split of type $D_{2n}$. The isomorphism class of the extension does not change if we conjugate $^Lj$ by an element of $\hat T$, so we may apply Lemma \ref{lem:ljx} to conclude that $^Lj$ restricts trivially to $I_F$. Let $f \in N(\hat T_\tx{ad},\hat G_\tx{ad})$ be the image of the Frobenius element in $W_F/I_F$ under $^Lj|_{W_F}$. Then $f$ is an elliptic element and the extension \eqref{eq:2ct2} is exactly the extension of Lemma \ref{lem:d2ng}.
\end{proof}

\begin{cor}
The extension $\mc{E}_{[x]}^0$ is split. Any extension of $\eta$ to $N(\hat T_\tx{sc},\hat G_\tx{sc})^\Gamma_{\theta,\eta}$ is a splitting.
\end{cor}
\begin{proof}
Pulling back the extension of Lemma \ref{lem:gm1} along the inclusion $A^{[x]} \to A^\Gamma$ and then pushing out along $[x] : \hat S_\tx{sc}^\Gamma \to \C^\times$ is another way to obtain $\mc{E}_{[x]}^0$. 
\end{proof}

\begin{lem} The character $\theta$ extends to $N(j_0S,G)(F)_\theta$.
\end{lem}
\begin{proof}
The proof is the same as for Lemma \ref{pro:dlcharext}, the only difference being that the equality $\Omega(S,G)(k)=N(S,G)(k)/S(k)$ used there and implied by Lang's theorem is replaced here by the equality $\Omega(S,G)(F)=N(j_0S,G)(F)/S(F)$ due to \cite[Lemma 3.4.10]{KalRSP}, which uses the fact that the point of $j_0S$ is absolutely special.
\end{proof}

\begin{cor}
The extension $\mc{E}_{[\varphi_S]}^0$ is split. Any extension of $\theta$ to $N(j_0S,G)(F)_\theta$ is a splitting of $\mc{E}_{[\varphi_S]}^0$.
\end{cor}
\begin{proof}
Immediate.
\end{proof}

This completes the construction of the bijection \eqref{eq:intstr} in the case when $\varphi$ is essentially of depth zero, to which the case of general $\varphi$ was reduced in the previous subsection. Let us now summarize the choice we have made in this construction and examine their interdependence. In the positive-depth case this is the choice of the embedding $j_0$, which ought to be determined by the Whittaker datum $\mf{w}$. In the depth-zero case, the choices are the extension $\tilde\theta$ of $\theta$, the extension $\tilde\eta$ for each $\eta$, and the normalization $\epsilon$ for each $j$.

Let us now argue that the choices of $\tilde\theta$ and normalization of $\epsilon$ for $j_0$ are linked, and for $\eta \neq 1$, the choices of $\tilde\eta$ and normalization $\epsilon$ are linked. Note that for $\eta_0=1$ we have the natural extension $\tilde\eta_0=1$. 

\begin{lem}
\begin{enumerate}
	\item Let $\delta \in \Omega(S,G)(F)_\theta^*$. If we replace $\tilde\theta$ and $\epsilon_{j_0}$ by $\tilde\theta \otimes \delta$ and $\epsilon_{j_0}\cdot \delta^{-1}$, the bijection \eqref{eq:intstr} remains unchanged.
	\item Let $\delta \in \Omega(S,G)(F)_{\theta,\eta}^*$ and let $j$ correspond to $\eta$. If we replace $\tilde\eta$ and $\epsilon_j$ by $\tilde\eta\otimes\delta$ and $\epsilon_j \cdot \delta$, the bijection \eqref{eq:intstr} remains unchanged.
\end{enumerate}
\end{lem}
\begin{proof}
The matching between the set of Deligne-Lusztig packets in $\Pi_\varphi(G)$ and the set of fibers of the restriction map $\tx{Irr}(\pi_0(S_\varphi^+)) \to \pi_0([\hat{\bar S}]^+)^*/\Omega(S,G)(F)_\theta$	depends only on the admissible embedding $j_0$. The choices of $\tilde\theta$, $\tilde\eta$, and $\epsilon_j$ only influence the bijection between $[\pi_j]$ and and $\tx{Irr}(\pi_0(S_\varphi^+),[\eta])$. This bijection is turn given by the bijection between $[\pi_j]$ and the set of $\tx{id}$-isotypic representations of $\mc{E}_{[x]}^{\varphi_S}$, as well as the isomorphism of extensions $\xi : \mc{E}_{[\varphi_S]}^{x} \to \mc{E}_{[x]}^{\varphi_S}$. The first of these depends on the choice of coherent splitting $\epsilon_j$, while the second comes from the isomorphism $\xi$ via \cite[Proposition 8.2]{KalLLCD} and hence depends on the choices of $\tilde\theta$ and $\tilde\eta$.

If we replace $\tilde\theta$ by $\delta\otimes\tilde\theta$ for some $\delta \in \Omega(S,G)(F)_\theta^*$, $\zeta$ is replaced by $\delta\cdot\zeta$, which according to \cite[Proposition 8.2]{KalLLCD} has the effect of replacing $\xi$ by $\delta\cdot\xi$, where we are identifying $\delta$ with its restriction to $\Omega(S,G)(F)_{\theta,\eta}$. If $\rho \in \tx{Irr}(N(jS,G)(F)_{\theta,\eta},\theta)$ corresponds to $\sigma \in \tx{Irr}(\pi_0(S_\varphi^+),[\eta])$ via $\xi$, then $\rho\otimes\delta$ corresponds to $\sigma$ via $\delta\cdot\xi$. If we replace $\tilde\eta$ by $\tilde\eta\otimes\delta$ for $\delta \in \Omega(S,G)(F)_{\theta,\eta}^*$, then $\zeta$ is replaced $\delta^{-1}\cdot\zeta$, hence $\xi$ is replaced by $\delta^{-1}\cdot\xi$. At the same time, Proposition \ref{pro:dzclass} shows that $\pi^{\delta^{-1}\epsilon_j}_{(S,\theta,\delta\otimes\rho)}=\pi^{\epsilon_j}_{(S,\theta,\rho)}$.
\end{proof}

\begin{rem} 
This lemma reduces the choices in the depth-zero case as follows: We may make arbitrary choices for $\tilde\theta$, as well as $\tilde\eta$ for $\eta \neq 1$. Once these are fixed, the real choice to be made is that of the coherent splittings $\epsilon$, one for each admissible embedding $j$.	
\end{rem}

\subsection{Remarks about stability and transfer} \label{sub:endo}

In the construction of the bijection $\tx{Irr}(\pi_0(S_\varphi^+)) \to \Pi_\varphi$ in the previous subsection we made a number of auxiliary choices. In this subsection we will sketch an argument showing that the resulting parameterized $L$-packet satisfies stability, and more generally endoscopic transfer for elements $s \in S_\varphi^+$ that lie in the subgroup $[\hat{\bar S}]^+$. The details of this argument, its extension to all $s \in S_\varphi^+$, and a canonical choice for the bijection $\tx{Irr}(\pi_0(S_\varphi^+)) \to \Pi_\varphi$, will be the subject of a forthcoming paper.

Thus we fix a rigid inner twist $(G',\xi,z)$ and consider the part $\Pi_\varphi(G',\xi,z)$ of the compound $L$-packet $\Pi_\varphi$ corresponding to it. We form the $s$-stable character
\[ \sum_{\substack{\rho \in \tx{Irr}(\pi_0(S_\varphi^+))\\ \rho \mapsto [z]}}\tx{tr}\rho(s)\Theta_{\pi_\rho}. \]
Since $\tx{tr}\rho(s)$ depends only on $\rho|_{[\hat{\bar S}]^+}$ we may break the sum as follows
\[ \sum_{\substack{\eta \in \pi_0(\hat S^\Gamma)^*/\Omega(S,G)(F)_\theta\\ \eta \mapsto [z]}} \sum_{\substack{\rho \in \tx{Irr}(\pi_0(S_\varphi^+))\\ \rho \mapsto \eta}}\tx{tr}\rho(s)\Theta_{\pi_\rho}, \]
and then use the fact that $\rho|_{\hat S^\Gamma}=[\Omega(S,G)(F)_\theta \cdot \eta]^{\oplus m(\eta,\rho)}$. According to Corollary \ref{cor:abext} and Proposition \ref{pro:exiso}, $m(\eta,\rho)$ is equal to the dimension of the representation of $N(jS,G)(F)_\theta$ corresponding to $\rho$, which in turn is equal to the multiplicity in $\pi_j$ of the irreducible constituent corresponding to $\rho$, i.e. $\pi_\rho$. Altogether the above double sum becomes
\[ \sum_{\substack{\eta \in \pi_0(\hat S^\Gamma)^*\\ \eta \mapsto [z]}} \eta(s)\Theta_{\pi_j}, \]
where $j$ is determined by $\tx{inv}(j_0,j)=\eta$. In other words, we have
\[ \sum_{j : S \to G'} \<\tx{inv}(j_0,j),s\> \Theta_{\pi_{(jS,\theta_j)}}, \]
where $j$ runs over the set of $G'(F)$-conjugacy classes of admissible embeddings $S \to G'$, and $j_0 : S \to G$ is the admissible embedding giving the base point in $\Pi_\varphi$.

To compute the character of $\pi_{(jS,\theta_j)}$ we apply recent \cite{Spice17} and ongoing work of DeBacker and Spice. It provides an inductive formula for the character of an irreducible supercuspidal representation obtained from J.K.Yu's construction in terms of the character of the corresponding representation of lower depth. Since the formula is additive, we may apply it to the reducible representation $\pi_{(jS,\theta_j)}$ and obtain an expression for its character in terms of the character of the depth-zero representation $\pi_{(G^0,S,\phi_{-1})}$, which in turn reduces to the character formula for the reducible Deligne-Lusztig induction $R_{\theta^\circ}$ and its extension $R_\theta$ from \S\ref{sec:nsdl-fin} and \S\ref{sub:nsdl-dz}. Altogether the formula we obtain for $\Theta_{\pi_{(jS,\theta_j)}}$ is virtually the same as that of \cite[\S4.8]{KalRSP} and the arguments of \cite[\S6]{KalRSP} apply.

The situation when $s \in S_\varphi^+$ does not lie in $[\hat{\bar S}]^+$ is rather different and much more subtle. For then $\tx{tr}\rho(s)$ depends not just on the restriction of $\rho$ to $[\hat{\bar S}]^+$. This means that different members of a given non-singular Deligne-Lusztig packet will contribute to the $s$-stable character with different weights.

\begin{appendices}

\renewcommand{\thesubsection}{\Alph{subsection}}

\subsection{Basic Clifford theory} \label{app:cliff}

We recall here basic facts about Clifford theory and offer some mild generalizations needed in this paper. 

Let $G$ be a group (not necessarily finite) and $\pi : G \to \tx{GL}(V)$ a representation (not necessarily finite-dimensional). We say that $V$ is irreducible if the only $G$-invariant subspaces are $V$ and $\{0\}$.

\begin{lem} \label{lem:ss}
The following are equivalent.
\begin{enumerate}
	\item $V$ is the sum of its irreducible subrepresentations.
	\item $V$ is a direct sum of irreducible subrepresentations.
	\item For every subrepresentation $W_1 \subset V$ there is a subrepresentation $W_2 \subset V$ s.t. $V=W_1 \oplus W_2$.
\end{enumerate}
\end{lem}
\begin{proof}
This is \cite[Proposition 2.2]{BH06}, where the assumption is made that $G$ is locally profinite and $V$ is smooth, but this assumption is not used in the proof.
\end{proof}

\begin{dfn} \label{dfn:ss}
A representation $V$ satisfying the assumptions of above Lemma is called \emph{semi-simple}.
\end{dfn}

\begin{fct} \label{fct:ssindres}
Let $H \subset G$ a subgroup of finite index. 
\begin{enumerate}
	\item A representation of $G$ is semi-simple if and only if its restriction to $H$ is.
	\item A representation of $H$ is semi-simple if and only if its induction to $G$ is.
\end{enumerate}	
\end{fct}
\begin{proof}
This is \cite[Lemma 2.7]{BH06}, where again the assumption that $G$ is locally profinite, $H$ is open, and $V$ is smooth, is not used. The only part stated here but not in loc. cit. is that if $\sigma$ is a representation of $H$ s.t. $\tx{Ind}_H^G\sigma$ is semi-simple, then $\sigma$ is semi-simple. But since there is an $H$-equivariant embedding $\sigma \to \tx{Ind}_H^G\sigma$ the statement is obvious.
\end{proof}

\begin{lem} \label{lem:cliff2} Let $B$ be a group, $A$ a normal subgroup of $B$, $\pi$ a \emph{finite-dimensional} irreducible representation of $B$ whose restriction to $A$ is semi-simple. We make no finiteness assumptions on $A$, $B$, or $C=B/A$. Then
\begin{enumerate}
	\item The set $S_\pi$ of irreducible constituents of $\pi|_A$ is a single $B/A$-orbit and each member of $S_\pi$ occurs with the same multiplicity $m_\pi$ in $\pi|_A$.
	\item If $B/A$ is abelian and $m_\pi=1$, then $\{\chi \in (B/A)^*|\chi\otimes\pi \cong \pi\}$ is the annihilator of the kernel of the action of $B/A$ on $S_\pi$.
	\item If $B/A$ is abelian and $\pi'$ is another finite dimensional irreducible representations of $B$ with semi-simple restriction to $A$ and s.t. $\tx{Hom}_A(\pi,\pi') \neq 0$, then $\pi'=\chi\otimes\pi$ for a character $\chi \in (B/A)^*$.
\end{enumerate}
\end{lem}
\begin{proof}
By assumption $\tx{Res}^B_A\pi = \bigoplus_{\sigma \in S_\pi} \sigma^{m_{\pi,\sigma}}$ for a (necessarily finite) set $S_\pi$ of irreducible representations of $A$ and positive integers $m_{\sigma,\pi}$. If $S' \subset S_\pi$ is $B$-invariant subset, then $\bigoplus_{\sigma \in S'} \sigma^{m_{\pi,\sigma}}$ is a $B$-subrepresentation, contradicting the irreducibility of $\pi$.

Since $\pi$ is finite-dimensional, $m_{\pi,\sigma}=\tx{dim}\tx{Hom}_A(\sigma,\pi)$ by Schur's lemma. Since $\pi|A\cong \pi\tx{Ad}(b)|_A$ for any $b \in B$ we see that $m_{\pi,\sigma}=m_{\pi,\sigma\circ\tx{Ad}(b)}$ and the claim follows from the previous paragraph.

Assume now that $C=B/A$ is abelian and $m_\pi=1$. The kernel of the action of $C$ on $S_\pi$ is the stabilizer $C_\sigma$ of one, hence any, $\sigma \in S_\pi$. It is a finite index subgroup of $C$. Let $B_\sigma$ be its preimage in $B$, a finite index subgroup of $B$. By Fact \ref{fct:ssindres} the representation $\pi|_{B_\sigma}$ is semi-simple, so we can write it as $\pi|_{B_\sigma}=\sigma_1 \oplus \dots \oplus \sigma_k$ with $\sigma_i$ and irreducible representation of $B_\sigma$. Then $\sigma_i|_A$ is a subrepresentation of $\pi|_A$ and hence semi-simple, so we can write it as $\sigma_{i,1}\oplus\dots\oplus\sigma_{i,k_i}$, where $\sigma_{i,j}$ are irreducible representations of $A$. The group $C_\sigma$ acts transitively on the set $\{\sigma_{i,1},\dots,\sigma_{i,k_i}\}$ by the previous point. On the other hand, $\sigma_{i,1}$ is an irreducible constituent of $\pi|_A$ and hence fixed by $C_\sigma$. It follows that $k_i=1$, i.e. $\sigma_i|_A$ is irreducible. Therefore Schur's lemma implies
\[ \tx{End}_A(\pi)=\tx{End}_{B_\sigma}(\pi) = \tx{Ind}_{B_\sigma}^B\C = \tx{Ind}_1^{C/C_\sigma}\C = \bigoplus_{\chi \in (C/C_\sigma)^*} \chi. \]

Keep the assumption that $C=B/A$ is abelian, but drop the assumption that $m_\pi=1$. Let $\pi'$ be a finite-dimensional irreducible representation of $B$ s.t. $\pi'|_A$ is semi-simple and $\tx{Hom}_A(\pi,\pi') \neq \{0\}$. Then $\tx{Hom}_\C(\pi,\pi')$ is a finite-dimensional representation of $B$, isomorphic to $\pi^\vee \otimes \pi'$. By a theorem of Chevalley \cite{SerTP} this is a semi-simple representation of $B$. Therefore, the subrepresentation $\tx{Hom}_A(\pi,\pi')$ is also semi-simple. But this is a representation of $C$. Since $C$ is abelian Schur's lemma implies that every finite-dimensional irreducible representation is 1-dimensional. Let $\chi$ be a character of $C$ that occurs in $\tx{Hom}_A(\pi,\pi')$. Then $\tx{Hom}_A(\pi\otimes\chi,\pi')$ contains the trivial character of $C$, i.e. $\tx{Hom}_B(\pi\otimes\chi,\chi') \neq 0$.
\end{proof}

\begin{rem} If we write $\tx{Irr}_{A-\tx{ss}}(B)$ for the set of finite-dimensional irreducible representations of $B$ whose restriction to $A$ is semi-simple, the above lemma shows that restriction to $A$ gives a well-defined map $\tx{Irr}_{A-\tx{ss}}(B) \to \tx{Irr}(A)/C$ and shows that $C^*$ acts transitively on the fibers of this map when $C$ is abelian. When $m_\pi=1$, the kernel of the action of $C^*$ on the fiber through $\pi$ and the kernel of the action of $C$ on $S_\pi$ are mutually annihilators.
\end{rem}

\begin{exa}
The simplest example that shows the necessity of the assumption $m_\pi=1$ above is given by the quaternion group $Q$, which is a non-abelian central extension
\[ 0 \to \Z/2\Z \to Q \to \Z/2\Z \oplus \Z/2\Z \to 0. \]
It has five irreducible representations, four of which are the characters of $\Z/2\Z \oplus \Z/2\Z$, pulled back to $Q$, and one of which is $2$-dimensional and on which $\Z/2\Z$ acts by its non-trivial character. This $2$-dimensional $\pi$ is preserved under tensor product by all characters of $\Z/2\Z \oplus \Z/2\Z$, while the unique irreducible constituent of its restriction to $\Z/2\Z$ is preserved by all elements of $\Z/2\Z \oplus \Z/2\Z$. Moreover, $\tx{End}_{\Z/2\Z}(\pi)$ is the regular representation of $\Z/2\Z \oplus \Z/2\Z$.
\end{exa}

By a finite-dimensional projective representation of a finite group $C$ we shall understand a set-theoretic map $\tau : C \to \tx{GL}(V)$, where $V$ is a finite-dimensional complex vector space, such that for $c_1,c_2 \in C$ there exists $z(c_1,c_2) \in \C^\times$ with $\tau(c_1)\tau(c_2)=z(c_1,c_2)\tau(c_1c_2)$. It is immediate that then $z \in Z^2(C,\C^\times)$.

Let $1 \to A \to B \to C \to 1$ be an extension of finite groups.

\begin{dfn} \label{dfn:mult1}
We shall say that the extension has the multiplicity 1 property if for every irreducible representation $\pi$ of $B$, every irreducible constituent of $\pi|_{A}$ occurs with multiplicity $1$.
\end{dfn}

Of particular interest for us will be the situation where both $A$ and $C$ are abelian. To this extension we can associate a commutator function as follows. Let $s : C \to B$ be a set-theoretic section. Then $s(c_1)s(c_2)s(c_1)^{-1}s(c_2)^{-1}$ is an element of $A$, which we call $f(c_1,c_2)$. If the action of $C$ on $A$ is trivial, then $f(c_1,c_2)$ does not depend on the choice of $s$. In general it does, but the image of $f(c_1,c_2)$ in the group $A_{\<c_1,c_2\>}$ of coinvariants of $A$ for the action of the subgroup of $C$ generated by $c_1$ and $c_2$ does not. It thus makes sense to ask if the image of $f(c_1,c_2)$ in $A_{\<c_1,c_2\>}$ is trivial.

\begin{dfn} \label{dfn:commtriv} We shall say that the extension $B$ has trivial commutator if for all $c_1,c_2 \in C$ the image of $f(c_1,c_2)$ in $A_{\<c_1,c_2\>}$ is trivial.
	
\end{dfn}

\begin{lem} \label{lem:exc} A central extension of a finite abelian group by $\C^\times$ is split if and only if it is abelian.
\end{lem}
\begin{proof} Clearly a split central extension of $\C^\times$ is abelian. Conversely let $1 \to \C^\times \to B \to C \to 1$ be a central extension and assume that $B$ is abelian. Given $c \in C$ there exists a lift $\dot c \in B$ such that $\tx{ord}(\dot c)=\tx{ord}(c)$. Fix an isomorphism $C \to \prod_{i=1}^k (\Z/n_i\Z)$ and let $c_j$ be the preimage of $e_j=(0,\dots,0,1,0,\dots 0) \in \prod_i (\Z/n_i\Z)$ for all $1 \leq j \leq k$. Let $\dot c_j \in B$ be a lift of $c_j$ of the same order. The restriction of $B \to C$ to the subgroup of $B$ generated by $\dot c_1,\dots,\dot c_k$ is an isomorphism.
\end{proof}

\begin{fct} \label{fct:projdim} An irreducible projective representation of the finite abelian group $C$ is of dimension $1$ if and only if its 2-cocycle $z \in Z^2(C,\C^\times)$ is cohomologically trivial.
\end{fct}
\begin{proof}
If the 2-cocycle is cohomologically trivial, i.e. if there is $f : C \to \C^\times$ with $z(c_1,c_2)=f(c_1) \cdot f(c_2)f(c_1c_2)^{-1}$, then $c \mapsto f(c)\tau(c)$ is an irreducible linear representation of $C$ on the vector space $V$, so $V$ is one-dimensional.

Conversely if $V$ is one-dimensional then $\tau : C \to \C^\times$ is the cochain whose coboundary is $z$.
\end{proof}

\begin{lem} \label{lem:cliff1} Let $B$ be a finite group, $A \subset B$ a normal subgroup, $C=B/A$.
\begin{enumerate}
\item The map $\tx{Irr}(B) \to \tx{Irr}(A)/C$ obtained by restriction is surjective.
\item We have $\tx{dim}(\pi)=|S_\pi|\cdot m_\pi \cdot \tx{dim}(\rho)$, for any $\rho \in S_\pi$.
\item Assume that $C$ is abelian. Let $\rho \in S_\pi$ and let $C_\rho \subset C$ be its stabilizer. If $H^2(C_\rho,\C^\times)=0$ then $m_\pi=1$.
\end{enumerate}
Assume now that both $A$ and $C$ are abelian.
\begin{enumerate}[resume]
\item We have $m_\pi=1$ if and only if $\rho(f(c_1,c_2))=1$ for all $c_1,c_2 \in C_\rho$.
\item The extension has the multiplicity one property if and only if its commutator is trivial.
\end{enumerate}
\end{lem}
\begin{proof}
For surjectivity, let $\rho \in \tx{Irr}(A)$, acting on a complex vector space $W$, and let $B_\rho$ be the stabilizer of the isomorphism class of $\rho$ for the action of $B$ on $\tx{Irr}(A)$. Choose a set of representatives $b_1,\dots,b_k$ for $C_\rho:=B_\rho/A$. For each $b_i$ choose a $T_i  \in \tx{Aut}_\C(V)$ giving an isomorphism $\rho\circ\tx{Ad}(b_i)^{-1} \to \rho$ of representations of $A$. Define a map $\tilde\rho : B_\rho \to \tx{Aut}_\C(W)$ by $\tilde\rho(b_ia)=T_i \circ \rho(a)$ for all $a \in A$. Then $\tilde\rho$ is a projective representation whose associated 2-cocycle $z \in Z^2(B_\rho,\C^\times)$, defined by $z(b_1,b_2)=\tilde\rho(b_1) \circ \tilde\rho(b_2) \circ \tilde\rho(b_1b_2)^{-1}$, is immediately checked to be inflated from $C_\rho$. Let $\tau$ be an irreducible projective representation of $C_\rho$ with 2-cocycle $z^{-1}$; it exists, c.f. \cite[Theorem 1.3]{Tapp77}. Then $\tilde\rho \otimes \tau$ is a linear representation of $B_\rho$. The map $\tx{End}_\C(\tau) \to \tx{End}_A(\tilde\rho\otimes\tau)$ given by $f \mapsto \tx{id}\otimes f$ is an isomorphism of $C_\rho$-representations, hence $\tilde\rho\otimes\tau$ is irreducible.
By Mackey's test, the induction $\pi$ of $\tilde\rho\otimes\tau$ to $B$ remains irreducible
and is a preimage of $\rho$ under the map $\tx{Irr}(B) \to \tx{Irr}(A)/C$.

We have $\tx{Res}^B_{B_\rho}\pi = \bigoplus_{b \in B_\rho\lmod B}(\tilde\rho\otimes\tau)^b$, all summands being pairwise non\-isomorphic. We see $|S_\pi|=[B:B_\rho]$, $m_\pi=\tx{dim}(\tau)$, and
\[ \tx{dim}(\pi)=[B:B_\rho]\cdot \tx{dim}(\tau)\cdot\tx{dim}(\rho) = |S_\pi|\cdot m_\pi\cdot \tx{dim}(\rho). \]

Assume now that $C=B/A$ is abelian. If $H^2(C_\rho,\C^\times)=0$, then $\tau$ is $1$-dimensional by Fact \ref{fct:projdim}.

Assume now that $A$ is also abelian. Let $\pi \in \tx{Irr}(B)$, let $\rho : A \to \C^\times$ be a character belonging to $S_\pi$, $C_\rho \subset C$ its stabilizer, and $z \in Z^2(C_\rho,\C^\times)$ the cocycle associated to some extension of $\rho$ to a projective representation of $B_\rho$, as in the previous part of the proof. The commutator function of the pushout of $1 \to A \to B_\rho \to C_\rho \to 1$ along $\rho$ is the composition of $f|_{C_\rho \times C_\rho}$ with $\rho$. Lemma \ref{lem:exc} implies that this pushout is a split extension of $C_\rho$ by $\C^\times$ if and only if $\rho(f(c_1,c_2))=1$ for all $c_1,c_2 \in C_\rho$. The cocycle $z$ represents the class of this extension. Fact \ref{fct:projdim} implies that $m_\pi=\tx{dim}(\tau)=1$ is equivalent to $\rho(f(c_1,c_2))=1$ for all $c_1,c_2 \in C_\rho$.

Given a subgroup $C' \subset C$ we can pull back the extension $1 \to A \to B \to C \to 1$ along the inclusion $C' \to C$ and then push it out along the projection $A \to A_{C'}$, where $A_{C'}$ is the group of $C'$-coinvariants of $A$. The result is a central extension $B'$ of $C'$ by $A_{C'}$. This extension is abelian, for every subgroup $C' \subset C$, if and only if the image of $f(c_1,c_2)$ in $A_{\<c_1,c_2\>}$ is trivial for all $c_1,c_2 \in C$. If these equivalent statements hold, then what we just proved implies $m_\pi=1$ for all $\pi \in \tx{Irr}(B)$, because $\rho \in S_\pi$ factors through $A_{C_\rho}$.

Assume now the converse -- there exists $C' \subset C$ such that the extension $B'$ is not abelian. Hence there exist $c_1,c_2 \in C'$ s.t. $1 \neq f(c_1,c_2) \in A_{C'}$. Let $\rho : A_{C'} \to \C^\times$ be a character s.t. $\rho(f(c_1,c_2)) \neq 1$. We inflate $\rho$ to a character of $A$ and let $C_\rho$ be its stabilizer in $C$. By construction $C' \subset C_\rho$. Now $\rho : A \to \C^\times$ descends to a character of $A_{C_\rho}$. Let $B_\rho$ be the central extension of $C_{\rho}$ by $A_{C_\rho}$. Pushing it out by $\rho$ we obtain a central extension of $C_\rho$ by $\C^\times$, non-abelian because its commutator at $c_1,c_2$ is $\rho(f(c_1,c_2)) \neq 1$. Applying Lemma \ref{lem:exc} we see that this extension is non-split. Thus the 2-cocycle $z \in Z^2(C_\rho,\C^\times)$ that corresponds to $\rho$ is not cohomologically trivial. By Fact \ref{fct:projdim} the irreducible projective representations $\tau$ of $C_\rho$ with this cocycle have dimension greater than $1$, implying $m_\pi>1$ for any $\pi \in \tx{Irr}(B)$ with $\rho \in S_\pi$.
\end{proof}

\begin{cor} \label{cor:m1stab}
The multiplicity 1 property for extensions of finite abelian groups is stable under pull-backs, push-outs, and Cartesian products.
\end{cor}
\begin{proof}
Given an extension $1 \to A \to B \to C \to 1$ with commutator function $f$, the commutator function of the pull-back along a homomorphism $ i : C' \to C$ is the restriction $f \circ (i \times i)$, that of the push-out along a $C$-equivariant homomorphism $p : A \to A'$ is the composition $p \circ f$, and the commutator function of the product of two extensions $B_1$ and $B_2$ is the product $(f_1,f_2)$.
\end{proof}

\subsection{A basis theorem} \label{app:basis}

We record here some remarks on an abstract form of the Harish-Chandra basis theorem.

Let $G$ and $N$ be locally profinite groups. Let $S \subset N$ be an open normal subgroup of finite index. Let $\Pi$ be a smooth semi-simple finite length representation of $G \times N$ on a complex vector space $V$ and assume that $S$ acts via a character $\theta$. Then for any $n \in N/S$ the subspace $\C\Pi(1 \times n) \subset \tx{End}_G(\Pi|_G)$ is well-defined (of dimension at most 1). Furthermore, we can decompose $\Pi = \bigoplus (\pi\otimes\rho)^{\oplus m_{\pi,\rho}}$, where $\pi$ and $\rho$ run over the set of irreducible smooth representations of $G$ and $N$, respectively, and $m_{\pi,\rho}$ are natural numbers. Then the condition $m_{\pi,\rho} \neq 0$ defines a correspondence 
\[ [\Pi|_G] \stackrel{m}{\longleftrightarrow} \tx{Irr}(N,\theta) \]
between the set $[\Pi|_G]$ of irreducible constituents of the $G$-representation $\Pi|_G$, and the set $\tx{Irr}(N,\theta)$ of irreducible representations of $N$ whose restriction to $S$ is $\theta$-isotypic.

\begin{lem} \label{lem:hcbp}
The following two statements are equivalent.
\begin{enumerate}
	\item The subspaces $\C\Pi(1 \times n)$ of $\tx{End}_G(\Pi|_G)$ indexed by $n \in N/S$ are both linearly independent and generating.
	\item We have $m_{\pi,\rho} \in \{0,1\}$ and the correspondence $m$ is a bijection.
\end{enumerate}
\end{lem} 
\begin{proof}
Write $M_{\pi,\rho}=\tx{Hom}_{G \times N}(\pi\otimes\rho,\Pi)$, so that 
\[ \Pi=\bigoplus_{\pi,\rho}\pi\otimes\rho\otimes M_{\pi,\rho},\qquad m_{\pi,\rho}=\tx{dim}_\C(M_{\pi,\rho}). \] 
Then
\begin{eqnarray*}
\tx{End}_G(\Pi)&=&\bigoplus_{\pi,\rho,\pi',\rho'} \tx{Hom}_G(\pi,\pi')\otimes\tx{Hom}_\C(\rho,\rho')\otimes\tx{Hom}_\C(M_{\pi,\rho},M_{\pi',\rho'})\\
&=&\bigoplus_{\pi,\rho,\rho'} \tx{Hom}_\C(\rho,\rho')\otimes\tx{Hom}_\C(M_{\pi,\rho},M_{\pi,\rho'}).\\
\end{eqnarray*}
and for $n \in N$ the image of $\Pi(1 \times n)$ under these isomorphisms is $\rho(n)\otimes\tx{id}$ in the components indexed by $(\pi,\rho,\rho)$ with $m_{\pi,\rho}>0$ and $0$ in the components indexed by $(\pi,\rho,\rho')$ with $\rho \neq \rho'$ or in the components indexed by $(\pi,\rho,\rho)$ with $m_{\pi,\rho}=0$.

$2 \Rightarrow 1$: Then we have $\tx{End}_G(\Pi)=\bigoplus_\rho \tx{End}_\C(\rho)$
and the claim follows from the orthogonality relations for projective representations of the finite group $N/S$.

$1 \Rightarrow 2$: We first use the fact that $\{\Pi(1 \times n)|n \in N\}$ is generating to bound $m_{\pi,\rho}$. Consider the subspace given by the condition $\rho'=\rho$, i.e.
\[ \bigoplus_\rho \tx{End}_\C(\rho) \otimes \left(\bigoplus_\pi \tx{End}_\C(M_{\pi,\rho})\right). \]
The set $\{\Pi(1 \rtimes n)|n \in N\}$ is contained in that subspace, so this must then be the whole space. Therefore for fixed $\pi$ we have $m_{\pi,\rho}>0$ for at most one $\rho$. Fixing now $\rho$, the only elements obtained from $\Pi(1 \times n)$ are of the form $\rho(n)\otimes(\oplus_\pi \tx{id}_{M_{\pi,\rho}})$, so 1 implies again that for each $\rho$ there is at most one $\pi$ with $m_{\pi,\rho}>0$ and moreover that then $m_{\pi,\rho}=1$. 

It is clear that for each $\pi$ there is $\rho$ with $m_{\pi,\rho}>0$, by virtue of $\pi \subset \Pi|_G$. It remains to show that conversely for a given $\rho$ there does exist a $\pi$ with $m_{\pi,\rho}>0$, and for this we use the linear independence of $\{\C\Pi(1 \times n)|n \in N/S\}$. We consider again the above displayed space. Each $\rho$-summand has dimension $\tx{dim}(\rho)^2$, so the entire space has dimension $\sum_\rho m_{\pi,\rho}\tx{dim}(\rho)^2 \leq |N/S|$ and linear independence implies that equality must hold, i.e. for every $\rho$ there does exist a $\pi$ with $m_{\pi,\rho}=1$.
\end{proof}

\begin{rem} The second statement in the Lemma above can be equivalently stated as follows: For each $\rho \in \tx{Irr}(N,\theta)$ the $\rho$-isotypic component $\pi_\rho$ of $\Pi|_N$ is $G$-irreducible and the map $\rho \mapsto \pi_\rho$ is a bijection $\tx{Irr}(N,\theta) \to [\Pi|_G]$.
\end{rem}

Let $G$ be a locally profinite group, $H \subset G$ an open (and hence closed) normal subgroup. Let $N \subset G$ be a closed subgroup, write $N_H = N \cap H$, and let $S \subset N_H$ be an abelian open normal subgroup of $N$ of finite index. 

The group $N$ acts on $G$ by conjugation and we can form $G \rtimes N$. Note that 
\[ N \to G \rtimes N,\qquad n \mapsto (n^{-1},n) \]
is an injective group homomorphism that embeds $N$ as a normal subgroup of $G \rtimes N$ that commutes with $G$ and therefore provides an isomorphism $G \rtimes N \to G \times N$.

We have the subgroup $H \rtimes N$ of $G \rtimes N$, normal if $G/H$ is abelian. Let $\sigma$ be a smooth finite-length semi-simple representation of $H \rtimes N$ on a complex vector space $V$, such that the central subgroup $\{(s^{-1},s)|s \in S\}$ 
acts by a smooth character $\theta$ of $S$.

\begin{pro} \label{pro:hcbi}
Assume that
\begin{enumerate}
	\item The set $\{\sigma(n^{-1} \rtimes n)|n \in N_H/S\}$ forms a basis of $\tx{End}_H(\sigma)$;
	\item For each $g \in G$ the representation $\sigma|_H^{g \rtimes 1}$ is isomorphic to $\sigma|_H$ if $g \in H \cdot N$ and disjoint from $\sigma|_H$ otherwise.
\end{enumerate}
Then the set $\{\tx{Ind}_{H\rtimes N}^{G\rtimes N}(\sigma)(n^{-1} \rtimes n)|n \in N/S\}$ forms a basis of $\tx{End}_G(\tx{Ind}_H^G\sigma)$.
\end{pro}
\begin{proof}
The proof is just a matter of unwinding the definitions. For $n \in N$ let $\beta(n) \in \tx{End}_G(\tx{Ind}_H^G\sigma)$ denote $\tx{Ind}_{H\rtimes N}^{G\rtimes N}(\sigma)(n^{-1} \rtimes n)$.

Choose a set of representatives $\dot g$ for the coset space $G/H$ such that the cosets in $N/N_H$ are represented by elements $\dot g$ of $N$. By Frobenius reciprocity and the Mackey theorem we have the isomorphism
\[ \tx{End}_G(\tx{Ind}_H^G\sigma) \to \bigoplus_{g \in G/H} \tx{Hom}_H\left(\sigma^{\dot g},\sigma\right) = \bigoplus_{n \in N/N_H} \tx{End}_H\left(\sigma\right), \]
where the equality is due to our second assumption. 

For a representative $\dot n \in N$ and $h \in N_H$ this isomorphism translates $\beta(\dot nh)$ to the tuple of homomorphisms that has all coordinates trivial except for the coordinate corresponding to $\dot nH \in G/H$, where it is
\[ \sigma(1 \rtimes \dot n)\sigma(h^{-1}\rtimes h). \]
The first assumption now implies that as $\dot nh$ runs over $N/S$ these elements form a basis.
\end{proof}

\subsection{Representations of extensions with abelian quotient} \label{app:repext}

Let $G$ be a locally profinite group and $N \subset G$ an open (and hence closed) normal subgroup such that $A=G/N$ is finite and abelian. We will collect some basic facts about the relationship between the finite-dimensional smooth irreducible representations of $G$ and those of $N$, writing $\tx{Irr}(G)$ and $\tx{Irr}(N)$ for the respective sets of isomorphism classes. For $\rho \in \tx{Irr}(N)$ we write $G_\rho$ and $A_\rho=G_\rho/N$ for the stabilizers in $G$ and $A$ of the isomorphism class $\rho$. For $\pi \in \tx{Irr}(G)$ we write
\[ m(\sigma,\pi) = \tx{dim}\tx{Hom}_G(\pi,\tx{Ind}_N^G\sigma)=\tx{dim}\tx{Hom}_N(\sigma,\tx{Res}_N^G\pi), \]
noting that $\tx{Ind}_N^G\sigma$ and $\tx{Res}_N^G\pi$ are semi-simple representations \cite[Lemma 2.7]{BH06}. Here $\tx{Ind}_N^G\sigma$ is defined as in the case of finite groups, and coincides with both smooth induction and compact induction due to the finiteness of $A$, implying that the functor $\tx{Ind}_N^G$ is both left and right adjoint to $\tx{Res}_N^G$, see \cite[\S2]{BH06}.

\begin{lem} Given $\sigma,\sigma' \in \tx{Irr}(N)$ the representations $\tx{Ind}_N^G\sigma$ and $\tx{Ind}_N^G\sigma'$ are either equal or disjoint. They are equal if and only if there exists $g \in G$ s.t. $\sigma'=\sigma\circ\tx{Ad}(g)$.
\end{lem}
\begin{proof}
If $\sigma'=\sigma\circ\tx{Ad}(g)$ then clearly $\tx{Ind}_N^G\sigma$ and $\tx{Ind}_N^G\sigma'$. Assume conversely that $\tx{Ind}_N^G\sigma$ and $\tx{Ind}_N^G\sigma'$ have a common irreducible constituent $\pi$. Then both $\sigma$ and $\sigma'$ are irreducible constituents of $\tx{Res}_N^G\pi$, hence conjugate under $G$ by Lemma \ref{lem:cliff2}.
\end{proof}

We now want to describe, for a given $\sigma$, the function $\tx{Irr}(G) \to \Z, \pi \mapsto m(\sigma,\pi)$, i.e. the decomposition of $\tx{Ind}_N^G\sigma$.

\begin{lem} \label{lem:abext1} Let $\sigma \in \tx{Irr}(N)$ and let $N \subset H \subset G_\sigma$ be a subgroup to which $\sigma$ extends. Then $H$ is maximal with this property if and only if for one, hence any, extension $\sigma_H$ of $\sigma$ to $H$ we have $G_{\sigma_H}=H$.
\end{lem}
\begin{proof}
We first claim that $G_{\sigma_H}$ depends only on $H$ and $\sigma$, but not on $\sigma_H$. For this, let $\sigma_H$ and $\sigma_H'$ be two extensions of $\sigma$ to $H$. By Lemma \ref{lem:cliff2} there exists a character $\chi_H \in (H/N)^*$ s.t. $\sigma_H' = \sigma_H \otimes \chi_H$. Since $A$ acts trivially on $H/N$ by conjugation we have $\chi_H(ghg^{-1})=\chi_H(h)$ for all $g \in G$ and $h \in H$. Thus for $f \in \tx{End}_\C(V_\sigma)$, $g \in G$, and $h \in H$ the conditions $f \circ \sigma_H'(ghg^{-1})=\sigma_H'(h) \circ f$ and $f \circ \sigma_H(ghg^{-1})=\sigma_H(h) \circ f$ are equivalent, and the claim follows.

Assume $G_{\sigma_H}=H$. Let $H \subset H' \subset G_\sigma$ be such that there is an extension $\sigma_{H'}$ of $\sigma$ to $H'$. Then $\sigma_{H'}|_H$ is an extension of $\sigma$ to $H$ that is clearly stabilized by $H'$, hence $H' \subset G_{\sigma_H}$, which by our assumption implies $H'=H$.

Assume conversely that $H$ is a maximal subgroup to which $\sigma$ extends. Choose an extension $\sigma_H$ of $\sigma$ to $H$. Let $g \in G$ be such that $\sigma_H \circ \tx{Ad}(g) \cong \sigma_H$. Let $H' \subset G_{\sigma_H}$ be the group generated by $H$ and $g$. Then $\sigma_H$ extends to a projective representation of $H'$ whose associated cohomology class lies in $H^2(H'/H,\C^\times)$. Since $H'/H$ is finite and cyclic we have $H^2(H'/H,\C^\times)=0$ and this projective representation linearizes, providing an extension of $\sigma_H$ to $H'$. The maximality of $H$ implies $H'=H$, i.e. $g \in H$, and we conclude $G_{\sigma_H}=H$.
\end{proof}

\begin{lem} \label{lem:abext2} Let $\sigma \in \tx{Irr}(N)$ and $\pi \in \tx{Irr}(G)$. Let $N \subset H \subset G$ be a maximal subgroup to which $\sigma$ extends. Then $m(\sigma,\pi)>0$ if and only if there exists an extension $\sigma_H$ of $\sigma$ to $H$ such that $\pi = \tx{Ind}_H^G\sigma_H$.
\end{lem}
\begin{proof}
If $\sigma_H$ is such an extension and $\pi=\tx{Ind}_H^G\sigma_H$ then Frobenius reciprocity implies $m(\pi,\sigma_H)>0$, hence $m(\pi,\sigma)>0$. Conversely assume $m(\pi,\sigma)>0$. Pick arbitrarily an extension $\sigma_H'$ and let $\pi'=\tx{Ind}_H^G\sigma_H'$. Then $\pi'$ is semi-simple by Fact \ref{fct:ssindres} and we can check its irreducibility by computing the dimension of its space of self-intertwiners, which is $1$ according to \cite{Ku77} and the fact that $H = G_{\sigma_H}$ of Lemma \ref{lem:abext1}. Moreover $m(\pi',\sigma)>0$ as just argued. According to Lemma \ref{lem:cliff2} there $\chi \in A^*$ s.t. $\pi=\pi'\otimes\chi=\tx{Ind}_H^G(\sigma_H'\otimes\chi|_H)$. Then $\sigma_H=\sigma_H'\otimes\chi|_H$ satisfies the requirement.
\end{proof}

Let $\sigma \in \tx{Irr}(N)$, let $N \subset H \subset G$ be a maximal subgroup to which $\sigma$ extends and let $A'=H/N$. We will construct an injective group homomorphism $A_\sigma/A' \to (A')^*$.

Consider $\tx{Ind}_N^H\sigma$. The set of its irreducible constituents is precisely the set of irreducible representations of $H$ whose restriction to $N$ contains $\sigma$. Thus $A_\sigma$ acts on this set and this action factors through $A_\sigma/A'$. Fix an extension $\sigma_H$ of $\sigma$ to $H$. By Lemma \ref{lem:abext2} the stabilizer of $\sigma_H$ in $A_\sigma$ equals $A'$. At the same time, we can apply Lemma \ref{lem:cliff2} and to conclude that the set of irreducible representations of $H$ whose restriction to $N$ contains $\sigma$ is precisely $\{\sigma_H\otimes\chi|\chi \in (A')^*\}$ and since $m(\sigma,\sigma_H)=1$ the map $\chi \mapsto \sigma_H \otimes\chi$ is injective. Thus, for a given $a \in A_\sigma/A'$ there is a unique $\chi_a \in (A')^*$ with $\sigma_H \circ \tx{Ad}(a)^{-1} = \sigma_H \otimes \chi_a$. Note that $\chi_a$ does not depend on the choice of extension $\sigma_H$ of $\sigma$. This implies that $a \mapsto \chi_a$ is a homomorphism. It is injective because the action of $A_\sigma/A'$ on the set $\{\sigma_H\otimes\chi\}$ is simple.

\begin{cor} \label{cor:abext} Let $\sigma \in \tx{Irr}(N)$ and $\pi \in \tx{Irr}(G)$ be s.t. $m(\sigma,\pi)>0$. Let $N \subset H \subset G$ be a maximal subgroup to which $\sigma$ extends and let $A'=H/N$. Then
\begin{enumerate}
\item $m(\sigma,\pi)=[A':A_\sigma]$.
\item $\tx{dim}(\pi)=\tx{dim}(\sigma)m(\sigma,\pi)[A_\sigma:A]=\tx{dim}(\sigma)[A':A]$.
\item We have the Cartesian square
\[ \xymatrix{
	\tx{Stab}(\pi,A^*)\ar[r]\ar[d]&A^*\ar[d]\\
	A_\sigma/A'\ar[r]&(A')^*
}
\]
\end{enumerate}
\end{cor}
\begin{proof}
Lemma \ref{lem:abext2} immediately gives $\tx{dim}(\pi)=\tx{dim}(\sigma)[A':A]$, while $\tx{dim}(\pi)=\tx{dim}(\sigma)m(\sigma,\pi)[A_\sigma:A]$ is immediate from Lemma \ref{lem:cliff2}.

To compute $\tx{Stab}(\pi,A^*)$, let $\chi \in A^*$. Write $\pi=\tx{Ind}_H^G\sigma_H$ as in Lemma \ref{lem:abext2}. Then $(\tx{Ind}_H^G \sigma_H)\otimes\chi = \tx{Ind}_H^G(\sigma_H \otimes\chi|_{A'})$ and by \cite{Ku77} this is isomorphic to $\tx{Ind}_H^G\sigma_H$ if and only if there exists $a \in A$ s.t. $\sigma_H\circ\tx{Ad}(a)^{-1}=\sigma_H\otimes\chi|_{A'}$. Restricting this relation from $H$ to $N$ we see $a \in A_\sigma$, and this relation becomes equivalent to $\chi|_{A'}=\chi_a$.
\end{proof}

\begin{lem} \label{lem:reppush} Assume that $N$ is abelian and $\sigma$ is a character. Then the set of $\pi \in \tx{Irr}(G)$ with $m(\sigma,\pi)>0$ is in canonical bijection with the set of $\tx{id}$-isotypic representations of the pushout of $G_\sigma$ by $\sigma$.
\end{lem}
\begin{proof}
Write $\tx{Irr}(G,\sigma)$ for the subset of $\pi \in \tx{Irr}(G)$ with $m(\sigma,\pi)>0$. By Lemma \ref{lem:cliff2} the $\pi_\sigma \in \tx{Irr}(G_\sigma,\sigma)$ are precisely those whose restriction to $N$ is $\sigma$-isotypic. This implies that if $g \in G$ is s.t. $\pi_\sigma \cong \pi'_\sigma \circ \tx{Ad}(g)$ then $g \in G_\sigma$ and hence $\pi_\sigma \cong \pi'_\sigma$. In particular $G_{\pi_\sigma}=G_\sigma$. Since $\pi=\tx{Ind}_{G_\sigma}^G\pi_\sigma$ is semi-simple, this implies via \cite{Ku77} that it is irreducible, and moreover that $\pi \cong \pi'$ implies $\pi_\sigma \cong \pi'_\sigma$. Thus $\pi_\sigma \to \pi$ is an injective map $\tx{Irr}(G_\sigma,\sigma) \to \tx{Irr}(G,\sigma)$, which is also surjective by Lemma \ref{lem:abext2}. The bijection between $\tx{Irr}(G_\sigma,\sigma)$ and $\tx{Irr}(G_\sigma \times_\sigma \C^\times,\tx{id})$ is immediate.
\end{proof}

\subsection{DL-varieties and homomorphisms with abelian kernel and cokernel} \label{app:dlreview}

We review here the material \cite[\S1.21-\S1.27]{DL76}. Let $\tilde G \to G$ be a homomorphism of connected reductive groups defined over a finite field $k$, with abelian kernel and cokernel. Let $S \subset G$ be a maximal torus, $\tilde S \subset \tilde G$ its inverse image, $\theta : S^F \to \bar\Q_l^\times$ a character, and $\tilde\theta : \tilde S^F \to \bar\Q_l^\times$ its pullback. Let $U \subset \tilde G$ be the unipotent radical of a Borel subgroup containing $\tilde S$. Then $U \subset G$ is also the unipotent radical of a Borel subgroup containing $S$.

The following results are proved in loc. cit. when $\tilde G$ is the simply connected cover of the derived subgroup of $G$. The proof given there works without change for the more general $\tilde G$ considered here:

\begin{enumerate}
	\item The natural map $\tx{cok}(\tilde S^F \to S^F) \to \tx{cok}(\tilde G^F \to G^F)$ is bijective.
	\item The $(G \times^Z S)^F$-space $Y^G_U$ is the induction of the $(\tilde G \times^{\tilde Z} \tilde S)^F$-space $Y^{\tilde G}_U$.
	\item We have an isomorphism $H^i_c(Y^G_U,\bar\Q_l)=\tx{Ind}_{\tilde S^F}^{S^F}(H^i_c(Y^{\tilde G}_U,\bar\Q_l))$ as modules for the action of $S^F$ on the right.
	\item The natural map $Y^{\tilde G}_U \to Y^G_U$ induces an isomorphism $H^i_c(Y^{\tilde G}_U,\bar\Q_l)_{\tilde\theta} \to H^i_c(Y^G_U,\bar\Q_l)_\theta$.
	\item For $\chi : \tx{cok}(\tilde G^F \to G^F) \to \bar\Q_l^\times$ we have
	\[ H^i_c(Y_U,\bar\Q_l)_{\theta\cdot\chi} = \chi \otimes H^i_c(Y_U,\bar\Q_l)_\theta. \]
\end{enumerate}

\subsection{Remarks about embeddings of tori}

Consider two reductive groups $G_1$ and $G_2$ and a rigid inner twist $(\xi,z) : G_1 \to G_2$, maximal tori $S_i \subset G_i$, and $g \in G(\bar F)$ such that $\xi\circ\tx{Ad}(g) : S_1 \to S_2$ is defined over $F$. Then $\xi\circ\tx{Ad}(g)$ induces an isomorphism $\Omega(S_1,G_1) \to \Omega(S_2,G_2)$ defined over $F$. We have the action of $\Omega(S_1,G_1)(F)$ on $H^1(F,S_1)$ induced by the action on $S_1$. Write $\delta : \Omega(S_1,G_1)(F) \to H^1(\Gamma,S_1)$ for the connecting homomorphism. Write $\eta_g$ for the class in $H^1(u \to W,Z \to S_1)$ of the 1-cocycle $w \mapsto g^{-1}z(w)\sigma_w(g)$. Given $w \in \Omega(S_1,G_1)(F)$ the class $\eta_{g \dot w}$ is independent of the choice of lift $\dot w \in N(S_1,G_1)(\bar F)$ of $w$ and will be denoted by $\eta_{gw}$. The proof of the following lemma is elementary and left to the reader.

\begin{lem}  \label{lem:c2}
Let $w \in \Omega(S_1,G_1)(F)$. Then
\begin{enumerate}
	\item $\eta_{gw}=w^{-1}\eta_g \cdot \delta w$;
	\item The image of $w$ in $\Omega(S_2,G_2)(F)$ belongs to the subgroup $N(S_2,G_2)(F)/S_2(F)$ if and only if $\eta_{gw}=\eta_g$;
	\item Assume that $N(S_1,G_1)(F)/S_1(F) = \Omega(S_1,G_1)(F)$. Then $\xi\circ\tx{Ad}(g)$ identifies the stabilizer of $\eta_g$ in $\Omega(S_1,G_1)(F)$ with the subgroup $N(S_2,G_2)(F)/S_2(F)$ of $\Omega(S_2,G_2)(F)$.
\end{enumerate}
\end{lem}

\subsection{Parahoric subgroups and restriction of scalars} \label{app:parahoric}

Let $E/F$ be a finite extension of non-archimedean local fields, not necessarily tamely ramified. Let $H$ be a connected reductive group over $E$ and $G=\tx{Res}_{E/F}H$. There is a natural identification $\mc{B}(G,F)=\mc{B}(H,E)$. Let $x$ be a point in this building, and $\mf{G}_x^\circ$ and $\mf{H}_x^\circ$ the corresponding (connected) parahoric group schemes defined over $O_F$ and $O_E$, respectively.

\begin{fct} The identity $G=\tx{Res}_{E/F}H$ extends to an isomorphism $\mf{G}_x^\circ=\tx{Res}_{O_E/O_F}\mf{H}_x^\circ$.
\end{fct}
\begin{proof} The $O_F$-group scheme $\mf{G}_x^\circ$ is affine according to \cite[\S4.6.2]{BT2}. Since Weil restriction preserves affineness, so is $\tx{Res}_{O_E/O_F}\mf{H}_x^\circ$. It will thus be enough to show that the $F$-algebra isomorphism between the coordinate rings of $G$ and $\tx{Res}_{E/F}H$ maps the coordinate ring of $\mf{G}_x^\circ$ bijectively onto the coordinate ring of $\tx{Res}_{O_E/O_F}\mf{H}_x^\circ$. Since $\mf{G}_x^\circ$ is smooth by loc. cit. and then so is $\tx{Res}_{O_E/O_F}\mf{H}_x^\circ$ by \cite[Proposition A.5.2]{CGP15}. We may thus apply \cite[Proposition 1.7.6]{BT2} to compute the coordinate rings and see that it is enough to show that under the identification $G(F^u)=\tx{Res}_{E/F}H(F^u)$ the subgroups $\mf{G}_x^\circ(O_{F^u})$ and $\tx{Res}_{O_E/O_F}\mf{H}_x^\circ(O_{F^u})$ become identified.

Let $F' \subset E$ be the maximal unramified subsextension of $F$. We have the compatible isomorphisms $O_E \otimes_{O_F} O_{F^u} \to O_{E^u}^{[F':F]}$ and $E \otimes_F F^u \to (E^u)^{[F':F]}$, giving rise to the compatible isomorphisms $\tx{Res}_{O_E/O_F}\mf{H}_x^\circ(O_{F^u}) \to \mf{H}_x^\circ(O_{E^u})^{[F':F]}$ and $\tx{Res}_{E/F}H(F^u) \to H(E^u)^{[F':F]}$, which show that $\tx{Res}_{O_E/O_F}\mf{H}_x^\circ(O_{F^u})$ is the parahoric subgroup of $\tx{Res}_{E/F}H(F^u)$ corresponding to the point $x$. Under the equality $G=\tx{Res}_{E/F}H$ this group is identified with $\mf{G}_x^\circ(O_{F^u})$
\end{proof}

Recall that we denote by $\ms{G}_x^\circ$ and $\ms{H}_x^\circ$ the reductive quotients of the special fibers of $\mf{G}_x^\circ$ and $\mf{H}_x^\circ$. The proof of the following Lemma was communicated to us by Brian Conrad.

\begin{lem} The isomorphism $\mf{G}_x^\circ = \tx{Res}_{O_E/O_F}\mf{H}_x^\circ$ induces an isomorphism $\ms{G}_x^\circ = \tx{Res}_{k_E/k_F}\ms{H}_x^\circ$.
\end{lem}
\begin{proof}
We apply base change to $k_F$ and use that Weil restriction of scalars commutes with base change to reduce to showing that the reductive quotient of $\tx{Res}_{A/k_F}\mf{H}_x^\circ$ is $\tx{Res}_{k_E/k_F}\ms{H}_x^\circ$, where $A = O_E \otimes_{O_F} k_F$ and we are now reusing the symbol $\mf{H}_x^\circ$ to denote the base-change of the original $\mf{H}_x^\circ$ to $A$. Note that $\mf{H}_x^\circ$ is still smooth connected affine and the reductive quotient of its special fiber is still $\ms{H}_x^\circ$.

Let $\mf{\bar H}_x^\circ$ denote the special fiber of $\mf{H}_x^\circ$. Reduction modulo the maximal ideal of $A$ gives a surjective morphism $\tx{Res}_{A/k}\mf{H}_x^\circ \to \tx{Res}_{k_E/k_F}\mf{\bar H}_x^\circ$ of $k_F$-groups with connected unipotent kernel (apply \cite[Proposition A.5.12]{CGP15} to successive powers of the maximal ideal of $A$).

The projection $\mf{\bar H}_x^\circ \to \ms{H}_x^\circ$ is a smooth surjective morphism of $k_E$-groups with connected unipotent kernel $U$. Applying $\tx{Res}_{k_E/k_F}$ to it gives a surjective morphism $\tx{Res}_{k_E/k_F}\mf{\bar H}_x^\circ \to \tx{Res}_{k_E/k_F}\ms{H}_x^\circ$ of $k_F$-groups with kernel given by the smooth affine $k_F$-group $\tx{Res}_{k_E/k_F}(U)$, see \cite[Proposition A.5.2(4) and Proposition A.5.14(3)]{CGP15}. This group is moreover connected and unipotent, for it is enough to check this over $k_E$, where it becomes $U^{[k_E:k_F]}$.
\end{proof}

\subsection{Absolutely special vertices} \label{app:absvert}

Let $G$ be a connected reductive group defined over $F$.

\begin{dfn} \label{dfn:abspec}
A point $x \in \mc{B}(G,F)$ is called \emph{absolutely special} if it is special in $\mc{B}(G,E)$ for every finite Galois extension $E/F$.	
\end{dfn}

Assume for a moment that $G$ is quasi-split. We recall some material due to Bruhat-Tits. A $\Gamma$-invariant pinning of $G$ provides a point in $\mc{B}(G,F)$ -- the Chevalley valuation corresponding to the pinning \cite[4.2.3]{BT2}. Fix such a point $o \in \mc{B}(G,F)$. For each $a \in R(A_T,G)_\tx{res}$ we have the sets $\Gamma_a$ and $\Gamma_a'$ defined in \cite[4.2.20]{BT2}. In the special case at hand where the valuation is discrete with image $\Z$ and the residual characteristic is not $2$, they are given as follows. Let $a \in R(A_T,G)_\tx{res}$ be a non-divisible root and let $\alpha \in R(T,G)$ be a lift. If $a$ is non-multipliable then $\Gamma_a=\Gamma_a'=e_\alpha^{-1}\Z$. If $a$ is multipliable then $\Gamma_a=\frac{1}{2}e_\alpha^{-1}\Z$, $\Gamma_a'=e_\alpha^{-1}\Z$, $\Gamma_{2a}=\Gamma_{2a}'=2e_\alpha^{-1}(\Z+\frac{1}{2})$. Here $e_\alpha$ is the ramification index of the extension $F_\alpha/F$. Note that the second case occurs only if $\alpha$ belongs to a component of type $A_{2n}$ and a power of the action of tame inertia preserves this component and maps $\alpha$ to $\alpha'$ s.t. $\beta=\alpha+\alpha'$ is also a root, in which case $2a$ is the image of $\beta$ and $e_{\beta}=\frac{1}{2}e_\alpha$. A point $x \in \mc{A}(T,F)$ is called \emph{special} if $\<a,x-o\> \in \Gamma_a$ for all non-divisible $a \in R(A_T,G)_\tx{res}$ \cite[4.6.15]{BT2}, \cite[6.2.13]{BT1}. It suffices to check this condition for the simple roots $a$ \cite[6.2.14]{BT1} corresponding to some choice of positive roots.

It is thus clear that the Chevalley valuation $o$ is special, and in fact absolutely special. The next lemma shows that the absolutely special points are precisely the Chevalley valuations.

\begin{lem} \label{lem:abspec} The following are equivalent
\begin{enumerate}
	\item The point $x \in \mc{A}(T,F)$ is absolutely special.
	\item $\<a,x-o\> \in \Gamma_a'$ for all simple $a \in R(A_T,G)$ relative to a Borel subgroup $T \subset B$ defined over $F$.
	\item $x=t\cdot o$ for some $t \in T_\tx{ad}(F)$.
	\item $x$ is a Chevalley valuation.
\end{enumerate}
\end{lem}
\begin{proof}
$1 \Rightarrow 2$: The point $x$ remains special in $\mc{B}(G,F_\alpha)$, thus $\<\alpha,x-o\> \in \Gamma^{F_\alpha}_\alpha$, where $\Gamma^{F_\alpha}_\alpha$ is the set $\Gamma_\alpha$ introduced above, but relative to the base field $F_\alpha$ and the valuation on $F_\alpha$ that extends the valuation on $F$. For the valuation on $F_\alpha$ with image $\Z$ we would have $\Gamma^{F_\alpha}_\alpha=\Z$, because $\alpha$ is now non-multipliable, but upon rescaling the valuation so that its image is $e_\alpha^{-1}\Z$ the set $\Gamma^{F_\alpha}_\alpha$ also rescales and becomes $e_\alpha^{-1}\Z$ and thus equal to $\Gamma_a'$.

$2 \Rightarrow 3$: The set of absolute simple roots $\Delta \subset R(T,G)$ provides an isomorphism
\[ T_\tx{ad} \to \prod_{\alpha \in \Delta/\Gamma} \tx{Res}_{F_\alpha/F}\mb{G}_m, \quad t \mapsto (\alpha(t))_\alpha. \]
Choose $t_\alpha \in F_\alpha^\times$ with $\tx{val}(t_\alpha)=\<a,x-o\>$ and $t_{\sigma(\alpha)}=\sigma(t_\alpha)$ for all $\sigma \in \Gamma$. The collection $(t_\alpha)$ is an $F$-point of the right-hand side and determines $t \in T_\tx{ad}(F)$ with $\alpha(t)=\<a,x-o\>$ for all $\alpha \in \Delta$ with image $a \in R(A_T,G)$. Then $x=t \cdot o$.

$3 \Rightarrow 4$: Immediate.

$4 \Rightarrow 1$: Immediate.
\end{proof}

\begin{lem} \label{lem:absqs}
If $\mc{B}(G,F)$ has an absolutely special point, then $G$ is quasi-split.
\end{lem}
\begin{proof}
Replacing $G$ by its adjoint group effects neither the assumption nor the conclusion of the lemma, so we assume that $G$ is adjoint. Let $T \subset G$ be a maximally unramified torus defined over $F$ such that its split subtorus $A_T$ is a maximal split torus. Such $T$ exists due to \cite[Corollary 5.1.12]{BT2}. Let $T' \subset T$ be the maximal unramified subtorus and let $F'/F$ be the splitting extension of $T'$, a finite unramified extension. The apartment $\mc{A}(A_T,F) \subset \mc{B}(G,F)$ is equal to the Frobenius-fixed points of the apartment $\mc{A}(T',F')$ of $\mc{B}(G,F')$. All apartments of $\mc{B}(G,F)$ are of this form. Therefore we may assume that $x \in \mc{A}(A_T,F)$. 

Now $T$ is a minimal Levi subgroup of the quasi-split group $G \times F^u$. Since $T=\tx{Cent}(T',G)$ and $T' \times F'$ is split, we see that $T$ is a minimal Levi subgroup of $G \times F'$. Thus $G \times F'$ is quasi-split.

Since $x$ is a absolutely special point of $\mc{B}(G,F')$, Lemma \ref{lem:abspec} applied to $G \times F'$ shows that $x$ is a Chevalley valuation. Thus there exists an $F'$-pinning $(T,B,\{X_\alpha\})$ of $G \times F'$ giving rise to $x$. Let $\sigma$ denote Frobenius. There exists a unique $g \in G(F')$ such that $\tx{Ad}(g)\sigma(T,B,\{X_\alpha\})=(T,B,\{X_\alpha\})$. Since $x$ is $\sigma$-fixed, both pinnings $\sigma(T,B,\{X_\alpha\})$ and $(T,B,\{X_\alpha\})$ induce the same Chevalley point $x$. This implies $g \in G(F')_x$. Since $x$ is special the stabilizer $G(F')_x$ equals the parahoric $G(F')_{x,0}$: By \cite[Proposition 4.6.28]{BT2} we have $G(F')_x=N(T,G)(F')_x\cdot G(F')_{x,0}$; since $x$ is special this product equals $T(F')_b \cdot G(F')_{x,0}$; since $G$ is adjoint we have $T(F')_b=T(F')_0 \subset G(F')_{x,0}$. The triviality of $H^1(\<\sigma\>,G(F')_{x,0})$ implies the existence of $h \in G(F')_{x,0}$ such that $h^{-1}\sigma(h)=g$. Then $h(T,B,\{X_\alpha\})$ is another $F'$-pinning of $G \times F'$ giving rise to $x$, which is now moreover fixed by $\sigma$, and hence is an $F$-pinning of $G$.
\end{proof}

\begin{cor} The simple roots in $R(A_T,G)_\tx{res}$ are contained (and thus give a set of simple roots) in $R(\ms{A_T}^\circ,\ms{G}_x^\circ)$ if and only if $x$ is absolutely special.
\end{cor}
\begin{proof} This follows immediately from the description \cite[4.6.12+4.6.23]{BT2} of $R(\ms{A_T}^\circ,\ms{G}_x^\circ)$ as the subset of $R(A_T,G)_\tx{res}$ consisting of those $a$ for which $\<a,x-o\> \in \Gamma_a'$.
\end{proof}

\begin{rem}
In \cite{KalRSP} we introduced the notion of a superspecial point. We recall that $x \in \mc{B}(G,F)$ is called \emph{superspecial} if it is special in $\mc{B}(G,F')$ for all finite \emph{unramified} extensions $F'/F$. When $G$ is unramified, then the notions of absolutely special, superspecial, and hyperspecial, all agree. When $G$ is ramified, hyperspecial vertices do not exist, and the notions of absolutely special and superspecial are a replacement. Clearly an absolutely special point is superspecial. The converse however is false, as the following example shows.
\end{rem}

\begin{exa} 

Consider a ramified unitary group in 3 variables. Let $T$ be a maximally split maximal torus and $A \subset T$ the maximal split subtorus. Then $X^*(A)=\Z$. We have $X^*(T)=\Z^3$ with inertia acting by $(a,b,c) \mapsto (-c,-b,-a)$. Then $X^*(A)$ is the torsion-free quotient of the coinvariants of this action, so we have the isomorphism $X^*(A) \to \Z$ given by $(a,b,c) \mapsto a-c$. Let $e \in X^*(A)$ be the preimage of $1 \in \Z$. The relative root system is of type $BC_1$ and is given by $\{e,2e,-e,-2e\} \subset X^*(A)$. The root $e$ is the restriction of both $e_1-e_2$ and $e_2-e_3$ in $\Z^3=X^*(T)$. The root $2e$ is the restriction of $e_1-e_3$.
	
We have $\Gamma_e=\frac{1}{4}\Z$, $\Gamma_e'=\frac{1}{2}\Z$, and $\Gamma_{2e}=\Gamma_{2e}'=\frac{1}{2}+\Z$. Let $o \in \mc{A}(T,F)$ be the absolutely special point given by an $F$-pinning. Identify $\mc{A}(T,F)$ with $X_*(A) \otimes \R$ by sending $o$ to $0$. Consider $x \in X_*(A) \otimes \R = \mc{A}(T,F)$ determined by $e(x)=\frac{1}{4}$. Then $x$ is a vertex, because it is a wall of an affine root with gradient $2e$. It is also special. However, it is not absolutely special. Indeed, the splitting field $E/F$ of the unitary group is a ramified quadratic extension and for all absolute roots $\alpha$ we have $\Gamma_\alpha=\Gamma_\alpha'=\frac{1}{2}\Z$ with respect to the normalized valuation on $F$, so $\<e_1-e_2,x\> = \frac{1}{4} \notin \Gamma_{e_1-e_2}$.

Slightly more generally one can consider a ramified unitary group in $2n+1$ variables. There are two special vertices, with connected reductive quotients $\tx{SO}(2n+1)$ and $\tx{Sp}(2n)$, respectively. The first is absolutely special, while the second is superspecial but not absolutely special. There are also other vertices, which are non-special, and have connected reductive quotient $\tx{SO}(2a+1) \times \tx{Sp}(2b)$ with $a+b=n$.
\end{exa}

It turns out that the odd ramified unitary groups provide the only examples of superspecial vertices that are not absolutely special, as the following argument due to Gopal Prasad shows.

\begin{pro}[Gopal Prasad] \label{pro:supergp}
Let $G$ be an absolutely almost simple group defined over $F$ that does not split over $F^u$. If $G$ is not of type $A_{2n}$ every superspecial vertex is absolutely special.
\end{pro}
\begin{proof}
Write $K=F^u$ and consider the base change of $G$ to $K$. It is a quasi-split group and we let $L/K$ be the splitting extension of $G$. Write $H$ for $G \times L$. Consider a special vertex $x \in \mc{B}(G,K)$ and let $\ms{G}_x^\tx{ss}$ and $\ms{H}_x^\tx{ss}$ be the semi-simple quotients of the special fibers of the parahoric groups schemes of $G$ and $H$ at $x$, respectively. These are connected reductive groups defined over the algebraic closure of a finite field and we have a natural embedding $\ms{G}_x^\tx{ss} \to \ms{H}_x^\tx{ss}$. Assume that $x$ is not a special vertex in $\mc{B}(G,L)$. This, it is either a non-special vertex or is contained in a facet of positive dimension. We will show that if $G$ is not of type $A_{2n}$ then dimension considerations rule out the existence of such an embedding $\ms{G}_x^\tx{ss} \to \ms{H}_x^\tx{ss}$.

There are four possible cases to consider: 
\begin{enumerate}
	\item $G$ is of type $E_6^{(2)}$. Then $\ms{G}_x^\tx{ss}$ is of type $F_4$, while $\ms{H}_x^\tx{ss}$ is of type $D_5$ or $A_5$.
	\item $G$ is of type $A_{2n+1}^{(2)}$. Then $\ms{G}_x^\tx{ss}$ is of type $C_{n+1}$ while $\ms{H}_x^\tx{ss}$ is either of type $A_{2n}$ or a product of groups of type $A_r$ and $A_{2n+1-r}$. 
	\item $G$ is of type $D_n^{(2)}$. Then $\ms{G}_x^\tx{ss}$ is of type $B_{n-1}$, while $\ms{H}_x^\tx{ss}$ is either of type $A_{n-1}$, or of type $D_{n-1}$, or a product of two groups of type $D_r$ and $D_{n-r}$.
	\item $G$ is of type $D_4^{(3)}$ or $D_4^{(6)}$. In both of these cases $\ms{G}_x^\tx{ss}$ is of type $G_2$, while $\ms{H}_x^\tx{ss}$ is either product of four copies of a group of type $A_1$ or a group of type $A_3$. 
\end{enumerate}
\end{proof}

\begin{cor} Let $G$ be a connected reductive group over $F$. If $G$ has a superspecial vertex, then it is quasi-split.
\end{cor}
\begin{proof}
We may assume without loss of generality that $G$ is adjoint. Then it is a product of $F$-simple factors and we may consider an individual such factor. It is of the form $\tx{Res}_{E/F}H$ for an absolutely simple adjoint group $H$ defined over a finite extension $E/F$. We have $\mc{B}(G,F)=\mc{B}(H,E)$ and $G$ is quasi-split over $F$ if and only if $H$ is quasi-split over $E$. Thus we may assume that $G$ is a absolutely simple.

If $G$ is of type $A_{2n}^{(2)}$ over $F^u$, then it is automatically quasi-split over $F$. If $G$ splits over an unramified extension, then the vertex is hyperspecial, so $G$ is quasi-split. If $G$ does not split over an unramified extension, then Proposition \ref{pro:supergp} shows that the vertex is absolutely special and the result follows from Lemma \ref{lem:absqs}.
\end{proof}

\subsection{Generic depth-zero supercuspidal representations} \label{app:gendz}

In this section we generalize \cite[\S6.1]{DR09} to the case of ramified groups. 

Let $G$ be a connected reductive group defined over $F$. Let $B=TU \subset G$ be a Borel subgroup defined over $F$ and let $\psi : U(F) \to \C^\times$ be a generic character. The $G(F)$-conjugacy class of the pair $(B,\psi)$ is called a Whittaker datum, and we shall denote it by $\mf{w}$.

Let $x \in \mc{A}(G,F)$. The image of $B \cap G(F^u)_{x,0}$ in $\ms{G}_x^\circ(\bar k_F)$ is a Borel subgroup defined over $k_F$, call it $\ms{B}$. Its unipotent radical $\ms{U}$ is the image of $U \cap G(F^u)_{x,0}$. We say that $\psi$ has depth zero at $x$ if it is trivial on $U \cap G(F)_{x,0+}$ and the character $\psi_x$ of $\ms{U}(k_F)$ it induces is generic.

\begin{lem} \label{lem:dzgenchar} Let $x \in \mc{B}(G,F)$ be a vertex.
\begin{enumerate}
	\item If $\psi$ has depth zero at $x$, then $x$ is absolutely special.
	\item If $\psi$ has depth zero at $x$ and $y$, then $x=y$.
	\item If $x$ is absolutely special and $\psi_x$ is a generic character of $\ms{U}(k_F)$, then it is the restriction of some generic $\psi : U(F) \to \C^\times$ that has depth zero at $x$.
\end{enumerate}
\end{lem}
\begin{proof}
If $x$ is not absolutely special, above Corollary implies that there exists a simple root of $R(A_T,G)_\tx{res}$ that is not contained in $R(\ms{A_T}^\circ,\ms{G}_x^\circ)$. Since $x$ is a vertex, the root system $R(\ms{A_T}^\circ,\ms{G}_x^\circ)$ has the same rank as the root system $R(A_T,G)_\tx{res}$, so there exists a non-simple root of $R(A_T,G)$ that is simple in $R(\ms{A_T}^\circ,\ms{G}_x^\circ)$. The character $\psi$ is trivial on the corresponding root subgroup, and thus its restriction to $\ms{U}(k_F)$ is not generic.

If $\psi$ has depth zero at $x$ and $y$, then both $x$ and $y$ are absolutely special, so there exists $t \in T_\tx{ad}(F)$ with $y=tx$. For $a \in \Delta(A_T,G)$ with lift $\alpha \in \Delta(T,G)$ the images of $U_{a,y}(O_F)$ and $U_{a,x}(O_F)$ in the 1-dimensional $F_\alpha$-vector space $U_a(F)/[U_a,U_a](F)$ are two $O_{F_\alpha}$-lattices, the second being obtain from the first by multiplication by $a(t) \in F_\alpha^\times$. If $\psi$ has depth zero at $x$ and $y$, then these lattices must agree, i.e. $a(t) \in O_{F_\alpha}^\times$. This holds for all $a \in \Delta(A_T,G)$, thus $t \in T_\tx{ad}(O_F)$, and hence $y=x$.

Assume now that $x$ is absolutely special and $\psi_x$ is a generic character of $\ms{U}(k_F)$. The inclusions $U_a \to U$ combine to an isomorphism
\[ \prod_{a \in \Delta(A_T,G)} U_a/[U_a,U_a] \to U/[U,U] \]
of $F$-groups, and the same is true over $k_F$. Note that $[U_a,U_a]=U_{2a}$. For each $a \in \Delta(A_T,G)_\tx{res}$ the reduction map $U_{a,x}(O_F) \to \ms{U}_a(k_F)$ is surjective and its kernel contains $U_{2a,x}(O_F)=U_{a,x}(O_F) \cap U_{2a}(F)$. It follows that the character $\psi_x$ induces a character of $\prod_{a \in \Delta(A_T,G)}U_{a,x}(O_F)/(U_{a,x}(O_F) \cap [U_a,U_a](F))$ that is non-trivial on each factor. Since $U_a(F)/[U_a,U_a](F)$ is locally compact abelian, and this character has finite order, it can be extended to $\prod_a U_a(F)/[U_a,U_a](F)$ by Pontryagin duality.
\end{proof}

Let $\pi$ be a depth zero supercuspidal representation of $G(F)$. According to \cite[Proposition 6.8]{MP96} there exists a vertex $z \in \mc{B}(G,F)$, an irreducible representation $\rho$ of $G(F)_x$, and a cuspidal irreducible representation $\sigma$ of $\ms{G}_x^\circ(k_F)$ s.t. $\pi=\tx{c-Ind}_{G(F)_x}^{G(F)}\rho$ and $\sigma \subset \rho|_{G(F)_{x,0}}$.

\begin{pro} \label{pro:dzgen} The following are equivalent
\begin{enumerate}
	\item $\pi$ is $\mf{w}$-generic.
	\item $x$ is absolutely special, there exists $t \in T(F)$ s.t. $\psi'=\psi\circ\tx{Ad}(t)$ has depth zero at $x$, and $\sigma$ is $\psi'_x$-generic.
\end{enumerate}
\end{pro}
\begin{proof}
By \cite{Ku77}, $\tx{Hom}_{U(F)}(\pi,\psi)$ is the product of $\tx{Hom}_{G(F)_x \cap gUg^{-1}(F)}(\rho,\psi^g)$ as $g$ runs over $G(F)_x \lmod G(F) /U(F)$. According to the Iwasawa decomposition this double coset space is equal to $N(F)_x \lmod N(F)$, where $N=N(T,G)$. The natural map $G_\tx{sc} \to G$ restricts to an isomorphism between the preimage of $U$ in $G_\tx{sc}$ and $U$, and this implies $G(F)_x \cap U(F) = G(F)_{x,0} \cap U(F)$, and the same for $U$ replaced by $nUn^{-1}$, $n \in N(F)$. The irreducible representations in the restriction $\rho|_{G(F)_{x,0}}$ are the $G(F)_x$-conjugates of $\sigma$, so we are looking at the product of $\tx{Hom}_{G(F)_{x,0} \cap gUg^{-1}(F)}(\sigma,\psi^g)$ as $g$ runs over $G(F)_{x,0} \lmod G(F) / U(F) = N(F)_{x,0} \lmod N(F)$.

If $\pi$ is $\mf{w}$-generic, then there exists $g \in N(F)$ for which the corresponding factor is non-trivial. In particular, the restriction of $\psi^g$ to $G(F)_{x,0+} \cap gUg^{-1}(F)$ is trivial, and thus $\psi^g$ induces a character of $g\ms{U}g^{-1}(k_F)$. If $x$ is not absolutely special, this character is not generic, and this contradicts the cuspidality of $\sigma$.

Conversely, if $x$ is absolutely special and $\sigma$ is $\psi_x'$ is generic, then the factor corresponding to $t \in T(F)$ is non-zero.
\end{proof}

\subsection{A study of $D_{2n}$} \label{app:d2n}

We will consider two parallel situations involving the split Spin group in $4n$ variables. The first situation is the following:

Let $\hat G_\tx{ad}$ be a complex semi-simple group of adjoint type $D_{2n}$ and let $\hat G_\tx{sc}$ be its simply connected cover. Let $s \in \hat G_\tx{ad}$ be a regular semi-simple element whose centralizer $\hat G_\tx{ad}^s$ has component group $(\Z/2\Z)^2$. Write $\hat T_\tx{ad} \subset \hat G_\tx{ad}$ for the connected centralizer of $s$, a maximal torus. The centralizer $\hat G_\tx{ad}^s$ is then equal to $\hat N_\tx{ad}^s$ -- the group of $\tx{Ad}(s)$-fixed points in the normalizer of $\hat T_\tx{ad}$. We have the exact sequence
\[ 1 \to \hat T_\tx{ad} \to \hat N_\tx{ad}^s \to (\Z/2\Z)^2 \to 0. \]
Let $\hat T_\tx{sc}$ be the preimage of $\hat T_\tx{ad}$ in $\hat G_\tx{sc}$, a maximal torus, and let $\hat N_\tx{sc}^+$ be the preimage of $\hat N_\tx{ad}^s$. We thus have the extension
\[ 1 \to \hat T_\tx{sc} \to \hat N_\tx{sc}^+ \to (\Z/2\Z)^2 \to 0. \]
Let $f \in \hat N_\tx{ad}$ be s.t. its image in $\hat N_\tx{ad}/\hat T_\tx{ad}$ commutes with every element of $\hat N_\tx{ad}^s/\hat T_\tx{ad} = (\Z/2\Z)^2$ and has no fixed points in $X_*(\hat T_\tx{ad})$. Then $\hat T_\tx{ad}^f$ is finite.  We write $(-)^f$ again for the groups of fixed points of $\tx{Ad}(f)$ in $\hat G_\tx{sc}$ as well as $\hat G_\tx{ad}$.

\begin{lem} \label{lem:d2ng} The natural map $\hat N_\tx{sc}^{+,f} \to (\Z/2\Z)^2$ is surjective. The extension
\[ 1 \to \hat T_\tx{sc}^f \to \hat N_\tx{sc}^{+,f} \to (\Z/2\Z)^2 \to 0 \]
has trivial commutator.
\end{lem}

The second situation is the following. Let $k$ be a finite field of characteristic different from $2$ and let $G_\tx{sc}$ be a simply connected group of type $D_{2n}$ defined over $k$. Let $S_\tx{sc} \subset G_\tx{sc}$ be an anisotropic torus and $\theta : S_\tx{sc}(k) \to \C^\times$ a non-singular character whose stabilizer in $\Omega(S_\tx{sc},G_\tx{sc})(k)$ is isomorphic to $(\Z/2\Z)^2$.

\begin{lem} \label{lem:d2nr} The extension $1 \to S_\tx{sc}(k) \to N(S_\tx{sc},G_\tx{sc})(k)_\theta \to \Omega(S_\tx{sc},G_\tx{sc})(k)_\theta \to 1$ has trivial commutator.
\end{lem}

\begin{proof}[Proof of Lemmas \ref{lem:d2ng} and \ref{lem:d2nr}]
Let $\dot x \in \hat N_\tx{sc}^+$ be a lift of some $x \in (\Z/2\Z)^2$. Then $\dot x^{-1} \cdot \tx{Ad}(f)(\dot x)$ lies in $\hat T_\tx{sc}$. The map $y \mapsto y^{-1} \cdot \tx{Ad}(f)(y)$ is an endomorphism of $\hat T_\tx{sc}$ with finite kernel, hence surjective. This allows us to modify the lift $\dot x$ by some $y \in \hat T_\tx{sc}$ to achieve $\dot x \in \hat N_\tx{sc}^{+,f}$. This proves the surjectivity claim of Lemma \ref{lem:d2ng}. The corresponding implicit surjectivity claim in Lemma \ref{lem:d2nr} is immediate from Lang's theorem.

Consider $R^+=\{e_i-e_j|1 \leq i<j \leq 2n\} \cup \{e_i+e_j|1 \leq i<j \leq 2n\} \subset \Z^{2n}$, this is the standard presentation of the system of positive roots for type $D_{2n}$. Let $Q \subset \Z^{2n} \subset P$ be the root and weight lattices, respectively. Thus $Q$ is the span of $R=R^+ \cup -R^+$, or equivalently the sublattice of $\Z^{2n}$ consisting of vectors whose sum of coordinates is divisible by $2$, while $P=\Z^{2n}+\frac{1}{2}\sum_{i=1}^{2n}e_i$. The standard inner product on $\R^{2n}$ identifies each root $\alpha \in R$ with its coroot $\alpha^\vee \in R^\vee$, and in particular the root lattice $Q$ with the coroot lattice $Q^\vee$ and the weight lattice $P$ with the coweight lattice $P^\vee$.

Let $\hat G=\tx{SO}_{2n}(\C)$, so that we have the isogenies $\hat G_\tx{sc} \to \hat G \to \hat G_\tx{ad}$. We obtain $\hat T_\tx{sc} = Q \otimes \C^\times$, $\hat T=\Z^{2n} \otimes \C^\times$, and $\hat T_\tx{ad} = P \otimes \C^\times$. We use the exponential sequence $0 \to \Z \to \C \to \C^\times \to 1$ and the isomorphisms $Q \otimes \C \to \Z^{2n} \otimes \C \to P \otimes \C$ to identify $\hat T_\tx{sc} = \C^{2n}/Q$ and $\hat T_\tx{ad} = \C^{2n}/P$. Of course, the isomorphism $X_*(\hat T) \cong \Z^{2n}$ used here involves a choice that in particular implies a choice of a positive Weyl chamber. We shall specify this choice further below.

We can do the same over the finite field $k$. For this, we fix arbitrarily an isomorphism of groups $\bar k^\times \to (\Q/\Z)_{p'}$ that will serve as a replacement of the exponential map. The action of Frobenius on $\bar k^\times$ is translated to multiplication by $q$ on $(\Q/\Z)_{p'}$. Define again $G=\tx{SO}_{2n}/k$ as above, so that $G_\tx{sc} \to G \to G_\tx{ad}$ are isogenies. Then we obtain $S_\tx{sc}(\bar k) \cong \Z^{2n} \otimes \bar k^\times \cong (\Q^{2n}/Q)_{p'}$.

Consider the group of signed permutations $\{\pm 1\}^{2n} \rtimes S_{2n}$ acting on $\Z^{2n}$. It preserves the root system $R$ and is the full group of automorphisms of $R$ unless $n=2$. The Weyl group $\Omega$ is the subgroup of index $2$ whose elements change an even number of signs. A signed permutation is elliptic, i.e. has no fixed points in $\Z^{2n}$, if and only if each cycle has an odd number of sign changes.

There is a choice of isomorphism $X_*(\hat T) \cong \Z^{2n}$ so that the quotient $\hat N_\tx{ad}^s/\hat T_\tx{ad} \cong (\Z/2\Z)^2$ is the subgroup of $\Omega=N/T$ generated by the elements $w_1:=\epsilon_1\epsilon_{2n}$ and $w_2:=(-1) \cdot m$, where $\epsilon_i$ sends $e_i$ to $-e_i$ and fixes $e_j$ for $j \neq i$, $(-1)$ is multiplication by $-1$ on $\Z^{2n}$, and $m(e_i)=e_{2n+1-i}$. The image of $f$ in $\Omega$ is an elliptic element $w_0 \in \Omega$ that commutes with both $w_1$ and $w_2$. Such an element can be brought, by conjugation by elements commuting with $w_1$ and $w_2$, into the following form: $w_0=w_0' \cdot mw_0'm^{-1}$, where $w_0'$ is a signed permutation of $\{1,\dots,n\}$, acting on $\Z^{2n}=\Z^n \oplus \Z^n$ by the natural action on the first factor and the identity on the second factor, and given by the product of consecutive increasing negative cycles, the first of them having length $1$. More precisely, there is a sequence of integers $1=i_1<2=i_2<i_3<\dots<i_k<i_{k+1}=n+1$ such that $w_0'(e_{i_{a+1}-1})=-e_{i_a}$ and $w_0'(e_j)=e_{j+1}$ for $j+1 \notin \{i_2,\dots,i_{k+1}\}$. Thus we may adjust our choice of isomorphism $X_*(\hat T) \cong \Z^{2n}$ to ensure that $w_0$ is of this form.

The element $f$ is a lift of $w_0$ to $\hat N_\tx{ad}$. We may further lift to $\hat N_\tx{sc}$. Since $\hat T_\tx{sc}^{w_0}$ is finite, all lifts of $w_0$ are conjugate under $\hat T_\tx{sc}$. Conjugating $f$ by $\hat T_\tx{sc}$ replaces the extension we are considering by an isomorphic extension. We may therefore arrange for $f$ to be any lift of $w_0$ we like. We choose a pinning of $\hat G_\tx{sc}$, involving the maximal torus $\hat T_\tx{sc}$, and we let $f=\dot w_0$ be the Tits lift of $w_0$ relative to that pinning \cite[\S2.1]{LS87}.

We have thus introduced coordinates into the situation of Lemma \ref{lem:d2ng} that will be helpful for our computations. We shall now do the same with the situation of Lemma \ref{lem:d2nr}. For this, we fix a split maximal torus $T \subset G$ and choose $g \in G$ s.t. $gTg^{-1}=S$. Then $w_0=g^{-1}\sigma(g) \in \Omega=N/T$ is an elliptic element. Let $w_1,w_2 \in \Omega$ generate the preimage under $\tx{Ad}(g)$ of the stabilizer of $\theta$. Then $w_0,w_1,w_2$ all commute. As argued above, there is a choice of isomorphism $X_*(T) \cong \Z^{2n}$ so that $w_0,w_1,w_2$ have the coordinate form given above. Fix a pinning of $G$ involving the torus $T$ and let $\dot w_0$ be the Tits lift of $w_0$ relative to that pinning. Then $\dot w_0 \in N(T,G)(k)$ is of finite order, and hence $\sigma \mapsto \dot w_0$ determines a 1-cocycle of $\tx{Gal}(\bar k/k)$ in $N(T,G)(\bar k)$. Both $\sigma \mapsto \dot w_0$ and $\sigma \mapsto g^{-1}\sigma(g)$ map to the same element of $Z^1(\tx{Gal}(\bar k/k),\Omega)$, so their difference is an element of $Z^1(\tx{Gal}(\bar k/k),T_{w_0})$, where $T_{w_0}$ denotes the torus $T$ with Frobenius action twisted by $w_0$. By Lang's theorem this latter element is of the form $t^{-1} \cdot w_0 \sigma(t)w_0^{-1}$. Thus, after replacing $g$ by $gt^{-1}$ we obtain $g^{-1}\sigma(g)=\dot w_0$.

To prove Lemma \ref{lem:d2ng} we will find lifts of $w_1$ and $w_2$ in $\hat N_\tx{sc}^{+,f}$ such that their commutator, which automatically lies in $\hat T_\tx{sc}^f$, vanishes in the group of $\<w_1,w_2\>$-coinvariants. To prove Lemma \ref{lem:d2nr} we will find lifts in $N(S_\tx{sc},G_\tx{sc})(k)$ of two generators of $\Omega(S_\tx{sc},G_\tx{sc})(k)_\theta$ so that their commutator, again automatically belonging to $S_\tx{sc}(k)$, vanishes in the group of coinvariants for the action of $\Omega(S_\tx{sc},G_\tx{sc})(k)_\theta$. The latter is equivalent to finding lifts in $N(T_\tx{sc},G_\tx{sc})(\bar k)$ of $w_1$ and $w_2$ that are fixed by $\tx{Ad}(\dot w_0)\circ \sigma$ and whose commutator, which lies in the $\tx{Ad}(\dot w_0)\circ\sigma$-fixed points of $T_\tx{sc}(\bar k)$, vanishes in the group of $\<w_1,w_2\>$-coinvariants.

Let $\dot w_1,\dot w_2$ be the Tits lifts of $w_1$ and $w_2$ with respect to the chosen pinning. They automatically lie in $\hat N_\tx{sc}^+$ (respectively $N(T_\tx{sc},G_\tx{sc})(k)$), but may not commute with $\dot w_0$.

Using \cite[Lemma 2.1.A]{LS87} we see that for any two commuting $u,v \in \Omega$ the commutator $[\dot u,\dot v] := \dot u \dot v \dot u^{-1} \dot v^{-1}$ is given $\lambda_{u,v}(-1)$, where $\lambda_{u,v}$ is the sum of the coroots for the set of roots
\[ \Lambda_{u,v} := \{\alpha>0, (uv)^{-1}\alpha>0 \} \cap (\{ u^{-1}\alpha<0, v^{-1}\alpha>0\} \cup \{ u^{-1}\alpha>0, v^{-1}\alpha<0\}). \]
The actions of $w_1^{-1}$ and $w_2^{-1}$ on $R^+$ are given by the following tables

 \begin{tabular}{|L||L|L|L|L|}
 \hline
 \alpha&e_1\pm e_{2n}& e_1\pm e_j,j<2n&e_i\pm e_{2n},i>1&e_i\pm e_j,1<i<j<2n\\
 \hline
 w_1^{-1}\alpha&-(e_1 \pm e_{2n}) &-(e_1 \mp e_j)&(e_i \mp e_{2n})&e_i \pm e_j\\
 \hline
 \end{tabular}

 \begin{tabular}{|L||L|L|}
 \hline
 \alpha&e_i-e_j&e_i+e_j\\
 \hline
 w_2^{-1}\alpha&e_{2n+1-j}-e_{2n+1-i}&-(e_{2n+1-j}+e_{2n+1-i})\\
 \hline
 \end{tabular}

For $w_0^{-1}$ it is enough to record which positive roots are sent to negative. For this, let $i_a'=2n+1-i_a$ and
\[ B=\{ i_a|a=1,\dots,k \} \cup \{ i_a'| a =1,\dots,k \}. \]
Then we have the following table

 \begin{tabular}{|L|L|L|L|}
 \hline
 &w_0^{-1}(e_i-e_j)&w_0^{-1}(e_i+e_j)\\
 \hline\hline
 i,j \notin B&+&+\\
 \hline
 i=i_a,j<i_{a+1}&-&+\\
 \hline
 i>i_{a+1}',j=i_a'&+&-\\
 \hline
 i=i_a,j\geq i_{a+1}&-&-\\
 \hline
 i=i_a'&-&-\\
 \hline
 i_a<i < i_{a+1},j \geq i_{a+1}&+&+\\
 \hline
 i_{a+1}' <i < i_a',j \geq i_a'&+&+\\
 \hline
 \end{tabular}

The element $\lambda_{u,v}(-1)$ is a torsion element of $\hat T_\tx{sc}$ or $T_\tx{sc}(\bar k)$ respectively. In order to unify the treatment we define $q=p=1$ in the situation of Lemma \ref{lem:d2ng} and interpret $(\Q^{2n}/Q)_{p'}$ to mean $(\Q^{2n}/Q)$. Then $(\Q^{2n}/Q)_{p'}$ is the subgroup of torsion elements of $\hat T_\tx{sc}$ in the case of Lemma \ref{lem:d2ng} and the full $T_\tx{sc}(\bar k)$ in the case of Lemma \ref{lem:d2nr}. Moreover, $qw_0$ is the action of $f$ in the former case and the action of $\tx{Ad}(\dot w_0)\circ\sigma$ in the latter case. We set $f=\tx{Ad}(\dot w_0)\circ\sigma$ in the latter case.

In both cases the element $\lambda_{u,v}(-1)$ is represented by $\frac{1}{2}\lambda_{u,v}$. For $u=w_1$ and $v=w_0$ we see that $\Lambda_{u,v}=\emptyset$, and thus $\dot w_1$ is fixed by $f$. For $u=w_2$ and $v=w_0$ we have
\begin{eqnarray*}
\Lambda_{u,v}&=&\{e_i\pm e_j|i=i_a,j<i_{a+1}\}\\
&\cup&\{e_i-e_j|i=i_a,j\geq i_{a+1},j \notin B\}\\
&\cup&\{e_i-e_j|i=i_a',j \notin B\}\\
&\cup&\{e_i+e_j|j=i_a,i \notin B\}\\
&\cup&\{e_i+e_j|j=i_a',i\leq i_{a+1}',i \notin B\}.
\end{eqnarray*}
Thus $f\dot w_2f^{-1}=\lambda_{w_2,w_0}(-1)\dot w_2$ and we need to multiply $w_2$ by an element $t \in \hat T_\tx{sc}$ (or $t \in T_\tx{sc}$ respectively) such that $ftf^{-1}=\lambda_{w_2,w_0}(-1)t$. Since $qw_0-1$ is invertible on $\Q^{2n}$ we can form $\mu=\frac{1}{2}(qw_0-1)^{-1}\lambda_{w_2,w_0} \in \Q^{2n}$. All denominators of this vector are powers of the form $q^{l_i}+1$, where $l_i$ are the lengths of the cycles in $w_0$. Since $\Z^{2n}/Q \cong \Z/2\Z$  and $p \neq 2$ we see that the image $t \in (\Q^{2n}/Q)$ of $\mu$ has order prime to $p$. Then $t\dot w_2$ is fixed by $f$.

Now we have the lifts $\dot w_1,t\dot w_2$ of $w_1$ and $w_2$ in $\hat N_\tx{sc}^{+}$ (respectively $N(T_\tx{sc},G_\tx{sc})(\bar k)$) fixed by $f$. We now compute their commutator $[\dot w_1,t\dot w_2]=(t^{-1} \cdot {^{w_1}t}) \cdot [\dot w_1,\dot w_2]$. First consider $t^{-1}\cdot {^{w_1}}t$. It is the image in $(\Q^{2n}/Q)_{p'}$ of $(w_1-1)\mu = \frac{1}{2}(w_1-1)(qw_0-1)^{-1}\lambda_{w_2,w_0} \in \Q^{2n}$. To compute it, we decompose $\Q^{2n}=\Q \oplus \Q^{2n-2} \oplus \Q$. Both $w_1$ and $w_0$ respect this decomposition. We have $w_1=(-1,\tx{id},-1)$ and $w_0=(-1,*,-1)$, hence $(w_1-1)(qw_0-1)^{-1}=(\frac{2}{q+1},0,\frac{2}{q+1})$.

To evaluate $\frac{1}{2}(w_1-1)(qw_0-1)^{-1}\lambda_{w_2,w_0}$ we thus need to only compute the contributions of $e_1$ and $e_{2n}$ to $\lambda_{w_2,w_0}$. Using the tables above we see that it is $b(e_1+e_{2n})$, where $b=2n-|B|$, so $\frac{1}{2}(w_1-1)(w_0-1)^{-1}\lambda_{w_2,w_0}=\frac{b}{q+1}(e_1+e_{2n})$. We note that $2|b$. Then the computation $(qw_0-1)\frac{b}{q+1}e_1=-be_1 \in Q$ shows that the image of $\frac{b}{q+1}e_1$ in $(\Q^{2n}/Q)_{p'}$ is fixed by $f$. Moreover, $(1-w_2)\frac{b}{q+1}e_1=\frac{b}{q+1}(e_1+e_{2n})$. We see that $(t^{-1}\cdot {^{w_1}}t)$ belongs in the $w_2$-coinvariants of $\hat T_\tx{sc}^f$ (or $T_\tx{sc}(\bar k)^{f}$ respectively).

Next consider $[\dot w_1,\dot w_2]$. We have
$\Lambda_{w_1,w_2} = \{e_1-e_j|1<j<2n\} \cup \{e_i+e_{2n}|1<i<2n\}$ and hence $\frac{1}{2}\lambda_{w_1,w_2}=\frac{1}{2}(2n-2)(e_1+e_{2n}) \in Q$, thus $[\dot w_1,\dot w_2]=1$.
\end{proof}

\end{appendices}

\bibliographystyle{amsalpha}
\bibliography{/Users/kaletha/Dropbox/Work/TexMain/bibliography.bib}

\end{document}

%% file: macros.tex
\def\mf#1{\mathfrak{#1}}
\def\mc#1{\mathcal{#1}}
\def\mb#1{\mathbb{#1}}
\def\tx#1{\textrm{#1}}
\def\tb#1{\textbf{#1}}

\def\ms#1{\mathsf{#1}}

\def\R{\mathbb{R}}

\def\C{\mathbb{C}}
\def\Q{\mathbb{Q}}

\def\Z{\mathbb{Z}}
\def\N{\mathbb{N}}

\def\lmod{\setminus}

\def\hat{\widehat}

\def\rw{\rightarrow}

\def\into{\hookrightarrow}

\def\onto{\twoheadrightarrow}

\def\sm{\smallsetminus}

\def\<{\langle}
\def\>{\rangle}

\newenvironment{mytitle}
{\begin{center}\large\sc}
{\end{center}}

%% file: ssp1-arxiv.bbl
\providecommand{\bysame}{\leavevmode\hbox to3em{\hrulefill}\thinspace}
\providecommand{\MR}{\relax\ifhmode\unskip\space\fi MR }
\providecommand{\MRhref}[2]{%
  \href{http://www.ams.org/mathscinet-getitem?mr=#1}{#2}
}
\providecommand{\href}[2]{#2}
\begin{thebibliography}{DLM92}

\bibitem[AMS]{AMS}
Anne~Marie Aubert, Ahmed Moussaoui, and Maarten Solleveld,
  \emph{Generalizations of the {S}pringer correspondence and cuspidal
  {L}anglands parameters}, preprint.

\bibitem[Art89]{ArtIOR1}
James Arthur, \emph{Intertwining operators and residues. {I}. {W}eighted
  characters}, J. Funct. Anal. \textbf{84} (1989), no.~1, 19--84. \MR{999488
  (90j:22018)}

\bibitem[AS09]{AS09}
Jeffrey~D. Adler and Loren Spice, \emph{Supercuspidal characters of reductive
  {$p$}-adic groups}, Amer. J. Math. \textbf{131} (2009), no.~4, 1137--1210.
  \MR{2543925 (2011a:22018)}

\bibitem[BDR17]{BDR17}
C\'edric Bonnaf\'e, Jean-Fran\c{c}ois Dat, and Rapha\"el Rouquier,
  \emph{Derived categories and {D}eligne-{L}usztig varieties {II}}, Ann. of
  Math. (2) \textbf{185} (2017), no.~2, 609--670. \MR{3612005}

\bibitem[BH06]{BH06}
Colin~J. Bushnell and Guy Henniart, \emph{The local {L}anglands conjecture for
  {$\rm GL(2)$}}, Grundlehren der Mathematischen Wissenschaften [Fundamental
  Principles of Mathematical Sciences], vol. 335, Springer-Verlag, Berlin,
  2006. \MR{2234120 (2007m:22013)}

\bibitem[BT72]{BT1}
F.~Bruhat and J.~Tits, \emph{Groupes r\'eductifs sur un corps local}, Inst.
  Hautes \'Etudes Sci. Publ. Math. (1972), no.~41, 5--251. \MR{0327923 (48
  \#6265)}

\bibitem[BT84]{BT2}
\bysame, \emph{Groupes r\'eductifs sur un corps local. {II}. {S}ch\'emas en
  groupes. {E}xistence d'une donn\'ee radicielle valu\'ee}, Inst. Hautes
  \'Etudes Sci. Publ. Math. (1984), no.~60, 197--376. \MR{756316 (86c:20042)}

\bibitem[CGP15]{CGP15}
Brian Conrad, Ofer Gabber, and Gopal Prasad, \emph{Pseudo-reductive groups},
  second ed., New Mathematical Monographs, vol.~26, Cambridge University Press,
  Cambridge, 2015. \MR{3362817}

\bibitem[DL76]{DL76}
P.~Deligne and G.~Lusztig, \emph{Representations of reductive groups over
  finite fields}, Ann. of Math. (2) \textbf{103} (1976), no.~1, 103--161.
  \MR{0393266 (52 \#14076)}

\bibitem[DLM92]{DLM92}
F.~Digne, G.~I. Lehrer, and J.~Michel, \emph{The characters of the group of
  rational points of a reductive group with nonconnected centre}, J. Reine
  Angew. Math. \textbf{425} (1992), 155--192. \MR{1151318}

\bibitem[DR09]{DR09}
Stephen DeBacker and Mark Reeder, \emph{Depth-zero supercuspidal {$L$}-packets
  and their stability}, Ann. of Math. (2) \textbf{169} (2009), no.~3, 795--901.
  \MR{2480618 (2010d:22023)}

\bibitem[DS18]{DS18}
Stephen DeBacker and Loren Spice, \emph{Stability of character sums for
  positive-depth, supercuspidal representations}, J. Reine Angew. Math.
  \textbf{742} (2018), 47--78. \MR{3849622}

\bibitem[Fin19]{Fin19}
Jessica Fintzen, \emph{On the construction of tame supercuspidal
  representations.}, arXiv:1908.09819 (2019).

\bibitem[FKS19]{FKS}
Jessica Fintzen, Tasho Kaletha, and Loren Spice, \emph{A twisted yu
  construction, harish-chandra characters, and endoscopy}, arXiv:1912.03286
  (2019).

\bibitem[He08]{He08}
Xuhua He, \emph{On the affineness of {D}eligne-{L}usztig varieties}, J. Algebra
  \textbf{320} (2008), no.~3, 1207--1219. \MR{2427638 (2009c:20085)}

\bibitem[HM08]{HM08}
Jeffrey Hakim and Fiona Murnaghan, \emph{Distinguished tame supercuspidal
  representations}, Int. Math. Res. Pap. IMRP (2008), no.~2, Art. ID rpn005,
  166. \MR{2431732 (2010a:22022)}

\bibitem[Kal]{KalLLCD}
Tasho Kaletha, \emph{On the local {L}anglands conjectures for disconnected
  groups}, in preparation.

\bibitem[Kal11]{KalECI}
\bysame, \emph{Endoscopic character identities for depth-zero supercuspidal
  {$L$}-packets}, Duke Math. J. \textbf{158} (2011), no.~2, 161--224.
  \MR{2805068 (2012f:22031)}

\bibitem[Kal16a]{KalSimons}
\bysame, \emph{The local {L}anglands conjectures for non-quasi-split groups},
  Families of automorphic forms and the trace formula, Simons Symp., Springer,
  [Cham], 2016, pp.~217--257. \MR{3675168}

\bibitem[Kal16b]{KalRI}
\bysame, \emph{Rigid inner forms of real and {$p$}-adic groups}, Ann. of Math.
  (2) \textbf{184} (2016), no.~2, 559--632. \MR{3548533}

\bibitem[Kal18a]{KalGRI}
\bysame, \emph{Global rigid inner forms and multiplicities of discrete
  automorphic representations}, Invent. Math. \textbf{213} (2018), no.~1,
  271--369. \MR{3815567}

\bibitem[Kal18b]{KalRIBG}
\bysame, \emph{Rigid inner forms vs isocrystals}, J. Eur. Math. Soc. (JEMS)
  \textbf{20} (2018), no.~1, 61--101. \MR{3743236}

\bibitem[Kal19a]{KalDC}
\bysame, \emph{On {$L$}-embeddings and double covers of tori over local
  fields}, arXiv:1907.05173 (2019).

\bibitem[Kal19b]{KalRSP}
\bysame, \emph{Regular supercuspidal representations}, J. Amer. Math. Soc.
  \textbf{32} (2019), no.~4, 1071--1170. \MR{4013740}

\bibitem[KS99]{KS99}
Robert~E. Kottwitz and Diana Shelstad, \emph{Foundations of twisted endoscopy},
  Ast\'erisque (1999), no.~255, vi+190. \MR{1687096 (2000k:22024)}

\bibitem[Kut77]{Ku77}
P.~C. Kutzko, \emph{Mackey's theorem for nonunitary representations}, Proc.
  Amer. Math. Soc. \textbf{64} (1977), no.~1, 173--175. \MR{0442145 (56 \#533)}

\bibitem[LL79]{LL79}
J.-P. Labesse and R.~P. Langlands, \emph{{$L$}-indistinguishability for {${\rm
  SL}(2)$}}, Canad. J. Math. \textbf{31} (1979), no.~4, 726--785. \MR{540902
  (81b:22017)}

\bibitem[LS87]{LS87}
R.~P. Langlands and D.~Shelstad, \emph{On the definition of transfer factors},
  Math. Ann. \textbf{278} (1987), no.~1-4, 219--271. \MR{909227 (89c:11172)}

\bibitem[Lus88]{Lus88}
G.~Lusztig, \emph{On the representations of reductive groups with disconnected
  centre}, Ast\'erisque (1988), no.~168, 10, 157--166, Orbites unipotentes et
  repr\'esentations, I. \MR{1021495}

\bibitem[MP96]{MP96}
Allen Moy and Gopal Prasad, \emph{Jacquet functors and unrefined minimal
  {$K$}-types}, Comment. Math. Helv. \textbf{71} (1996), no.~1, 98--121.
  \MR{1371680 (97c:22021)}

\bibitem[Pra01]{Pr01}
Gopal Prasad, \emph{Galois-fixed points in the {B}ruhat-{T}its building of a
  reductive group}, Bull. Soc. Math. France \textbf{129} (2001), no.~2,
  169--174. \MR{1871292 (2002j:20088)}

\bibitem[Ree10]{Ree10}
Mark Reeder, \emph{Torsion automorphisms of simple {L}ie algebras}, Enseign.
  Math. (2) \textbf{56} (2010), no.~1-2, 3--47. \MR{2674853 (2012b:17040)}

\bibitem[Ser94]{SerTP}
Jean-Pierre Serre, \emph{Sur la semi-simplicit\'{e} des produits tensoriels de
  repr\'{e}sentations de groupes}, Invent. Math. \textbf{116} (1994), no.~1-3,
  513--530. \MR{1253203}

\bibitem[Spi17]{Spice17}
Loren Spice, \emph{Explicit asymptotic expansions for tame, supercuspidal
  characters}, preprint, arXiv:1701.02417, 2017.

\bibitem[SS70]{SS70}
T.~A. Springer and R.~Steinberg, \emph{Conjugacy classes}, Seminar on
  {A}lgebraic {G}roups and {R}elated {F}inite {G}roups ({T}he {I}nstitute for
  {A}dvanced {S}tudy, {P}rinceton, {N}.{J}., 1968/69), Lecture Notes in
  Mathematics, Vol. 131, Springer, Berlin, 1970, pp.~167--266. \MR{0268192 (42
  \#3091)}

\bibitem[Tap77]{Tapp77}
J\"urgen Tappe, \emph{Irreducible projective representations of finite groups},
  Manuscripta Math. \textbf{22} (1977), no.~1, 33--45. \MR{0466292}

\bibitem[Yu01]{Yu01}
Jiu-Kang Yu, \emph{Construction of tame supercuspidal representations}, J.
  Amer. Math. Soc. \textbf{14} (2001), no.~3, 579--622 (electronic).
  \MR{1824988 (2002f:22033)}

\bibitem[Yu09]{Yu09}
\bysame, \emph{On the local {L}anglands correspondence for tori}, Ottawa
  lectures on admissible representations of reductive {$p$}-adic groups, Fields
  Inst. Monogr., vol.~26, Amer. Math. Soc., Providence, RI, 2009, pp.~177--183.
  \MR{2508725 (2009m:11201)}

\end{thebibliography}
